\documentclass[a4paper,11pt,reqno]{amsart}
\usepackage{graphicx} 
\graphicspath{{figures/}}
\usepackage[colorinlistoftodos]{todonotes} 
\usepackage{subcaption}
\usepackage[colorlinks=true,linkcolor=black,urlcolor=black,citecolor=black]{hyperref}
\usepackage{amsfonts}
\usepackage{amssymb,amsmath}
\usepackage{amsthm}
\usepackage{fixmath}
\usepackage{epstopdf}
\usepackage{algorithm}
\usepackage{algpseudocode}
\usepackage{enumerate}
\usepackage{multirow}

\usepackage{mathtools}
\usepackage{xcolor}
\usepackage{fullpage}
\usepackage{subcaption}
\numberwithin{equation}{section}
\numberwithin{figure}{section}
\numberwithin{equation}{section}

\usepackage{amssymb}

\usepackage{amsmath}

\usepackage{comment}
\usepackage{algorithm}
\usepackage{algpseudocode}
\usepackage{graphicx}
\graphicspath{{figures/}}

\newcommand{\x}{x}
\newcommand{\xx}{y}
\renewcommand{\d}{\mathop{}\!\mathrm{d}}
\newcommand{\bbR}{\mathbb{R}}

\newcommand{\bbZ}{\mathbb{Z}}

\newcommand{\EX}{\nabla\EXd}
\newcommand{\IM}{\nabla \IMd}
\newcommand{\IMd}{{g}}
\newcommand{\EXd}{{f}}

\newcommand{\cker}{c_{\gamma}}
\newcommand{\cpot}{c_F}
\newcommand{\kernel}{\gamma}
\newcommand{\kerneld}{\gamma_\delta}
\newcommand{\J}{\mathcal{J}}
\newcommand{\E}{\mathcal{E}}
\renewcommand{\tt}{\Delta t}
\newcommand{\D}{\Omega}

\renewcommand{\t}{\theta}
\newcommand{\p}{u}

\newcommand{\proxl}{{\rm prox}_{\frac{1}{\lambda}\psi}}
\newcommand{\proxlN}{{\rm prox}_{\frac{1}{\lambda_N}\psi}}
\newcommand{\Jim}{ \J^{2,\rm im}_k}
\newcommand{\Jex}{ \J^{2,\rm ex}_k}
\newcommand{\Sh}{S_h}
\newcommand{\Shw}{\hat{S}_h}

\newcommand{\thetalog}{\theta_c}
\newcommand*{\norm}[1]{{\left\lVert#1\right\rVert}}
\newcommand*{\normd}[1]{{\left\lVert#1\right\rVert}_2}
\newcommand*{\inner}[1]{{\left(#1\right)}}
\newcommand*{\innerd}[1]{{\left\langle#1\right\rangle}}

\newcommand{\calL}{\mathcal{L}}

\DeclareMathOperator{\argmin}{arg\min}

\newtheorem{theorem}{Theorem}[section]
\newtheorem{lemma}[theorem]{Lemma}
\newtheorem{assumption}[theorem]{Assumption}
\newtheorem{proposition}[theorem]{Proposition}

\newtheorem{remark}{Remark}[section]

\newtheorem{example}{Example}
\DeclareMathOperator{\ii}{i}

\numberwithin{equation}{section}

\newcommand{\astcirc}{%
  \mathrel{\ooalign{\hfil$\ast$\hfil\cr\hfil\textcircled{\hfil}\hfil}}%
}

\begin{document}

\title[Proximal-based approach for nonlocal Allen-Cahn equations]{An efficient proximal-based approach for solving nonlocal Allen-Cahn equations}

\author[O. Burkovska \and I. Mustapha]{Olena Burkovska${}^{1}$ \and Ilyas Mustapha${}^{2}$}
\address{${}^1$ Computer Science and Mathematics Division, Oak Ridge National Laboratory, TN, USA}
\email{burkovskao@ornl.gov}
\address{${}^2$ Washington University in St. Louis, MO, USA}
\email{ilyasm@wustl.edu}
\thanks{
This manuscript has been authored in part by UT-Battelle, LLC, under contract DE-AC05-00OR22725 with the US Department of Energy (DOE). 
The US government retains and the publisher, by accepting the article for publication, acknowledges that the US government retains a nonexclusive, paid-up, irrevocable, worldwide license to publish or reproduce the published form of this manuscript, or allow others to do so, for US government purposes. DOE will provide public access to these results of federally sponsored research in accordance with the DOE Public Access Plan (https://www.energy.gov/doe-public-access-plan).
}


\begin{abstract}
In this work, we present an efficient approach for the spatial and temporal discretization of the nonlocal Allen-Cahn equation, which incorporates various double-well potentials and an integrable kernel, with a particular focus on a non-smooth obstacle potential.
While nonlocal models offer enhanced flexibility for complex phenomena, they often lead to increased computational costs and there is a need to design efficient spatial and temporal discretization schemes, especially in the non-smooth setting.
To address this, we propose first- and second-order energy-stable time-stepping schemes combined with the Fourier collocation approach for spatial discretization. 
 We provide energy stability estimates for the developed time-stepping schemes. 
A key aspect to our approach involves a representation of a solution via proximal operators. This together with the spatial and temporal discretizations enables direct evaluation of the solution that can bypass the solution of nonlinear, non-smooth, and nonlocal system.
This method significantly improves computational efficiency, especially in the case of non-smooth obstacle potentials, and facilitates rapid solution evaluations in both two and three dimensions. 
We provide several numerical experiments for isothermal and non-isothermal nonlocal Allen-Cahn systems to illustrate the effectiveness of our approach.
\end{abstract}

\maketitle

\section{Introduction}\label{sec:introduction}

Phase-field models, such as Allen-Cahn and Cahn-Hilliard, are widely used to describe phase-transition phenomena in various applications, including solidification in materials science~\cite{kobayashi1993,wheeler1993}, image processing~\cite{benes2004}, and mean-curvature flow~\cite{DuFeng2020}. Typically, local phase-field models, based on partial derivatives, are considered. However, incorporating nonlocality in phase-field modeling enhances versatility for capturing anomalous processes and is better suited to describe 
sharp interfaces. Unlike local models, nonlocal models rely on integral operators with a nonlocal kernel that accounts for long-range interactions. While nonlocal settings provide greater modeling flexibility, they often require more careful numerical treatment to develop efficient methods for discretizing the integral operator.

In this work, we propose an efficient discretization of the nonlocal Allen-Cahn equation with various potential functions, including \emph{regular}, \emph{singular} (\emph{Flory-Huggins} or \emph{logarithmic}), and \emph{obstacle} potentials. This encompasses a broad range of models, from nonlinear to nonlinear non-smooth types. Our method is based on characterizing the solution using the proximal operator, combined with efficient evaluation of the convolution through fast Fourier transform (FFT).
In contrast to existing methods, our approach can handle a wide variety of non-smooth or singular potentials without the need for additional approximations, such as regularizing the logarithmic potential~\cite{GLWW2014} or using Moreau-Yosida regularization for the obstacle potential~\cite{gwiazda2024rate,hintermuller2011CH}.

More specifically, we consider the model of the following form:
\begin{equation}
     \partial_t u + \calL u + \partial F(u) \ni  0,\quad\text{ in } (0,T)\times\D,\label{eq:general}
\end{equation}
where $\D\subset\mathbb{R}^n$, $n\leq 3$, is a bounded domain, $T>0$ is a fixed time, and the model is supplemented with a suitable initial condition and periodic boundary conditions. 
The operator $\calL\colon L^2(\D)\to L^2(\D)$ is a nonlocal diffusion operator defined as
\begin{eqnarray}
    \mathcal{L}u(x) = \int_\Omega (u(\x)-u(\xx)) \kernel(\x-\xx)\d\xx=\cker u(\x)-(\kernel*u)(\x),\quad\cker>0,\label{nonlocal_op}
\end{eqnarray}
where $\cker =(\kernel*1)=\int_{\D}\kernel(\x-\xx)\d\xx$ and $\kernel\colon\D\to\mathbb{R}$ is an integrable convolution kernel that sets up the law and extend of nonlocal interactions. The nonlocal energy associated with~\eqref{eq:general} has the following form:
\begin{equation}
    \E(u):=
   \frac{1}{4}\int_{\D}\int_{\D}(u(\x)-u(\xx))^2\kernel(\x-\xx)\d\x\d\xx+\int_{\D}F(u)\d\x,
    \label{eq:GL_nonlocal}
\end{equation}
which is an analogue of the local Ginzburg-Landau energy.

Here, $F$ is a double-well potential that attains two local minima when the phase-field variable $u$ is at pure phases $\pm 1$ (e.g., solid and liquid) and is the driving force of the phase-separation encouraging the solution to attain pure phases. 
More specifically, we set 
\begin{equation}
    F(u)=F_0(u)+\psi(u),\label{eq:potential_general}
\end{equation}
where $F_0\colon\bbR\to\bbR $ is concave differentiable part of $F$ and $\psi\colon\bbR\to\bbR\cup\{\infty\}$ is a  (potentially non-differentiable) convex lower semi-continuous function. 
Three common types of double-well potentials: regular, logarithmic (Flory-Huggins) and obstacle, can be expressed in the form of~\eqref{eq:potential_general} with
\begin{equation}
    F_0(u)=\frac{\cpot}{2}(1-u^2),\quad \cpot>0,\label{eq:potential_f0}
\end{equation}
and $\psi$ to be defined for the obstacle potential by a convex indicator function:
\begin{equation}
     \psi_{\rm obs}(u)=\mathbb{I}_{[-1,1]}(u)= 
    \begin{cases}
        0,~\text{ if } u\in[-1,1]\\
        +\infty, \text{ otherwise}
    \end{cases},\label{eq:potential_psi_obstacle}
\end{equation}
for the logarithmic (Flory-Huggins) potential\footnote{Here, the values at $u=\pm 1$ are defined by the limit of the function that leads to ${\thetalog}\ln(2)$. } by
\begin{equation}
    \psi_{\rm log}(u) =\ 
    \begin{cases}
        \frac{\thetalog}{2}\left((1+u)\ln(1+u)+(1-u)\ln(1-u)\right),\quad\text{if}\; u\in[-1,1],\\
        +\infty,\quad \text{otherwise},
    \end{cases}\label{eq:pot_logarithm}
    \end{equation}
with $0<\thetalog<\cpot$ and for the regular potential by
\begin{eqnarray}
    \psi_{\rm reg}(u)=\frac{\cpot}{4}\left(u^4-1\right),\quad \cpot>0.\label{regular_potential}
\end{eqnarray}
The corresponding generalized subdifferential~\eqref{eq:subdiff_potential} is then defined as follows
\begin{equation}
    \partial F(u)=F_0^\prime(u)+\partial\psi(u)=-\cpot u + \partial\psi(u),\label{eq:subdiff_potential}
\end{equation}
with the subdifferential $\partial\psi(u)$ being
    \begin{subequations}
        \begin{align}
        \partial\psi_{\rm obs}=&\ \partial \mathbb{I}_{[-1, 1]}(u):=
\begin{cases} (-\infty, 0], & u=-1,\\ 
0, & u \in (-1, 1),\\
[0, \infty), & u=1,
\end{cases}\label{eq:subdiff_obs}\\
\partial\psi_{\rm log} =&
\begin{cases} 
\emptyset, & u\not\in[-1,1],\\
\left\{\frac{\thetalog}{2}\left(\ln(1+u)-\ln(1-u)\right)\right\}, & u\in (-1,1),
\end{cases} \label{eq:subdiff_log}\\
\partial\psi_{\rm reg} = &\ \{ \cpot u^3\}.\label{eq:subdiff_reg}
    \end{align}\label{eq:subdiff_all}
 \end{subequations}
Using the definitions~\eqref{nonlocal_op} and~\eqref{eq:subdiff_potential}, the equation~\eqref{eq:general} transforms to:
\begin{eqnarray}
    \partial_t u +\xi u -\kernel*u +\partial\psi(u)\ni 0,\label{eq:general_xi}
\end{eqnarray}
where $\xi=\cker-\cpot\geq 0$ is a nonlocal interface parameter that defines the width of the transition region between pure phases: larger values of $\xi$ lead to more diffuse interfaces, while smaller values result in sharper transitions. 

A notable feature of nonlocal models is that they are more naturally suited to describe sharp-interface dynamics, see, e.g.,~\cite{BatesChmaj1999,bates1997,burkovska2023,burkovska2021nonlocal,Chen1997,du2019,DuYang2016,DuYang2017,fife1997,LWY2017} and the references cited therein. In particular, the nonlocal Allen-Cahn and Cahn-Hilliard models allow for well-defined sharp-interfaces in the solution at steady state and transient times, respectively. Unlike the corresponding local models, where sharp interfaces in the solution are achieved only in the asymptotic case of vanishing local interface parameter, nonlocal framework provides more natural and well-defined description to model discontinuous profiles in the solution. The choice of a double-well potential function greatly influences the interface sharpness. In particular, for $\xi=0$ an obstacle potential can lead to solutions with perfectly sharp interfaces, with jump discontinuities in the solution, representing pure phases with no intermediate transitions~\cite{burkovska2023,burkovska2021nonlocal}, while a regular potential for $\xi<0$ results in interfaces that, although discontinuous, may exhibit a narrow transition region between phases~\cite{DuYang2016}. 
To illustrate this, we plot on Figure~\ref{potentials} various potentials and the corresponding solutions of~\eqref{eq:general} close to a steady-state. We observe that the obstacle potential yields the sharpest interface, while the regular potential results in a more diffuse one. For the Flory-Huggins potential, the sharpness of the interface is influenced by the parameter \(\thetalog\); specifically, smaller values of \(\thetalog\) produce a sharper interface, as a decrease in \(\thetalog\) causes the logarithmic potential to approach the non-smooth obstacle potential.

The sharp interface property is inherent to the structure of the model~\eqref{eq:general}, which we are also going to explore to design efficient temporal and spatial discretizations and solution strategies. Specifically, exploiting~\eqref{eq:general} and the potential~\eqref{eq:potential_general}, we can derive a representation formula for the solution using proximal operators, as detailed below. This representation enables simplified and efficient computation of the solution, in some cases completely eliminating the need to solve nonlinear and non-smooth systems.

\subsection{Representation formula}\label{sec:introduction1}
\begin{figure}[t!]
    \centering
    \begin{subfigure}[h]{0.22\textwidth}
        \centering        \includegraphics[width=\textwidth]{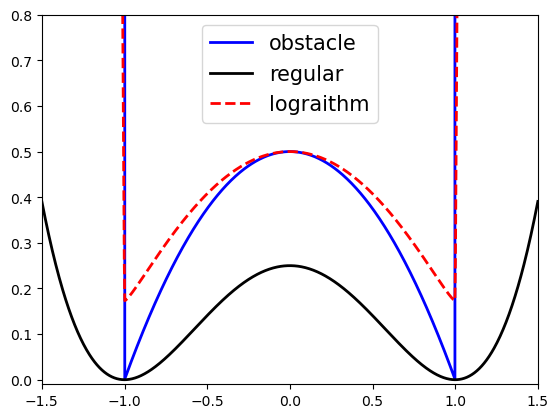}
        \caption*{potential}
    \end{subfigure}
    \begin{subfigure}[h]{0.22\textwidth}
        \centering        \includegraphics[width=\textwidth]{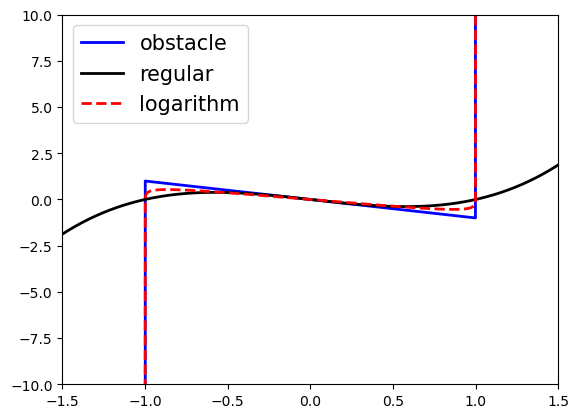}
        \caption*{subdifferential}
    \end{subfigure}
    \begin{subfigure}[h]{0.22\textwidth}
        \centering        \includegraphics[width=\textwidth]{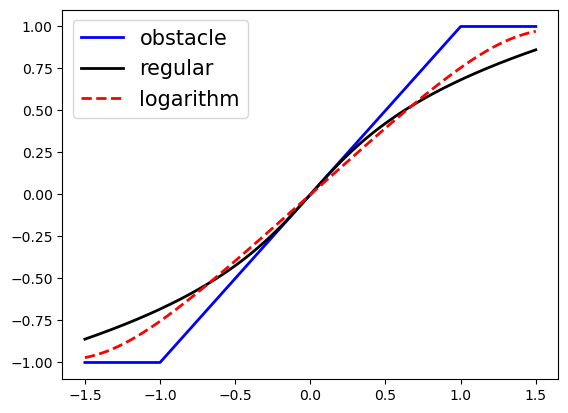}
        \caption*{proximal operator}
    \end{subfigure}
     \begin{subfigure}[h]{0.22\textwidth}
        \centering        \includegraphics[width=\textwidth]{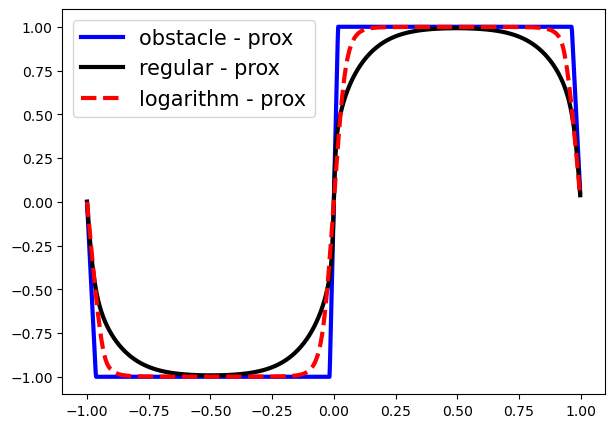}
        \caption*{solution}
    \end{subfigure}
    \caption{From left to right: Illustration of different potentials $F$, corresponding sub-differentials $\partial F$, and proximal operators (for $\lambda = 1$), and the solution of~\eqref{eq:general_xi} at $t=10$ ($\cpot=1$, $\xi=0.001$, and $\theta=0.25$).}
    \label{potentials}
\end{figure}
We discuss briefly, how one can characterize the time-discrete solution of~\eqref{eq:general_xi} in terms of the proximal operators and derivation of the representation formula. 
Employing a semi-implicit Euler time-stepping scheme (more details on this and higher order schemes are discussed in Section~\ref{sec:temporal_disc}) with an explicit discretization of the convolution to~\eqref{eq:general_xi} leads to:
\begin{equation}
    \frac{1}{\tt}\left(u^{k}-u^{k-1}\right)+\xi u^k-\kernel*u^{k-1}+\partial\psi(u^k)\ni 0,\label{eq:first_order_tk_intro}
\end{equation}
with $k=1,\dots,K$, $\tt=T/K$, $K\in\mathbb{N}$, which is also equivalent to
\begin{equation}
    u^k = \proxl\left(\frac{1}{\lambda}\left(\kernel*u^{k-1}+\frac{1}{\tt}u^{k-1}\right)\right),\label{eq:prox_u}
\end{equation}
with $\lambda=\xi+1/\tt>0$. Here, $\proxl\colon\bbR\to\bbR$ is a proximal operator of $\frac{1}{\lambda}\psi$~\cite{bauschke2017book,boyd2013proximal}:  defined as 
\begin{equation*}
    \proxl(v) = \argmin_{s} \left(\frac{1}{\lambda}\psi(s)+\frac{1}{2}|s-v|^2\right),
\end{equation*}
which is also related to the subdifferential $\partial\psi$ as
\begin{equation}
    \proxl=\left(I+\frac{1}{\lambda}\partial\psi\right)^{-1},\quad \lambda = 1/\tt +\xi>0.\label{eq:prox}
\end{equation} 
Here, the relation~\eqref{eq:prox_u} should be understood pointwise, in a sense that $u^k(\x)=\proxl(\zeta^k(x)/\lambda)$ with $\zeta^k(x)=(\kernel*u^{k-1})(\x)+\frac{1}{\tt}u^{k-1}(\x)$ for $\x\in\Omega$.

The representaton~\eqref{eq:prox_u} provides a useful approach from a computational standpoint. Knowing the proximal operator one can directly evaluate the solution at each time step using its values from the previous time step. In particular, for the case of an obstacle potential ($\psi(u)=\mathbb{I}_{[-1,1]}(u)$), the proximal operator reduces to the pointwise projection operator and~\eqref{eq:prox_u} becomes
\begin{equation}
    u^k = P_{[-1,1]}\left(\frac{1}{\lambda}\left(\kernel*u^{k-1}+\frac{1}{\tt}u^{k-1}\right)\right).\label{eq:proj_u}
\end{equation}
For the case of a regular potential, the proximal operator is computed by using Cardano's formula \cite{witula2010cardano} and  \eqref{eq:prox_u}  becomes
\begin{equation}
    u^k = \proxl\left(\frac{1}{\lambda}\zeta_k\right)=\sqrt[3]{\tilde{\zeta}_k+\sqrt{\tilde{\zeta}_k^2 + \left({\lambda}/{3c_F}\right)^3}} + \sqrt[3]{\tilde{\zeta}_k-\sqrt{\tilde{\zeta}_k^2 + \left({\lambda}/{3c_F}\right)^3}},\label{eq:prox_formula_reg_order1}
\end{equation}
where as before $ \zeta_k = \left(\kernel*u^{k-1}+\frac{1}{\tt}u^{k-1}\right)$ and $\tilde{\zeta}_k=\zeta_k{\lambda}/(2\cpot)$. 
For the Flory-Huggins potential, to the best of our knowledge, there is no known closed-form expression for the proximal operator. Instead, numerical approaches such as Newton's or bisection methods can be used to approximate the solution. This can be performed very efficiently, as it is a one-dimensional problem at each point in space and can be easily parallelized.
Overall, since representation \eqref{eq:proj_u} only involves explicit terms, it allows for direct computation at each time step, eliminating the need for nonlinear or non-smooth solvers. Moreover, the proximal operator representation enhances stability properties of the solution for the cases of regular and logarithmic potentials.

\subsection{Contribution and related works}

There have been numerous works on mathematical and numerical aspects of nonlocal phase-field models, see, e.g., survey papers~\cite{bates2006survey,DuFeng2020,miranville2017review}. Most commonly, models with smooth potentials of regular or logarithmic type have been studied. Phase-field models with the non-smooth obstacle potential in the local settings have been analyzed in~\cite{blank2012,blank2011,bosch2014,hintermuller2013}. Discretization of nonlocal obstacle models have been studied in, e.g.,~\cite{borthagarayObstacle2019,BurkovskaGunzburger2019,burkovska2021nonlocal,GuanObstacle2017}.
In this work, we present a novel approach for a spatial and temporal discretization of the nonlocal Allen-Cahn model based on discretization of an equivalent solution representation via proximal operators in \eqref{eq:prox_u}. This representation for the obstacle potential, which leads to a projection formula, has been derived in our previous works~\cite{BurkovskaGunzburger2019, burkovska2023}. Here, we generalize it to regular and logarithmic potentials and use it to design efficient temporal and spatial discretizations, including energy-stable first and second-order time-stepping schemes.

We comment on related works that investigate spatial and temporal discretizations for nonlocal phase-field systems. Design and study of second-order convex splitting time-stepping schemes for the nonlocal Allen-Cahn and Cahn-Hilliard equations with the smooth potential have been conducted in~\cite{GLWW2014,guan2017}. There, the authors discretize convolution term explicitly, while the nonlinearity is treated implicitly. Explicit discretization of both convolution and nonlinearity has been done in~\cite{bates2009AC}, where $L^\infty$ stability and convergence are provided. Exponential time differencing method (ETD) was adopted for the nonlocal Allen-Cahn model with the regular potential in~\cite{du2019}. 
In \cite{DuJuLiQiao2018}, the authors developed stabilized linear schemes with implicit discretization of the nonlocal term for the Cahn-Hilliard model with regular potential. This approach allows a solution of the discrete system in frequency space using a Fourier collocation method with fast Fourier transform (FFT). We adopt a similar strategy for fast convolution evaluation, but focus on a solution method in physical space instead of frequency space.

For the spatial discretization of nonlocal models, various methods have been developed in recent years, based on finite elements~\cite{AinsworthGlusa2017,TianDuAC2014}, finite-differences~\cite{bates2009AC,guan2017,GLWW2014}, spectral approximations~\cite{alali2020fourier,alali2021fourier,jafarzadeh2020efficient,mustapha2023regularity,mustapha2025fourier} or model-reduction approaches~\cite{BurkovskaGunzburger2019,Guan2017RBM}; see also survey papers~\cite{DElia2020Acta,TianDuSiamRev2020}. 
Finite element or finite difference approaches can provide more flexibility in terms of boundary conditions and geometry of the domain, but may lack efficiency compared to spectral approximations. When dealing with nonlocal phase-field system, an additional nonlinearity and non-smoothness inherent to the model needs to be taken into account for the design of efficient discretization strategies. 
For smooth potential functions, spatial discretizations of nonlocal phase-field model based on finite-differences~\cite{bates2009AC,du2019,guan2017,GLWW2014}
 or spectral approximations~\cite{DuJuLiQiao2018,DuYang2016,hartley2009} have been actively studied in the literature.
 Recent deep learning algorithms based on neural networks have also been successfully employed to predict the dynamics of both local~\cite{GTWJ2024} and nonlocal~\cite{GengBur2024} Allen-Cahn and Cahn-Hilliard equations.
 For the case of the obstacle potential finite element discretization has been provided in our previous works~\cite{burkovska2023,BurkovskaGunzburger2019}. 
Adapting spectral approximations to such systems is not straightforward due to the underlying nonlinear and non-smooth nature. Common solution algorithms for these systems, whether in local or nonlocal settings, typically involve a primal-dual active set algorithm~\cite{kunisch2002} based on a semi-smooth Newton method. This approach requires an iterative method at each time step to identify an admissible set, which can be computationally expensive in a nonlocal context. 
Instead, we exploit the structure of the problem and representation~\eqref{eq:proj_u} and adopt Fourier spectral approximation for the convolution that can be efficiently performed with FFT with $\mathcal{O}(N\log N)$ complexity, where $N$ is the number of collocation points.
The corresponding solution algorithm simplifies to evaluating  the pointwise projection and allows generalization to regular and logarithmic potentials through the use of proximal operators.
We also comment that the proximal operator has been employed in the context of the Allen-Cahn equation with regular potential on graphs in~\cite{luo2017convergence}.

The rest of the paper is organized as follows. In Section~\ref{sec:temporal_disc}, we present the construction and analysis of first and second-order time-stepping schemes. We prove the energy stability properties of the time-stepping schemes, as well as the convergence via fixed-point iteration needed for the efficient realization of the second-order scheme. In Section~\ref{sec:spatial_disc}, details are provided on spatial discretization using Fourier collocation approximation based on FFT, including the discretization of the representation formula. Section~\ref{sec:numerics} presents several numerical results in one, two, and three dimensions that illustrate the method's efficiency and accuracy. Furthermore, in Section~\ref{sec:NAC}, we demonstrate the applicability of the method to a non-isothermal Allen-Cahn system with several test cases.
Finally, in Section~\ref{sec:conclusion}, we summarize the findings and discuss potential avenues for future research.

\section{Temporal discretization}\label{sec:temporal_disc}

In what follows we assume that the nonlocal kernel $\kernel$ is integrable, positive, radial $\kernel(\x)=\hat{\kernel}(|\x|)\in L^1(\Omega)$, $\D$-periodic, and has a finite second moment $\int_{\D}\kernel(\x)|\x|^2\d\x=2n$, $\Omega\subset\bbR^n$, $n\leq 3$. 
We define the bilinear and quadratic form
\begin{equation}
  \inner{u,v}_{\kernel}=(\kernel*u,v),\quad \norm{u}_\kernel^2 = (\kernel*u,u),
\end{equation}
where $(\cdot,\cdot)$ stands for the usual $L^2$-inner product.
We also recall the following property for a differentiable convex functional $q\colon L^2(\D)\to\bbR$:
\begin{equation}
    q(u)\geq q(v)+(\nabla q(v),u-v),\quad\forall u,v\in L^2(\D). \label{eq:convexity}
\end{equation}
Then, we introduce the following convex and concave functions
\begin{align}
g(u):=\frac{\xi}{2}\norm{u}^2+\frac{\cpot}{2}|\D|,\quad f(u):=-\frac12(\kernel*u,u)=-\frac12\norm{u}^2_{\gamma},  \label{eq:tk_terms_en}
\end{align}
where $\xi=\cker-\cpot\geq 0$, and the corresponding derivatives:
\begin{equation}
 \nabla \IMd(u)=\xi u\quad\text{and}\quad\nabla\EXd(u)=-\gamma*u,\label{eq:tk_terms}
\end{equation}
where the gradient is understood as a Hilbert space gradient with respect to $L^2(\Omega)$ inner product.
In terms of the above notation we can rewrite~\eqref{eq:general_xi} as follows
\begin{equation}
  \partial_t u + \IM(u)+ \EX(u)+\partial\psi(u)\ni 0,  \label{eq:general_form}
\end{equation}
together with the corresponding nonlocal Ginzburg-Landau energy~\eqref{eq:GL_nonlocal}:
\begin{align}
    \E(u)
    =\frac{\xi}{2}\norm{u}^2-\frac{1}{2}\norm{u}^2_{\gamma}
    +\frac{\cpot}{2}|\D|+\int_{\D}\psi(u)\d\x
    =\IMd(u)+\EXd(u)+\int_{\D}\psi(u)\d\x.\label{eq:GL_nonlocal2}
\end{align}
We also introduce the following set of admissible solutions
\begin{equation*}
    \mathcal{K}:=\{v\in L^2(\D)\colon \psi(v)<+\infty\},
\end{equation*}
where $\psi$ is one of the options in~\eqref{eq:potential_psi_obstacle}--\eqref{regular_potential}. We point out that in case of an obstacle and logarithmic potential $\mathcal{K}=\{v\in L^2(\D)\colon |v|\leq 1\}$, whereas for the regular potential $\mathcal{K}=L^2(\D)$.

Next, we introduce a temporal discretization of the problem and derive the first and second order energy stable time-stepping schemes. 

\subsection{Semi-implicit first order scheme}\label{subsec:first_order_disc}
For $T>0$ and $K\in\mathbb{N}$ we define $\tt=T/K$, $t^k=k\tt$, $k=1,\dots,K$. 
Then, for a given $u^0\in L^2(\D)$ as already outlined in Section~\ref{sec:introduction1} we consider the following semi-implicit time discretization of~\eqref{eq:general_form}:
\begin{equation}
    \frac{1}{\tt}(u^{k}-u^{k-1})+\IM(u^{k})+\EX(u^{k-1})+\partial\psi(u^{k}) \ni 0.\label{eq:first_order_tk}
\end{equation}
We also note that for the obstacle potential, 
$\partial\psi(u)=\partial\mathbb{I}_{[-1,1]}(u)$~\eqref{eq:subdiff_obs}, the variational formulation of~\eqref{eq:first_order_tk} leads to the following semi-discrete variational inequality problem: Find $u^k\in\mathcal{K}$ such that the following holds:
 \begin{equation}
    \left(\frac{1}{\tt}(u^{k}-u^{k-1})+\IM(u^{k})+\EX(u^{k-1}),v-u^k\right)\geq 0,\quad \forall v\in\mathcal{K}.\label{eq:vi_tk1}
\end{equation}
 
Furthermore, it can be shown, see, e.g,~\cite{burkovska2021nonlocal}, that at each time step $k=1,\dots,K$ the solution $u^k$ of~\eqref{eq:first_order_tk} can be equivalently found by minimizing the following energy
\begin{equation*}
      u^k=\argmin \J_k(u),\quad k=1,\dots,K,
\end{equation*}
with
\begin{multline}
 \J_k(u):=\frac{\xi}{2}\norm{u}^2-\frac{1}{2}\norm{u^{k-1}}_\gamma^2-(\kernel*u^{k-1},u-u^{k-1})
    +\frac{1}{2\tt}\norm{u-u^{k-1}}^2\\
    +\frac{\cpot}{2}|\D|+\int_{\D}\psi(u)\d\x
    =g(u)+f(u^{k-1})+\left(\nabla f(u^{k-1}),u-u^{k-1}\right)\\
    +\frac{1}{2\tt}\norm{u-u^{k-1}}^2+\int_{\D}\psi(u)\d\x.
    \label{energy_tk}
\end{multline}
The above minimization problem admits a closed form solution given by the representation formula~\eqref{eq:prox_u}. Furthermore, the semi-implicit scheme is unconditionally energy stable, as stated by the following result.

\begin{proposition}[Energy stability for the first-order scheme]\label{prop:energy_stab1}
For the solution {$u^k\in\mathcal{K}$} of~\eqref{eq:first_order_tk} the following energy stability estimates hold
\begin{eqnarray}
    \E(u^k)\leq\J_k(u^k)\leq \J_k(u^{k-1})=\E(u^{k-1}).\label{energy_comp1}
\end{eqnarray}
\end{proposition}
\begin{proof}
    The inequality $\J_k(u^k)\leq \J_k(u^{k-1})$ and an equality $\E_k(u^{k-1})= \J_k(u^{k-1}) $ follow directly from the definition of $\J_k(u)$ and $\E(u)$, and fact that $u^k$ is the solution of~\eqref{energy_tk}. To prove the first inequality for $u\in\mathcal{K}$ we consider 
    \begin{multline*}
        \E(u)-\J_k(u)=f(u)-f(u^{k-1})-(\nabla f(u^{k-1}),u-u^{k-1})\\
        -\frac{1}{2\tt}\norm{u-u^{k-1}}^2
        \leq-\frac{1}{2\tt}\norm{u-u^{k-1}}^2\leq 0,
    \end{multline*}
where we have used the fact that $-f$ is a convex function and, thus, ${f}(v)\leq {f}(z)+\inner{\nabla {f}(z),v-z}$. Taking $u=u^k$ in the above estimate leads to the desired result and concludes the proof. 
\end{proof}

\subsection{Second-order scheme}
Next, we provide the second-order time-stepping scheme and derive the corresponding energy-stability estimates. 
\subsubsection{Implicit second-order scheme}
For the obstacle potential we consider the following second-order time-stepping scheme: Find $u^k\in\mathcal{K}$, such that
\begin{equation}
    \frac{1}{\tt}\left(u^{k}-u^{k-1}\right)+\IM(u^{k/2})+\EX(u^{k/2})+\partial\psi(u^{k})\ni 0,\label{eq:second_order_im_tk_obs}
\end{equation}
where $u^{k/2}=\frac{1}{2}(u^k+u^{k-1})$,  $\IM(u^{k/2})=\frac{1}{2}
\left(\IM(u^{k-1})+\IM(u^{k})\right)=\frac12\left(\xi u^{k-1}+\xi u^k\right)$, and $\EX(u^{k/2})=\frac{1}{2}
\left(\EX(u^{k-1})+\EX(u^{k})\right)=-\frac12\left(\kernel* u^{k-1}+\kernel* u^k\right)$.
Analogously, for the regular and logarithmic potentials we consider:
\begin{equation}
    \frac{1}{\tt}\left(u^{k}-u^{k-1}\right)+\IM(u^{k/2})+\EX(u^{k/2})+\frac{1}{2}\left(\psi^\prime(u^{k})+\psi^\prime(u^{k-1})\right)=0,\label{eq:second_order_im_tk_reglog}
\end{equation}
where $\psi^\prime$ is defined in~\eqref{eq:subdiff_reg} or~\eqref{eq:pot_logarithm}.
\begin{remark}
We note that, in the presence of obstacle potential $\partial\psi(u)$ is discretized implicitly (rather than via a midpoint average) in~\eqref{eq:second_order_im_tk_obs} to ensure that the new update remains admissible at the current time step, i.e., $|u^{k}|\leq 1$.  
Similar implicit time‐stepping schemes are employed when solving the local obstacle problems, see, e.g.,~\cite{glowinski1981}.
\end{remark}

It can be shown that at each time step the above semi-discrete problems~\eqref{eq:second_order_im_tk_obs} and~\eqref{eq:second_order_im_tk_reglog} are equivalent to the following minimization problem: for $k=1,\dots,K$ find $u^k$ that minimizes
\begin{equation}
     u^k=\argmin\Jim(u), \quad k=1,\dots,K,   \label{eq:energy_second_order_im}
\end{equation}
where
\begin{align}
\begin{aligned}
   \Jim(u):=\ \frac{1}{2\tt}\norm{u-u^{k-1}}^2 + \frac{1}{2}\left(g(u^{k-1})+g(u)\right) + \frac12\left(f(u^{k-1})+f(u)\right) \\
   \ + \frac12\left(\nabla f(u^{k-1}),u-u^{k-1}\right) + \frac12\left(\nabla g(u^{k-1}),u-u^{k-1}\right){+\int_{\D}\Psi_k(u)\d\x},
\end{aligned}\label{eq:Jim}
\end{align}
and 
\begin{equation}
    \Psi_k(u)=\begin{cases}
    \psi(u) \quad\text{(obstacle potential)},\\
    \displaystyle\frac{1}{2}\left(\psi(u)+\psi(u^{k-1})\right)+\frac12\psi^\prime(u^{k-1})\left(u-u^{k-1}\right)\quad\text{(regular or logarithmic potential)}.
    \end{cases} \label{eq:psi_definition}
\end{equation}

Next, we realize the above scheme by using a fixed point method. Namely,
given $u^{k-1}$, $k=1,\dots, K$ we find $u^k$ by solving 
\begin{multline}
    \frac{1}{\tt}\left(u_m^{k}-u^{k-1}\right)+\IM(u_m^{k/2})+\EX(u_{m-1}^{k/2})+\partial\Psi_k(u_m^{k})\\
     =\frac{1}{\tt}\left(u_m^{k}-u^{k-1}\right)+\xi u_m^{k/2}-\kernel*u_{m-1}^{k/2}+\partial\Psi_k(u_m^{k})\ni 0,\label{eq:second_oder_fi}
\end{multline}
where 
\[
u_m^{\frac{k}{2}}:=\frac{1}{2}\left(u_m^k+u^{k-1}\right)\quad\text{ and }\quad \partial\Psi_k(u_m^k)=\frac12\left(\psi^\prime(u_m^k)+\psi^\prime(u^{k-1})\right)
\]
for the regular and logarithmic potentials and  $\partial\Psi_k(u_m^k)=\partial\psi(u_m^k)$ for the obstacle potential.

The above is iterated for $m=1,2,\dots,M$, with initialization $u_0^k=u^{k-1}$, until the tolerance is reached: $\norm{u^{k}_m-u_{m-1}^k}<TOL$. Furthermore, each iteration of the fixed point method can be realized as a minimization problem. Namely, for $m=1,\dots,M$ with $u_0^k=u^{k-1}$, the solution $u_m^k$ solves 
\begin{equation}
    u^k_m = \argmin j_k[u_{m-1}^k](u),\quad m=1,\dots,M,\quad k=1,\dots,K,\label{eq:min_problem_fi}
\end{equation}
where
\begin{align}
    j_k[\zeta](u)=&\ \frac{1}{2\tt}\norm{u-u^{k-1}}^2 + \frac{1}{2}\left(g(u^{k-1})+g(u)\right) + \frac12\left(f(u^{k-1})+f(\zeta)\right) \nonumber\\
   &\ + \frac12\left(\nabla g(u^{k-1}),u-u^{k-1}\right) +  \frac12\left(\nabla f(u^{k-1}),u-u^{k-1}\right) \nonumber\\
   &\ + \frac12\left(\nabla f(\zeta),u-\zeta\right){+\int_{\D}\Psi_k(u)\d\x }.
   \label{eq:jim_k}
\end{align}
Then, similarly as for the first-order time-stepping scheme, at each $m$-th iteration the solution of~\eqref{eq:second_oder_fi}, or equivalently~\eqref{eq:min_problem_fi}, can be represented using the proximal operator~\eqref{eq:prox}. 

In particular, in case of the obstacle potential, for $m=1,2,\dots$, with $u_0^k=u^{k-1}$ until $\norm{u_m^k-u_{m-1}^k}<TOL$ we have 
\begin{align}
   u_m^k=&\ {\rm prox}_{\frac{1}{\lambda}\psi}\left(\frac{1}{\lambda}\left(\left(\frac{1}{\tt}-\frac{\xi}{2}\right)u^{k-1}+\kernel*u^{k/2}_{m-1}\right)\right)\nonumber\\
   =&\ P_{[-1,1]}\left(\frac{1}{\lambda}\left(\left(\frac{1}{\tt}-\frac{\xi}{2}\right)u^{k-1}+\kernel*u^{k/2}_{m-1}\right)\right),\label{eq:projection_m}
\end{align}
with $\lambda = 1/\tt + \xi/2$ and $\kernel*u^{k/2}_{m-1}=\kernel*\left(u^{k-1}+u_{m-1}^k\right)/2$. Analogously, for the regular~\eqref{regular_potential} and Flory-Huggins~\eqref{eq:pot_logarithm} potentials, the solution of~\eqref{eq:second_oder_fi}, or equivalently~\eqref{eq:min_problem_fi} admits the following representation
\begin{align}
   u_m^k={\rm prox}_{\frac{1}{2\lambda}\psi}\left(\frac{1}{\lambda}\left(\left(\frac{1}{\tt}-\frac{\xi}{2}\right)u^{k-1}+\kernel*u^{k/2}_{m-1}-\frac12\psi^\prime(u^{k-1})\right)\right).\label{eq:prox_m}
\end{align}
We recall that ${\rm prox}_{\frac{1}{2\lambda}\psi}=(I+\frac{1}{2\lambda}\partial\psi)^{-1}$ with $\lambda = 1/\tt+\xi/2$.
The proximal representations~\eqref{eq:projection_m} and~\eqref{eq:prox_m} above allow us to directly evaluate the solution, as it involves only explicit terms. 
Furthermore, for the regular potential, 
the proximal operator is obtained using Cardano's formula. That is, \eqref{eq:prox_m} becomes
\begin{equation}
    u_m^k = {\rm prox}_{\frac{1}{2\lambda}\psi}\left(\frac{1}{\lambda}\zeta^k_m\right)=\sqrt[3]{\tilde{\zeta}_m^k+\sqrt{(\tilde{\zeta}_m^k)^2 + \left({2\lambda}/{3c_F}\right)^3}} + \sqrt[3]{\tilde{\zeta}_m^k-\sqrt{(\tilde{\zeta}_m^k)^2 + \left({2\lambda}/{3c_F}\right)^3}},
    \label{eq:prox_formula_reg_order2}
\end{equation}
where $\zeta^k_m = \left(\left(\frac{1}{\tt}-\frac{\xi}{2}\right)u^{k-1}+\kernel*u^{k/2}_{m-1}-\frac12\psi^{\prime}(u^{k-1})\right)$ and $\tilde{\zeta}^k_m={2\lambda}\zeta^k_m/\cpot$. For the Flory-Huggins potential, however, there is no known closed-form expression for the proximal operator and we use Newton’s method to compute it.

\subsubsection{Explicit second order scheme}
Given $u^{k-1}$, $k=1,\dots,K$, we find the solution at the next time step, $u^{k}$, using a two-stage scheme (which is obtained by taking two steps of the fixed point iteration of the second order implicit scheme~\eqref{eq:second_oder_fi}): First, given $u^{k-1}$ we find $u_1^k$ by solving
\begin{equation}
    \frac{1}{\tt}\left(u_1^k-u^{k-1}\right)+\IM(u_1^{k/2})+\EX(u^{k-1})+\partial\Psi_k(u_1^k) \ni 0,
\end{equation}
and having $u_1^k$ and an old iterate $u^{k-1}$ we find $u^k$ by solving
\begin{equation}
    \frac{1}{\tt}\left(u^k-u^{k-1}\right)+\IM(u^{{k}/{2}})+\frac{1}{2}\left(\EX(u^{k-1})+\EX(u_1^k)\right)+\partial\Psi_k(u^k)\ni 0.
\end{equation}
This is equivalent to solving the following minimization problems:
\begin{align}\label{eq:energy_second_order_ex}
\begin{aligned}
     u^k=&\ \argmin\Jex(u)\quad\text{with}\quad
   \Jex(u):=j_k[u^k_1](u),\\
   \text{and}\quad u_1^k=&\ \argmin j_k[u^{k-1}](u).
\end{aligned}
\end{align}
Again, the above scheme can be realized by using the representation formula~\eqref{eq:prox_m} with $m=2$. 
In the Algorithm~\ref{alg:fixed_point_iteration} we summarize the computational procedure for the second order scheme. 
\begin{algorithm}[ht!]
\caption{Solution algorithm}
\label{alg:fixed_point_iteration}
\begin{algorithmic}[1]
\State \textbf{Input:} $u^{0}$, $TOL,M,K>0$.
\State \textbf{Output:} Solution at a final time step $u^{K}$.
\For{$k=1,\dots,K$}
\State \textbf{Set} $m=1$ and $u^k_0=u^{k-1}$.
\State Compute $\kernel*u^{k-1}$.
        \While{$m < M$ \textbf{or} $\norm{u_{m-1}^k-u^k_m}>TOL$}
        \State Compute $\kernel*u_{m-1}^k$ for $m>1$.
          \State Evaluate $u_m^k$ from~\eqref{eq:projection_m} or~\eqref{eq:prox_m}.
           \State {Set} $m=m+1$ and continue.
        \EndWhile
        \State Set $u^k=u_m^k$ and continue.
\EndFor
\end{algorithmic}
\end{algorithm}

\subsubsection{Energy stability estimates}
Next, we establish energy stability estimates for the second-order implicit scheme. 
First, we derive an auxiliary estimate for an upper bound of the nonlocal Ginzburg-Landau energy~\eqref{eq:GL_nonlocal}.

We define the set $\mathcal{K}_{\rho}$ as
\begin{equation}
    \mathcal{K}_\rho:=\{v\in\mathcal{K}\colon |v|\leq\rho, \;\;\rho>0\},
\end{equation}
and impose the following assumption: 
\begin{assumption}
For the regular potential we assume that all solutions ($u^k$, $u^{k-1}$, $u_m^k$) under consideration are bounded by some $\rho<\infty$. For the logarithmic potential we assume the same for $\rho<1$.
\end{assumption}
Then, it is easy to show that the following holds true.
\begin{proposition}
    For the regular potential and any $\rho<\infty$ or for the logarithmic potential and any $\rho<1$, there exists $C_\rho<\infty$, such that $|\psi^{\prime\prime}(u)|<C_\rho$ for any $u\in\mathcal{K}_{\rho}$ almost everywhere in $\Omega$.
\end{proposition}
Next, we  derive the following bound for the energy.
\begin{proposition}\label{prop:energy_comparison_upperbd}
For the obstacle potential $\psi(u)$~\eqref{eq:potential_psi_obstacle}, and assuming $\tt<2/\xi$  we get:
    \begin{equation}
        \E(u)\leq \Jim(u), \quad\forall u\in\mathcal{K}.\label{eq:energy_comparison_upperbd}
    \end{equation}
    Furthermore, the above holds true for the regular and logarithmic potential function $\psi$~\eqref{regular_potential},~\eqref{eq:pot_logarithm}, under an additional assumption 
    $u\in\mathcal{K}_\rho$ and $u^{k-1}\in\mathcal{K}_{\rho}$ for $\tt<2/(\xi+C_\rho
)$.
\end{proposition}
\begin{proof}
First, we consider the case of an obstacle potential. Using the definitions of $\E(u)$,\ $\Jim(u)$,~\eqref{eq:tk_terms_en}, \eqref{eq:tk_terms}, and the fact that  $f$ is concave, i.e., $-f$ satisfies~\eqref{eq:convexity}, we estimate
\begin{multline*}
    \E(u)-\Jim(u)= -\frac{1}{2\tt}\norm{u-u^{k-1}}^2
    +\frac12\left(\IMd(u)-\IMd(u^{k-1}) -(\nabla\IMd(u^{k-1}),u-u^{k-1})\right)\\
    +\frac12\left( \EXd(u)-\EXd(u^{k-1})-(\nabla\EXd(u^{k-1}),u-u^{k-1})\right)\\
    \leq 
    -\frac{1}{2\tt}\norm{u-u^{k-1}}^2 +\frac{\xi}{4}\norm{u}^2-\frac{\xi}{4}\norm{u^{k-1}}^2+\frac{\xi}{2}\left(u^{k-1},u^{k-1}-u\right)\\
    \leq \left(\frac{\xi}{4}-\frac{1}{2\tt}\right)\norm{u-u^{k-1}}^2\leq 0,
\end{multline*}
where we have used the the relation $2a(a-b)=a^2+(a-b)^2-b^2$.
For the regular and logarithmic potentials the proof follows along similar lines. In particular, using the definition of $\Psi_k$~\eqref{eq:psi_definition} and an estimate for the remainder of the Taylor's series we first consider
\begin{multline*}
    \psi(u)-\Psi_k(u)=\frac{1}{2}\psi(u) -\frac12\psi(u^{k-1})-\frac12\psi^\prime(u^{k-1})\left(u-u^{k-1}\right)\\
    \leq\frac12\int_0^1|\psi^{\prime\prime}(u_{\tau})|(u-u^{k-1})^2(1-\tau)\d\tau
    =\frac14|\psi^{\prime\prime}(u_{\tau})|(u-u^{k-1})^2,
\end{multline*}
where $u_\tau=\tau u+(1-\tau)u^{k-1}$, $\tau\in [0,1]$ and $|u_\tau|\leq \rho$ for $|u|\leq \rho$. Note that $|u|<\rho$ holds since $u, u^{k-1}\in\mathcal{K}_\rho$ and $\mathcal{K}_\rho$ is convex.
Then, using the above estimate with the assumptions that $|\psi^{\prime\prime}(u)|\leq C_\rho$ and $\Delta t<2/(\xi+C_\rho)$ we obtain
\begin{multline*}
    \E(u)-\Jim(u) \leq \left(\frac{\xi}{4}-\frac{1}{2\tt}\right)\norm{u-u^{k-1}}^2 + \int_{\D}\left(\psi(u)-\Psi_k(u)\right)\d\x\\
    \leq\left(\frac{\xi}{4}-\frac{1}{2\tt}+\frac{{C_\rho}}{4}\right)\norm{u-u^{k-1}}^2\leq 0.
\end{multline*}
 This concludes the proof.
\end{proof}
\begin{remark}
    We note $\xi$ is typically chosen to be sufficiently small, thus $\tt<2/\xi$ is not restrictive condition in the above statement of the theorem. 
\end{remark}

Next, we establish the energy stability property for the second order implicit and explicit schemes. First, we prove an intermediate estimate for solutions $u_m^k$ of~\eqref{eq:energy_second_order_im}.

\begin{lemma}
For the solutions $u^k, u_m^k\in\mathcal{K}$ of of~\eqref{eq:energy_second_order_im}  and~\eqref{eq:min_problem_fi},  respectively, under the conditions of Proposition~\ref{prop:energy_comparison_upperbd} the following holds true
\begin{align}
     \Jim(u^k)\leq \Jim(u_m^k)\leq \Jim(u_{m-1}^k)
    \leq  \Jim(u^{k-1})=\E(u^{k-1}).\label{energy_comp2}
\end{align}
\end{lemma}
\begin{proof}
 The equality $\E_k(u^{k-1})= \Jim(u^{k-1})$ follows directly from the definition of $\Jim(u)$~\eqref{eq:Jim} and $\E(u)$~\eqref{eq:GL_nonlocal2}. The 
inequality $\Jim(u^k)\leq\Jim(u_m^k)$ holds true, since $u^k$ is the minimizer in~\eqref{eq:energy_second_order_im} and hence $\Jim(u^k)\leq\Jim(v)$ for any $v\in\mathcal{K}$. 
Next, we prove the second estimate in~\eqref{energy_comp2}, $\Jim(u_m^k)\leq\Jim(u_{m-1}^k)$. 

Using the definitions~\eqref{eq:jim_k},~\eqref{eq:Jim} and~\eqref{eq:tk_terms_en},~\eqref{eq:tk_terms} we obtain
 \begin{multline}
           j_k[\zeta](u)=\Jim(u)+\frac{1}{2}\left(f(\zeta)-f(u)\right)+\frac12\left(\EX(\zeta),u-\zeta\right)\\
           =
           \Jim(u)-\frac14\norm{\zeta}^2_\gamma+\frac14\norm{u}^2_\gamma-\frac12\left(\gamma*\zeta,u-\zeta\right)\\
=\Jim(u)+\frac14\left(\norm{\zeta}^2_\gamma+\norm{u}^2_\gamma-\left(\gamma*\zeta,u\right)\right)
    =\Jim(u)+\frac14\norm{u-\zeta}^2_\gamma.\label{eq:proof_jk_relation}
     \end{multline}
Since, $u_m^k$ is the minimizer of~\eqref{eq:min_problem_fi} it follows that
 \begin{equation*}
     j_k[u^{k}_{m-1}](u_m^k)\leq j_k[u_{m-1}^k](u),\quad\forall {u\in\mathcal{K}}.
 \end{equation*}
Now using~\eqref{eq:proof_jk_relation} with $\zeta=u_{m-1}^k$ and the above estimate leads to
 \begin{equation*}
    \Jim(u_m^k)\leq \Jim(u_m^k)+\frac14\norm{u_m^k-u_{m-1}^k}_\gamma^2=j_k[u_{m-1}^k](u_m^k)\leq\Jim(u)+\frac14\norm{u-u_{m-1}^k}_\gamma^2,\quad\forall {u\in\mathcal{K}}.
 \end{equation*}
 Then, taking $u=u_{m-1}^k$ in the above we obtain that
 \begin{equation}
     \Jim(u_m^k)\leq\Jim(u_{m-1}^k),\label{eq:proof_m_ener}
 \end{equation}
 which proves the second inequality in~\eqref{energy_comp2}. The remaining inequality $\Jim(u^k_{m-1})\leq\Jim(u^{k-1})$ in~\eqref{energy_comp2} follows directly by the induction argument. More specifically, taking $m=1$ in~\eqref{eq:proof_m_ener} leads to $\Jim(u_1^k)\leq\Jim(u_0^k)=\Jim(u^{k-1})$, then the rest is iterated using an induction procedure. This concludes the proof.  
\end{proof}

The following theorem establishes the stability result for the second order explicit and implicit schemes. 
\begin{theorem}[Energy stability estimates for the second order scheme]

Let $u^k$ be the solution either of the second order implicit~\eqref{eq:energy_second_order_im} or second order explicit~\eqref{eq:energy_second_order_ex} schemes. 
Then, under the conditions of Proposition~\eqref{prop:energy_comparison_upperbd}, we have the following energy stability estimates for $k=1,\dots,K$:
\begin{equation}
    \E(u^k)\leq \Jim(u^k)\leq \E(u^{k-1}).\label{eq:energy_stability2}
\end{equation}
\end{theorem}
\begin{proof}
    The proof follows directly by combining~\eqref{eq:energy_comparison_upperbd} together with the estimates~\eqref{energy_comp2} (with $m=2$ for the second-order explicit scheme).
\end{proof}

Next, we derive the convergence result of the fixed point iteration. 
    \begin{proposition}[Convergence of the fixed point method]
        Let $u^k$ be the solution of~\eqref{eq:second_order_im_tk_obs} or~\eqref{eq:second_order_im_tk_reglog}, $k=1,\dots,K$. Then, for $\tt<2/\cpot$ the solution algorithm (Algorithm~\ref{alg:fixed_point_iteration}) converges linearly with respect to the number of fixed point iterations $m\in\mathbb{N}$:
        \begin{equation}
\norm{u^k-u^k_m}\leq \left(\frac{\cker}{2\lambda}\right)^m\norm{u^k-u^{k-1}},\quad k=1,\dots,K,
        \end{equation}
        with $\xi=\cker-\cpot{\geq 0}$, $\cker=\kernel*1$, and $\lambda=1/\tt+\xi/2$. 
    \end{proposition}
    \begin{proof}
       We present the proof for the obstacle potential, since the proof for the case of regular or Flory-Huggins potentials follows along the same lines. For the solution $u_m^k$ of~\eqref{eq:min_problem_fi} that admits~\eqref{eq:projection_m}, we denote
   $\zeta^k_m=\kernel*u_{m-1}^{k/2}+({1}/{\tt}-{\xi}/{2})u^{k-1}$. Then, using~\eqref{eq:prox_m}, Young's inequality for convolutions, and the fact that the proximal operator is nonexpansive~\cite{boyd2013proximal}, i.e., Lipschitz continuous with constant one, we estimate 
       \begin{multline*}
           \norm{u_m^k-u_{m-1}^k}=\norm{\proxl \left(\frac{\zeta_m^k}{\lambda}\right)-\proxl\left(\frac{\zeta_{m-1}^k}{\lambda}\right)}\\
           \leq\frac{1}{\lambda}\norm{\zeta_{m}^k-\zeta_{m-1}^k}=\frac{1}{{\lambda}}\norm{\kernel*\left(u_{m-1}^{k/2}-u_{m-2}^{k/2}\right)}\leq\frac{1}{2\lambda}{\norm{\kernel}_{L^1(\D)}}\norm{u_{m-1}^{k}-u_{m-2}^k}\\
           = \frac{\cker}{2/\tt+\xi} \norm{u_{m-1}^k-u_{m-2}^k}.
       \end{multline*}
    Thus, for $\cker/(2\lambda)={\cker}/({2/\tt+\xi})<1$, which holds for $\tt<2/\cpot$, we obtain that the fixed point operator is a contraction mapping and by the Banach fixed point theorem there exists a unique $u^k$ such that 
    \begin{multline*}
        \norm{u^k-u_{m+1}^k}\leq \frac{\cker}{2/\tt+\xi}\norm{u^k-u_{m}^k}\leq \left(\frac{\cker}{2/\tt+\xi}\right)^2\norm{u^k-u_{m-1}^k}\\
        \leq 
        \left(\frac{\cker}{2/\tt+\xi}\right)^{m+1}\norm{u^k-u^{k-1}},
    \end{multline*}
    which concludes the proof. 
  \end{proof}

\section{Spatial discretization}\label{sec:spatial_disc}
In this section we present a spatial discretization of~\eqref{eq:general_xi} based on the spectral collocation method~\cite{shen2011book,trefethen2000spectral,DuJuLiQiao2018} with the discrete Fourier transform for the evaluation of the convolution. 
For simplicity of presentation we focus on a model problem posed in a two-dimensional domain, $\Omega = (-X,X)\times(-Y,Y)\subset \bbR^2$ and supplemented with periodic boundary conditions. The results can be easily generalized and remains valid in one- and three- dimensional settings. 

We define the discretization of $\Omega$ by a collocation points which are equidistantly spaced in each dimension:
\begin{align*}
    \Omega_h = \left\{(x_i,y_j) : x_i=-X+ih_x, y_j=-Y+jh_y,\quad 
    0\le i\le N_x-1, ~0\le j\le N_y-1\right\},
\end{align*}
with uniform mesh sizes $h_x = 2X/N_x$, $h_y=2Y/N_y$,
where $N_x$, $N_y$ are even numbers, and we set $h = \max\{h_x,h_y\}$. We introduce index sets for spatial and frequency domains:
\begin{eqnarray*}
    \Sh &=& \left\{(i,j)\in\bbZ^2\colon 0\le i\le N_x-1,\quad 0\le j\le N_y-1\right\},\\
    \Shw &=& \left\{(l,m)\in\bbZ^2\colon -\frac{N_x}{2}+1\le l\le \frac{N_x}{2},~-\frac{N_y}{2}+1\le m\le \frac{N_y}{2}\right\}.
\end{eqnarray*}
For all grid periodic functions $u,v\colon\Omega_h\to\mathbb{R}$ we define  
the discrete $L^2$ inner product $\innerd{\cdot,\cdot}$ and norm $\normd{\cdot}$:
\begin{equation*}
    \innerd{u,v}=h_xh_y\sum_{(i,j)\in\Sh}u_{ij}v_{ij},\quad \normd{u}=\sqrt{\innerd{u,u}},
\end{equation*}
where and in what follows we use $u_{ij}=u(x_i,y_j)$. We also define a discrete Fourier transform
\begin{eqnarray}
    \widehat{u}_{lm}=(\mathcal{F}_Nu)_{lm}=\sum_{(i,j)\in\Sh}{u_{ij}}e^{- \ii\pi\left(\frac{lx_i}{X} + \frac{my_j}{Y}\right)},\quad(l,m)\in\Shw,\label{f_tranform}
\end{eqnarray}
and the corresponding inverse discrete Fourier transform 
\begin{eqnarray}
   {u_{ij}}=\left(\mathcal{F}_N^{-1}\widehat{u}\right)_{ij}=\frac{1}{N_xN_y}\sum_{(l,m)\in \widehat{S}_h}\widehat{u}_{lm}e^{ \ii\pi\left(\frac{lx_i}{X} + \frac{my_j}{Y}\right)},\quad (i,j)\in\Sh.\label{f_inverse}
\end{eqnarray}
Here, by $\astcirc$ we denote a discrete circular convolution for the periodic functions $u,v$:
\begin{equation*}
    (u\astcirc v)_{ij} = h_xh_y\sum_{(p,q)\in\Sh}u_{i-p,j-q}v_{pq},~~~~~(i,j)\in\Sh.\label{convolution_def}
\end{equation*}

\subsection*{Evaluation of the convolution}
We employ discrete Fourier transform using FFT to efficiently compute the discrete convolution in $\mathcal{O}(N\log N)$ complexity, with $N=N_xN_y$. More specifically, taking into account that
\begin{eqnarray*}
    \widehat{(u\astcirc v)}_{lm} = h_xh_y\widehat{u}_{lm}\widehat{v}_{lm},\quad (l,m)\in\Shw,
\end{eqnarray*}
we have that
    \begin{align*}
        (u \astcirc v)_{ij}=\left(\mathcal{F}_N^{-1}\widehat{(u\astcirc v)}\right)_{ij}=
        \frac{h_xh_y}{N_xN_y}\sum_{(l,m)\in \Shw}\widehat{u}_{lm}\widehat{v}_{lm}e^{ \ii\pi\left(\frac{lx_i}{X} + \frac{my_j}{Y}\right)},~~~~~(i,j)\in S_h.\label{discrete_conv}
    \end{align*}
Now, by taking $u\equiv 1$ and $v=\kernel_N$ in the above we can directly compute $\xi_N=\cker^N-\cpot$ with 
\begin{eqnarray*}
    \cker^N=\kernel_N\astcirc 1 = h_xh_y \widehat{\kernel_N}_{00}=h_xh_y\sum_{(i,j)\in\Sh}\kernel_{N,ij}.
\end{eqnarray*}
Furthermore, the discrete version $\mathcal{L}_N$ of the operator $\calL$ in~\eqref{nonlocal_op} reads as $\calL_N u=\cker^N u-\kernel_N\astcirc u$ for $u\colon\D_h\to\mathbb{R}$.

\subsection*{Discrete representation formula}

Now, we present a full discretization of the nonlocal problem~\eqref{eq:general_xi} and the corresponding discretization of the representation formula. For simplicity we focus on the first order time-stepping scheme~\eqref{eq:first_order_tk_intro} that can be directly extended to the second-order scheme.

We denote $ V_N:={\rm Re}\left({\rm span}_{\mathbb{C}}\{e^{\ii\pi\left(\frac{ l x}{X}+\frac{ m y}{Y}\right)}\colon (l,m)\in\Shw\}\right)$ and construct an approximate solution to~\eqref{eq:first_order_tk_intro}, $u_N^k \in\mathcal{K}_N\subset V_N$, 
with $u_N(0)=I_N(u_0)$, where $I_N$ is a nodal interpolant and we define
\begin{equation*}
    \mathcal{K}_N:=\{v_N\in V_N\colon \psi(v_N(x_i,y_j))<+\infty,\quad i,j\in\Sh\}.
\end{equation*}
We also introduce a Lagrangian basis associated with the grid notes such that
\begin{align*}
u_N^k(x,y)=\frac{1}{N_xN_y}\sum_{(l,m)\in\Shw}\hat{u}_{lm}^ke^{\ii\left(\frac{\pi l x}{X}+\frac{\pi my}{Y}\right)}=
    \sum_{(i,j)\in\Sh}u_N^k(x_i,y_j)\phi_i(x)\phi_j(y),
\end{align*}
with $\phi_j(x_i)=\delta_{ij}$, $\forall i,j\in\Sh$. It is known that $\phi_j(x)=\frac{1}{N_x}\sin \left(N_x\pi\frac{x-x_j}{X}\right)\cot \left(\pi\frac{x-x_j}{X}\right)$. 
Now, for $k=1,\dots,K$ introducing a discrete Lagrange multiplier $\mu_N^k\colon\D_h\to\bbR$, $\mu_N^k\in\partial\psi(u_N^k)$ we can derive a fully discrete version of~\eqref{eq:first_order_tk_intro}: Find $u_N^k\in\mathcal{K}_N$ such that
\begin{align*}
   \left(\frac{1}{\tt}+\xi_N\right)u_{ij}^{k}+\mu_{ij}^k=\frac{1}{\tt}u^{k-1}_{ij}+\left(\kernel_N\astcirc u^{k-1}\right)_{ij},\quad i,j\in \Sh,
\end{align*}
where $u_{ij}=u_N(x_i,y_j)$ and $\mu_{ij}^k=(u_{ij}^k)^3$, $\mu_{ij}^k = \frac{\thetalog}{2}\left(\ln(1+u_{ij}^k)-\ln(1-u_{ij}^k)\right) $ and $\mu_{ij}^k={\mu}_{+,ij}^k-{\mu}_{-,ij}^k$, ${\mu}_{\pm,ij}^k\geq 0$ for the regular, logarithmic and obstacle potentials, respectively.
Invoking the properties of the discrete Lagrange multipliers this reduces to the following pointwise relation:
\begin{eqnarray*}
    u^k_{ij} = \proxlN\left(\frac{1}{\lambda_N}\left[(\gamma\astcirc u^{k-1})_{ij} +\frac{1}{\tt}u^{k-1}_{ij}\right]\right),\quad \lambda_N=\xi_N+1/\tt>0.
\end{eqnarray*}
We can also show that the above discrete problem is the Karush-Kuhn-Tucker (KKT) system of the following discrete minimization problem of finding $U=(u_{ij})_{ij\in\Sh}\in\mathbb{R}^N$ that minimizes
\begin{multline*}
    \min_{U\in\bbR^{N}} \J_N^k(U):=h_xh_y\sum_{(i,j)\in\Sh}\left[\frac{\xi_N}{2}u^2_{ij}-\frac{1}{2}(\kernel \astcirc u^{k-1})_{ij}\cdot u_{ij}^{k-1}\right.\\
    \left.-(\kernel \astcirc u^{k-1})_{ij}\cdot(u_{ij}-u_{ij}^{k-1})\right.\\
    \left. +\frac{1}{2\tt}(u_{ij}-u_{ij}^{k-1})^2+\frac{\cpot}{2}|\D|+\psi(u_{ij})\right].
\end{multline*}
Following the same lines as in the proof of Proposition~\ref{prop:energy_stab1} we derive the discrete energy stability estimates: 
\begin{eqnarray*}
    \E_N(U^k)\leq\J^k_N(U^k)\leq \J^k_N(U^{k-1})=\E_N(U^{k-1}),
\end{eqnarray*} 
where 
\begin{equation*}
    \E_N(U)=h_xh_y\sum_{(i,j)\in\Sh}\left[\frac{\xi_N}{2}u^2_{ij}-\frac{1}{2}(\kernel \astcirc u^{k-1})_{ij}\cdot u_{ij}^{k-1}+\frac{\cpot}{2}|\D|+\psi(u_{ij})\right].
    \label{eq:GL_disc}
\end{equation*}

\section{Numerical experiments}\label{sec:numerics}
In this section, we present several numerical experiments in one-, two-, and three-dimensional settings. For numerical simulations, we consider a nonlocal kernel $\gamma$ to be a scaled Gaussian~\cite{DuJuLiQiao2018}, which can be periodically extended to $\bbR^n$ with respect to $\Omega$ 
\begin{eqnarray}
    \kernel(\x-\xx)=\kerneld(|z|) = \frac{4\varepsilon^2}{\pi^{n/2}\delta^{n+2}}e^{-\frac{|z|^2}{\delta^2}},~~z\in\D,~\delta>0, \label{kernel}
\end{eqnarray}
where the scaling is chosen such that $\int_{\mathbb{R}^n}\kerneld(\x)|\x|^2\d\x=2n$ and the parameter $\varepsilon>0$ is introduced to resemble the local operator $Lu\approx-\varepsilon^2\upDelta u$ for vanishing nonlocal interactions, $\delta\to 0$. Here, $\delta>0$ is a nonlocal interaction parameter that defines an approximate of nonlocal interactions. With the choice of the kernel~\eqref{kernel} we obtain $\cker=(\kernel*1)={4\varepsilon^2}/{\delta^2}$. 

For the spatial discretization we employ the Fourier spectral collocation method on a uniform mesh with $N$ collocation points, with $N$ specified for each case separately.  For the temporal discretization we consider time-stepping schemes described in Section~\ref{sec:temporal_disc}.

We note that the proximal operator defined in~\eqref{eq:prox} is computed explicitly for the regular and obstacle potential
using Cardano's formula~\eqref{eq:prox_formula_reg_order1},~\eqref{eq:prox_formula_reg_order2} and a pointwise projection~\eqref{eq:proj_u},~\eqref{eq:projection_m}, respectively. For the logarithmic potential, we use a Newton method to approximate the proximal operator.

\subsection{Convergence rates in time}
 We investigate numerically the performance of the first- and second-order time-stepping schemes. We consider the numerical simulation of~\eqref{eq:general} using various time step sizes $\Delta t = 0.005 \times 2^{-k},~~k = 0, 1, 2, \dots, 5$ and compute the error of the corresponding numerical solution with the benchmark solution, which is obtained by the second order implicit scheme with $\Delta t = 0.005\times 2^{-10}$, and $TOL = 10^{-15}$ in Algorithm~\ref{alg:fixed_point_iteration}. 
 
 \begin{example}\label{ex:convergence}
     We choose $\D=(-1,1)^n$, for $n=2,3$, and discretize it using a uniform Cartesian grid with $N=512^n$. We set the parameter $\cpot=1$, the final time $T=0.2$, and the initial conditions
    \begin{align*}
        u_0({\bf{x}}) &= \prod_{i=1}^n \cos(\pi x_i), \quad {\bf{x}}=(x_1,\dots,x_n)\in\bbR^n.
    \end{align*}
    We test all three potentials in two dimensions and additionally test the obstacle potential in three dimensions.
 \end{example}

In Figure~\ref{convergenceRates}, we plot the $L^2$-error of the numerical solution for the obstacle potential in two and three dimensions, for different time step sizes and for different values of $\varepsilon$ and $\delta$. More specifically, we consider $\xi = 0$, $\xi = 1.56$, $\xi = 3$, and $\xi = 8$, corresponding to $(\varepsilon,\delta)$ pairs of $(0.05, 0.1)$, $(0.08, 0.1)$, $(0.1, 0.1)$, and $(0.18, 0.12)$, respectively. We observe the expected first- and second-order convergence rates for the first- and second-order schemes, respectively, which are independent of $\delta$ and $\varepsilon$.
\begin{figure}[t]
    \centering
    \begin{subfigure}[h]{0.49\textwidth}
        \centering
        \includegraphics[width=\textwidth]{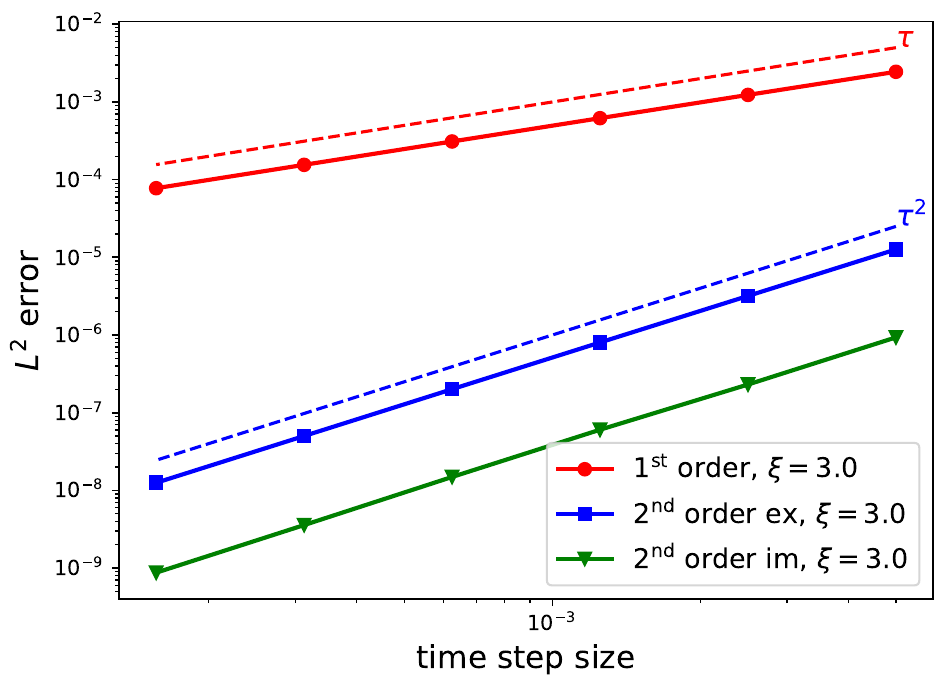}
        \caption{2D case with $\varepsilon = 0.1$ and $\delta = 0.1$}
        \label{error2D_1}
    \end{subfigure}
    \begin{subfigure}[h]{0.49\textwidth}
        \centering
        \includegraphics[width=\textwidth]{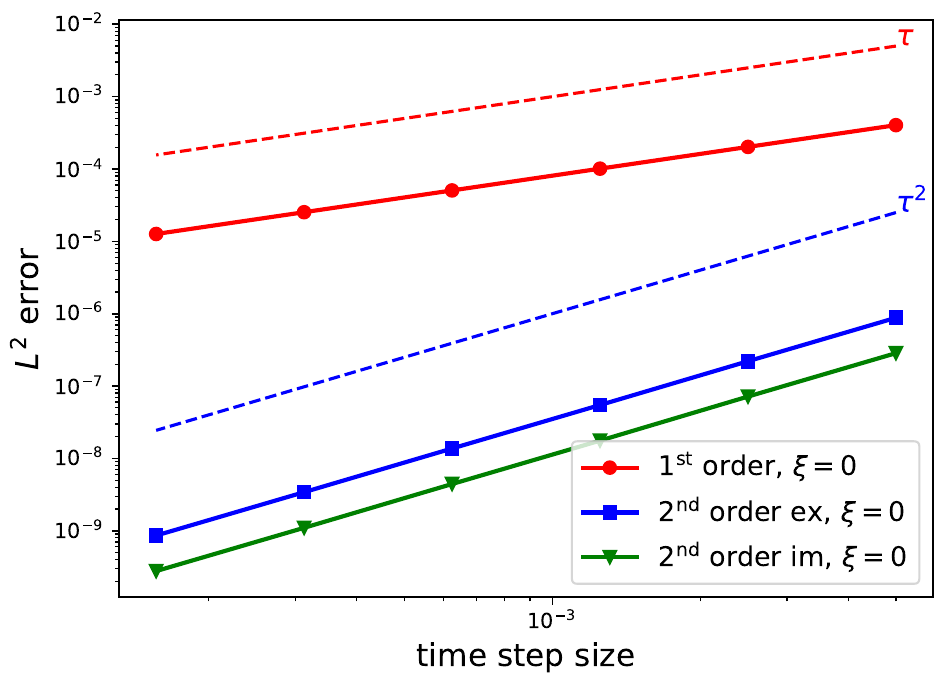}
        \caption{2D case with $\varepsilon = 0.05$ and $\delta = 0.1$}
        \label{error2D_2}
    \end{subfigure}
    \begin{subfigure}[h]{0.49\textwidth}
        \centering
        \includegraphics[width=\textwidth]{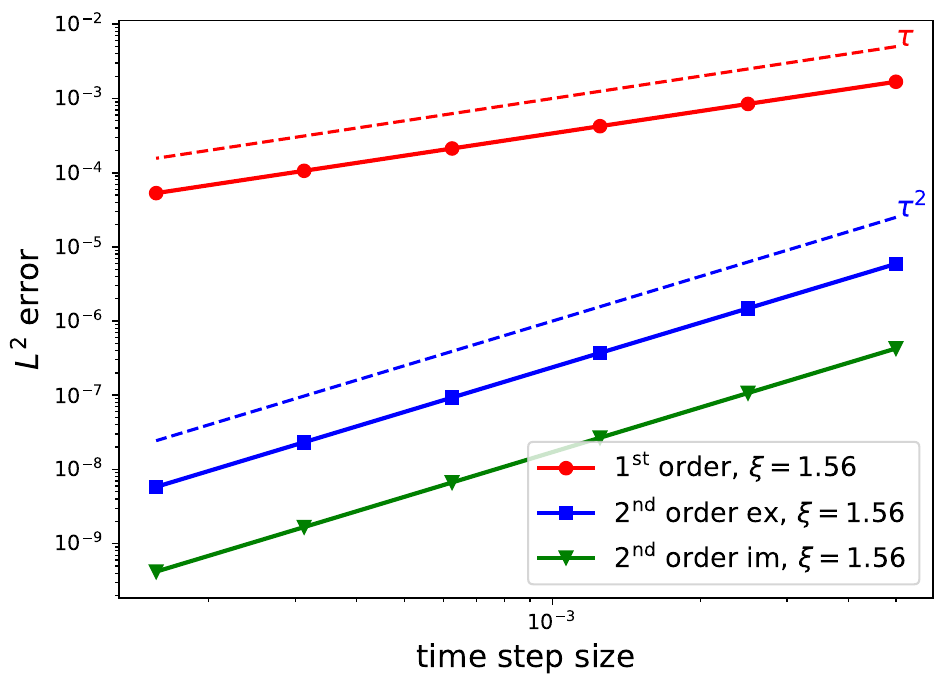}
        \caption{3D case with $\varepsilon=0.08$ and $\delta=0.1$}
        \label{error3D_1}
    \end{subfigure}
    \begin{subfigure}[h]{0.49\textwidth}
        \centering
        \includegraphics[width=\textwidth]{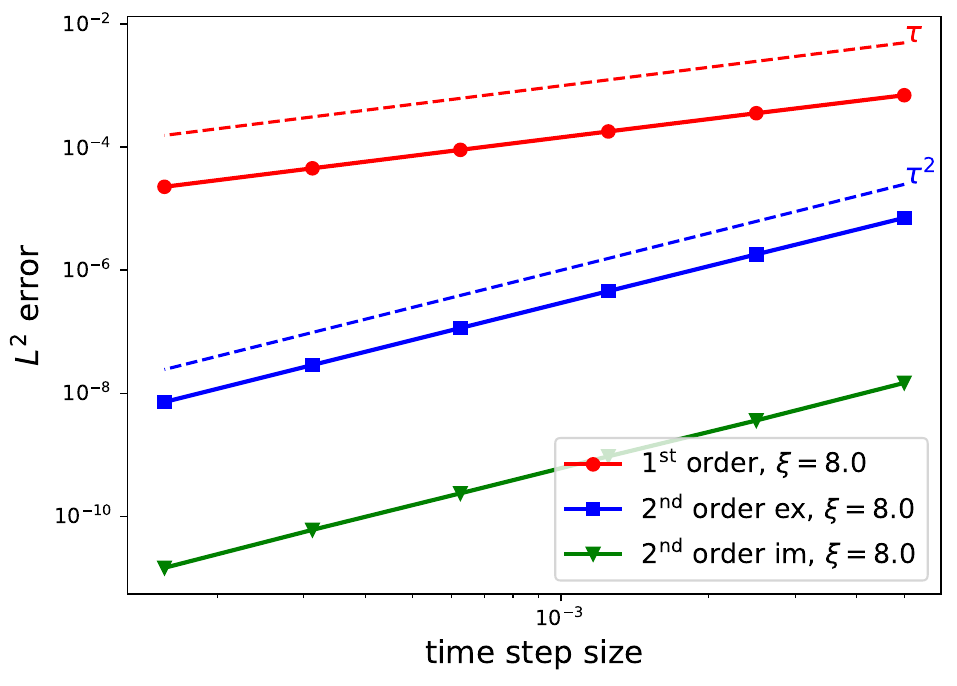}
        \caption{3D case with $\varepsilon=0.18$ and $\delta=0.12$}
        \label{error3D_2}
    \end{subfigure}
    \caption{Convergence rates in time of  the first-order, second-order explicit, and second-order implicit schemes for two- and three-dimensional test cases in Example~1 for obstacle potential.}
    \label{convergenceRates}
\end{figure}

Furthermore, in Figure~\ref{convergenceRates2}, we plot the $L^2$-error of the numerical solution for the regular and logarithmic potentials (with $\theta_c=0.2$). We consider the cases  $\xi=0$ and $\xi=3.0$ for both potentials and observe the expected first- and second-order convergence rates, independent of $\delta$ and $\varepsilon$. 

\begin{figure}[ht]
    \begin{subfigure}[h]{0.49\textwidth}
        \centering
        \includegraphics[width=\textwidth]{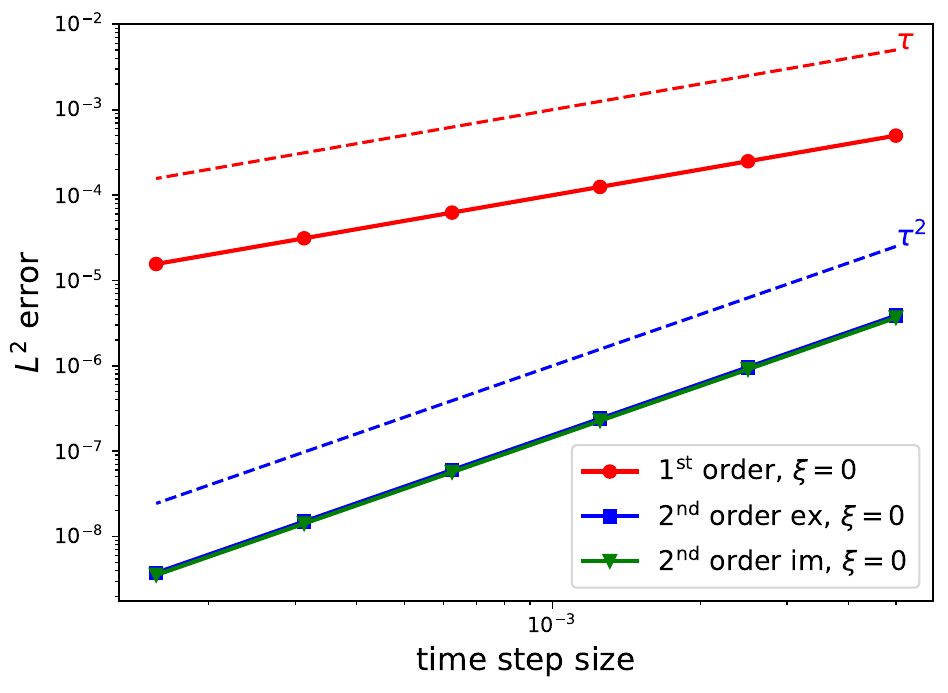}
    \end{subfigure}
\begin{subfigure}[h]{0.49\textwidth}
        \centering
        \includegraphics[width=\textwidth]{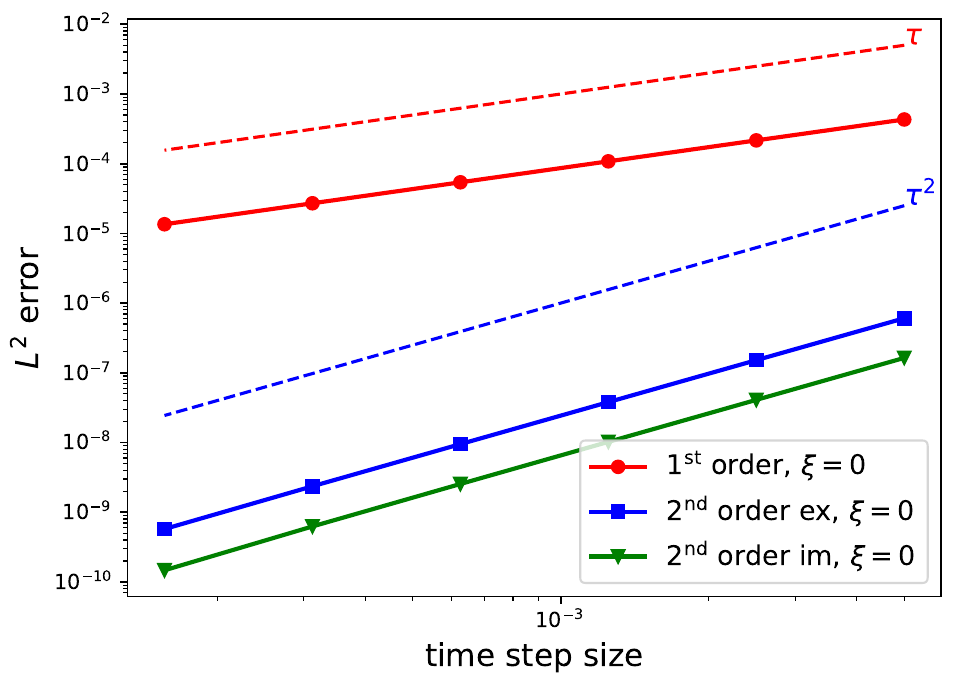}
    \end{subfigure}
    \begin{subfigure}[h]{0.49\textwidth}
        \centering
        \includegraphics[width=\textwidth]{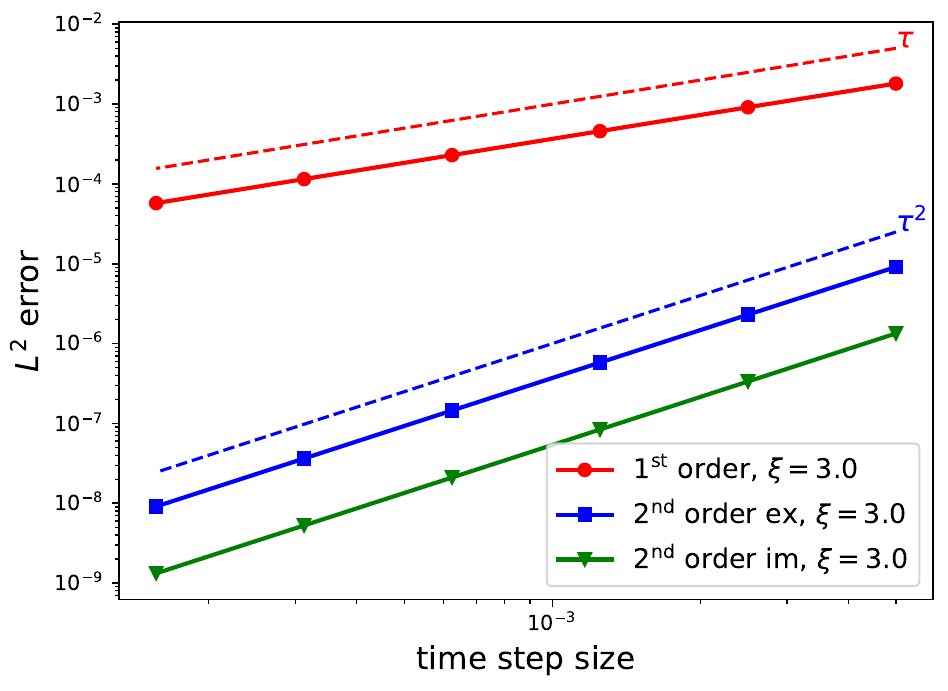}
    \end{subfigure}
    \begin{subfigure}[h]{0.49\textwidth}
        \centering
        \includegraphics[width=\textwidth]{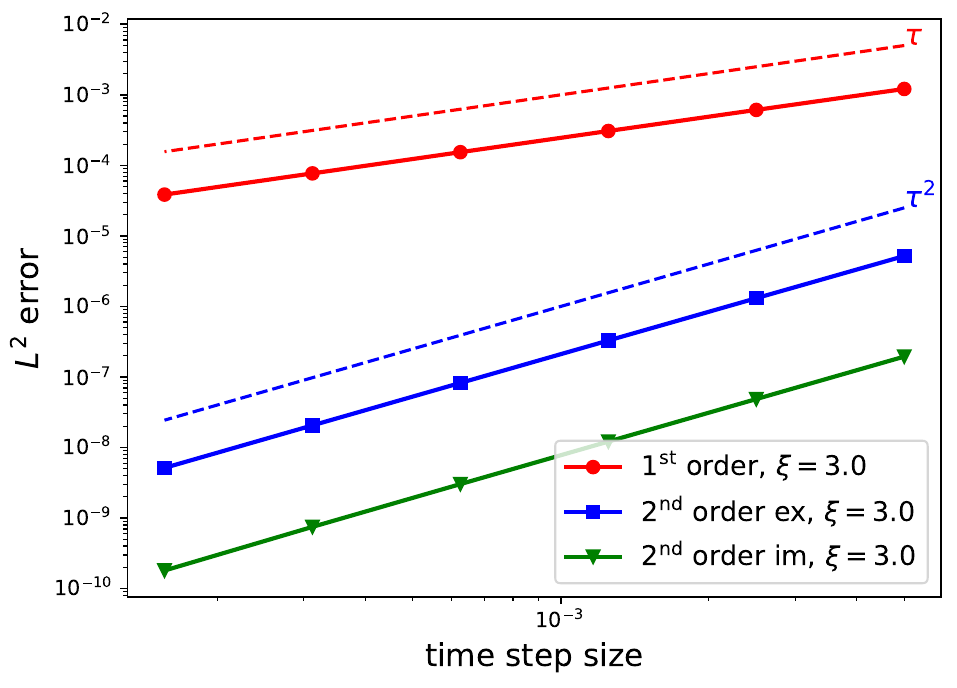}
    \end{subfigure}
    \caption{Convergence rates in time of the first- and second-order schemes in the settings of Example~1 for logarithmic potential with $\thetalog=0.2$ (first column) and regular potential (second column).}
    \label{convergenceRates2}
\end{figure}

\subsection{Sharp interfaces}
Next, we investigate the sharpness of the interfaces of the the solution of~\eqref{eq:general} simulated in one- and two-dimensional domain for different values of $\varepsilon$ and $\delta$.  We recall that sharpness of interface is characterized by the values of $\xi$, with perfectly sharp interfaces attained at a steady-state for $\xi = 0$ and the choice of the obstacle potential.

\begin{example}\label{allen_cahn_1}
    We set $\Omega = (-1,1)$, $\cpot= 1$, $N=1024$, and $\Delta t = 0.005$. We use the second-order explicit scheme to simulate the solution of~\eqref{eq:general} until the solution reaches a steady-state with the initial condition given by
$$u_0(x) = 0.5(\sin(\pi x) + \sin(2\pi x)).$$
\end{example}

The snapshots of the solutions of the model with the obstacle, regular, and logarithmic potentials at different time instances ($t=10,\ 20$, and $250$) are depicted in Figure~\ref{allen_cahn1}. The solutions are plotted for varying nonlocal interaction parameters $\delta = 0.1, 0.16,\ 0.1999$ and fixed $\varepsilon=0.1$, which corresponds to $\xi = 3.0,\ 0.5625,\ 0.001$ respectively. The figures in the first and second rows correspond to the solutions of the model with obstacle and regular potentials, whereas the third and fourth rows represent solutions using logarithmic potential with $\thetalog=0.01$ and $\thetalog=0.2$, respectively.

Moreover, in Figure~\ref{fig:AC_var_eps}, we plot a time evolution of the solution of the model~\eqref{eq:general} with obstacle potential for different values of $\varepsilon = 0.2, 0.15,$ and $0.1001$ with fixed $\delta = 0.2$, which correspond to $\xi=3.0,\ 1.25,\ 0.002$, respectively. 
We observe that the solution develops discontinuity in the interface profile for $\xi\approx 0$ as it approaches the steady-state. This is in agreement with the previously reported results for the obstacle~\cite{burkovska2023} and regular~\cite{DuYang2016} potentials. 
When using the obstacle potential and $\xi=0$, the solution exhibits sharp interfaces with only pure phases (Figure~\ref{allen_cahn1} and Figure~\ref{fig:AC_var_eps}). However, with the regular potential, although a discontinuity develops, it is not a jump discontinuity (the magnitude of the discontinuity being smaller than 2), and a small layer of the solution at the transient phase remains.
For the Flory-Huggins potential, the sharpness of the interface is influenced by the parameter $\thetalog$, with the sharpest interface achieved for a smaller value of $\thetalog$, since for $\thetalog\to 0$ the logarithmic potential approaches the obstacle potential. 
\begin{figure}[ht!]
    \centering
    \begin{subfigure}[h]{0.24\textwidth}
        \centering        \includegraphics[width=\textwidth]{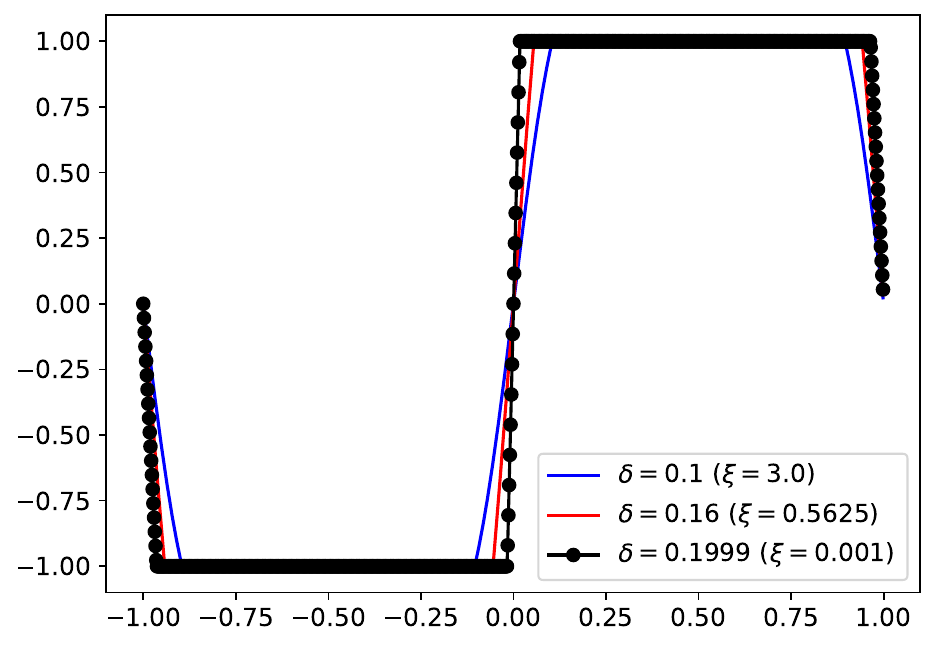}
    \end{subfigure}
    \begin{subfigure}[h]{0.24\textwidth}
        \centering        \includegraphics[width=\textwidth]{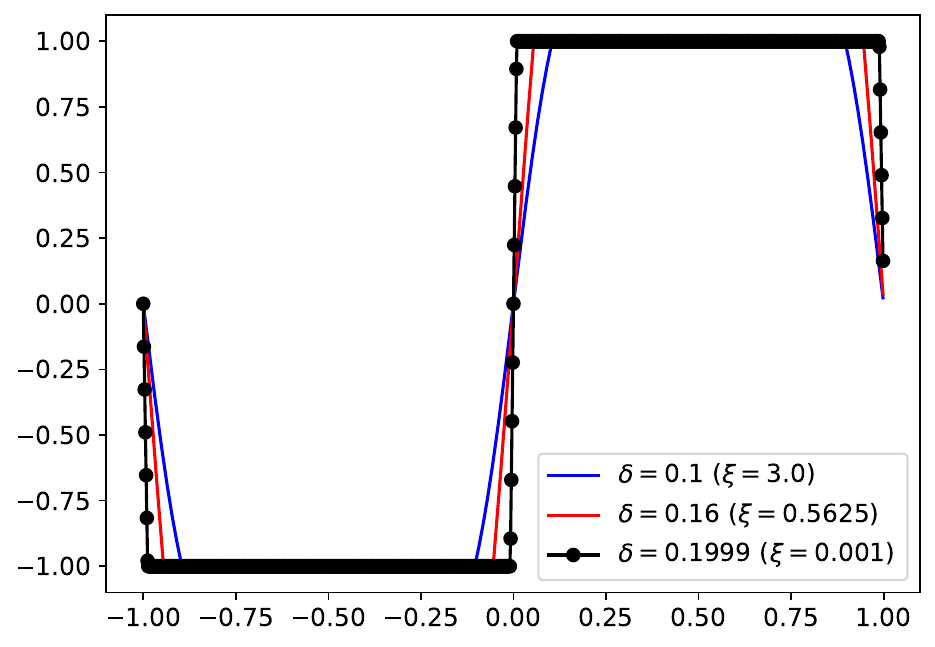}
    \end{subfigure}
    \begin{subfigure}[h]{0.24\textwidth}
        \centering        \includegraphics[width=\textwidth]{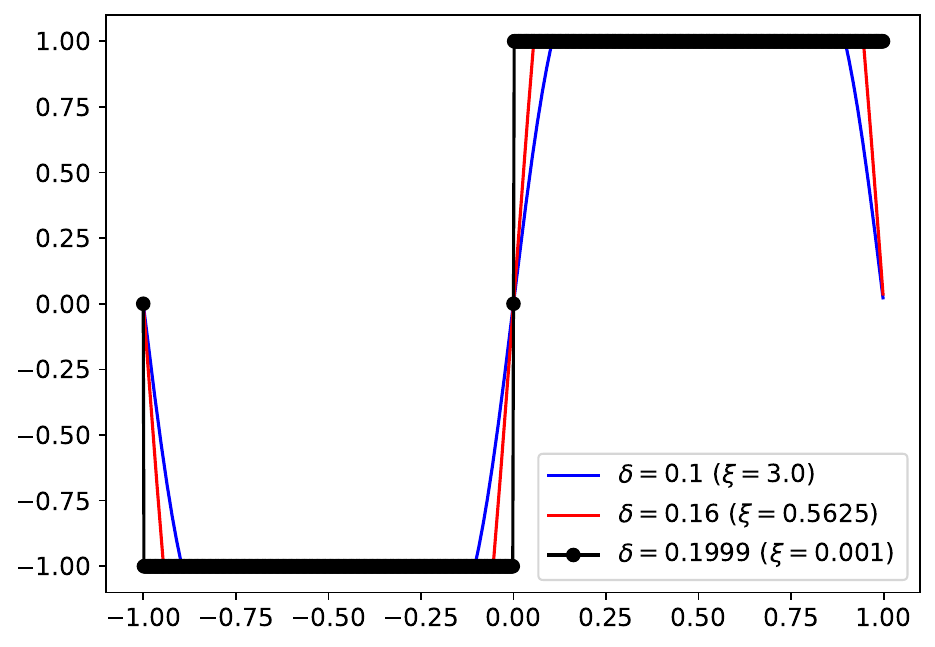}
    \end{subfigure}
    \begin{subfigure}[h]{0.24\textwidth}
        \centering        \includegraphics[width=\textwidth]{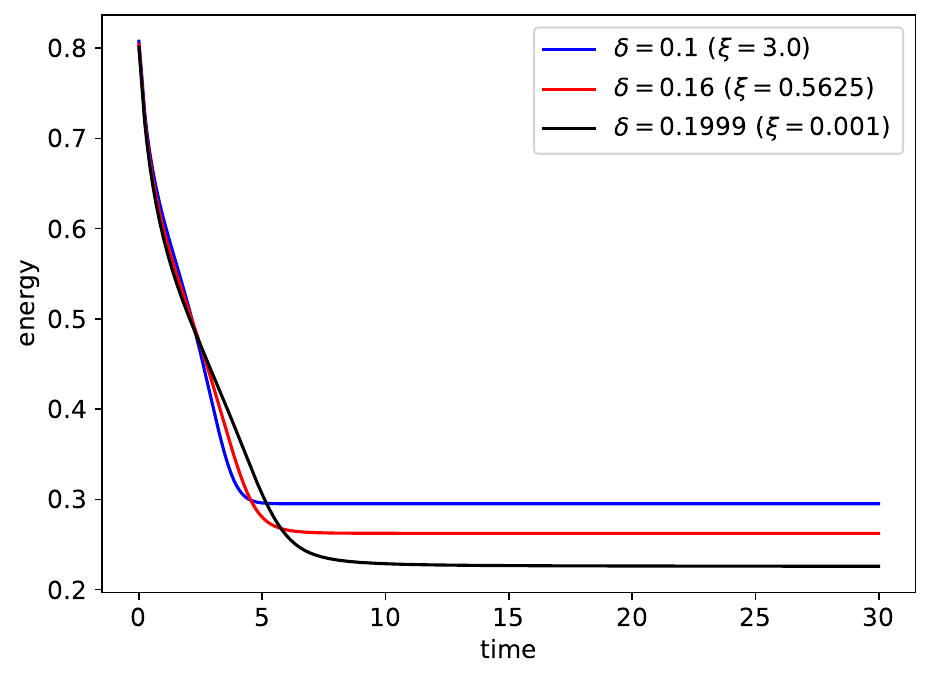}
    \end{subfigure}
    \begin{subfigure}[h]{0.24\textwidth}
        \centering        \includegraphics[width=\textwidth]{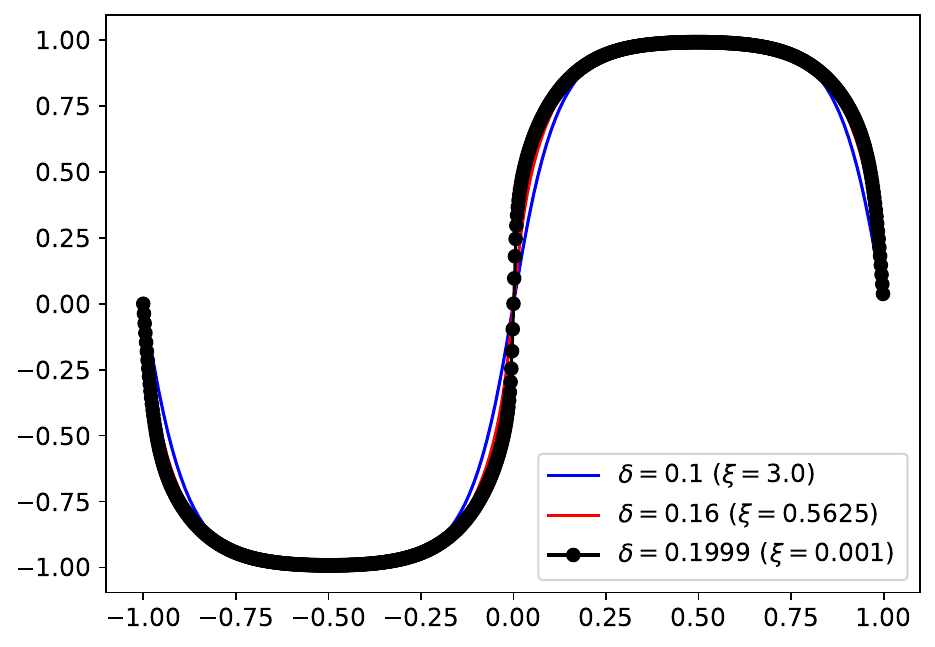}
    \end{subfigure}
    \begin{subfigure}[h]{0.24\textwidth}
        \centering        \includegraphics[width=\textwidth]{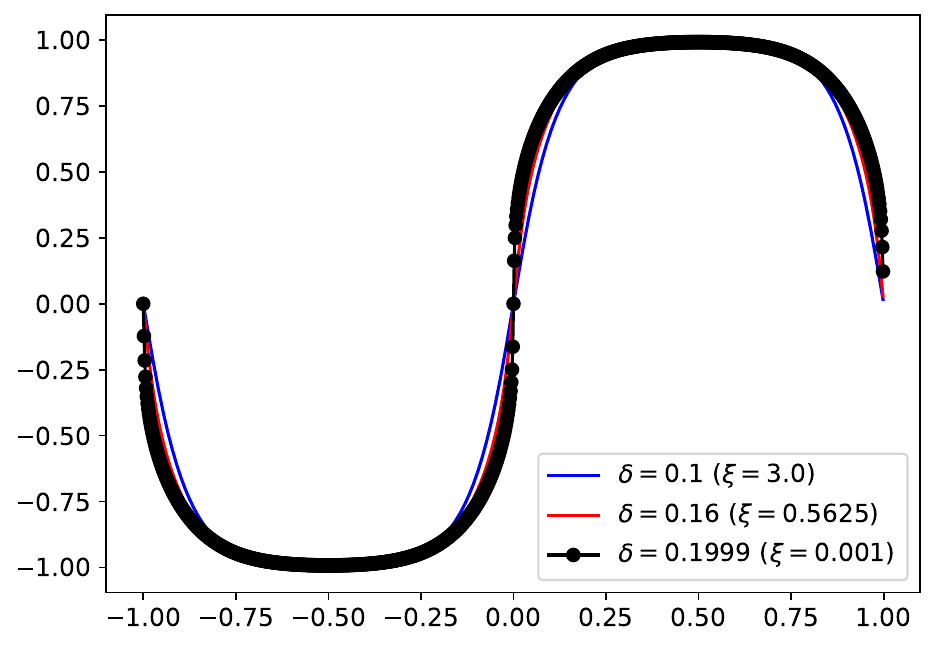}
    \end{subfigure}
    \begin{subfigure}[h]{0.24\textwidth}
        \centering        \includegraphics[width=\textwidth]{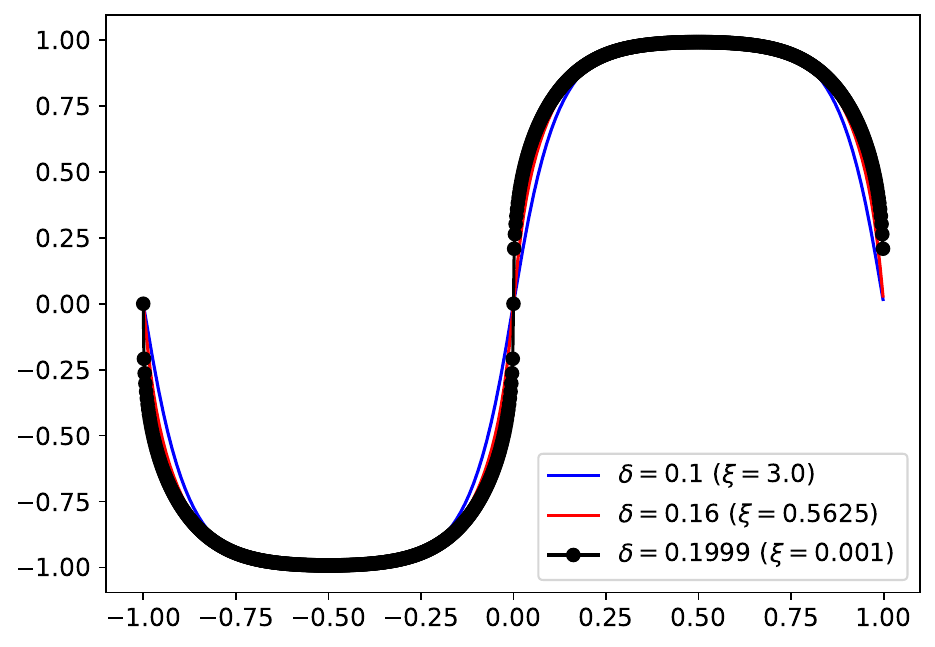}
    \end{subfigure}
    \begin{subfigure}[h]{0.24\textwidth}
        \centering        \includegraphics[width=\textwidth]{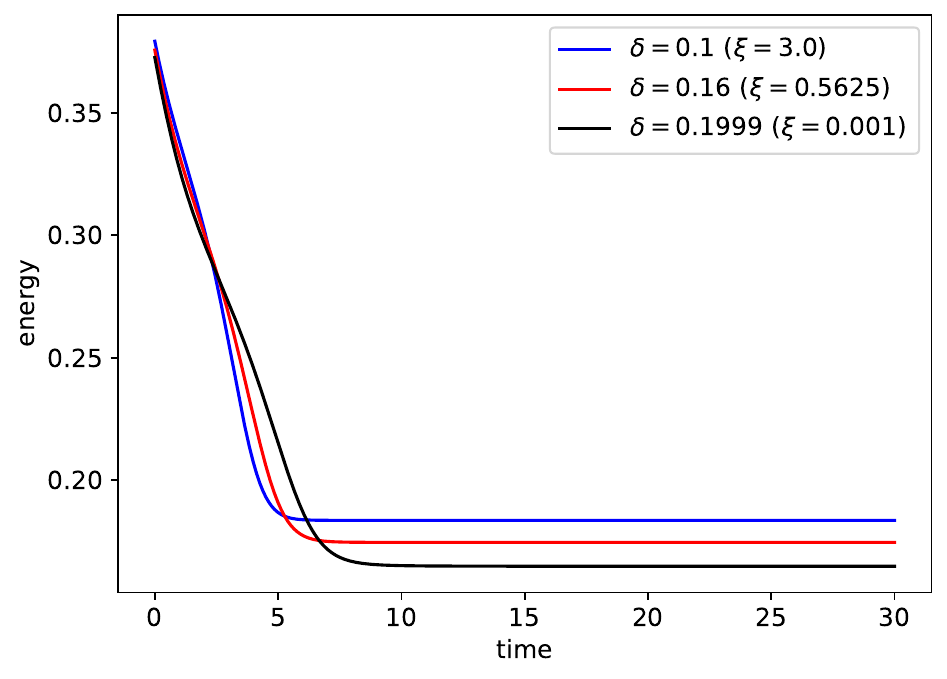}
    \end{subfigure}
    \begin{subfigure}[h]{0.24\textwidth}
        \centering        \includegraphics[width=\textwidth]{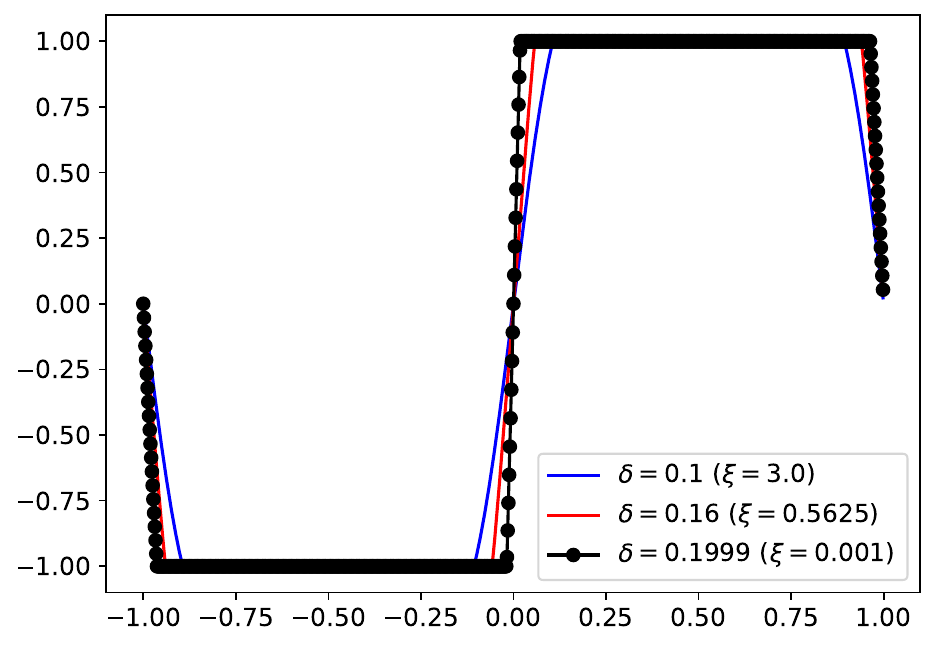}
    \end{subfigure}
    \begin{subfigure}[h]{0.24\textwidth}
        \centering        \includegraphics[width=\textwidth]{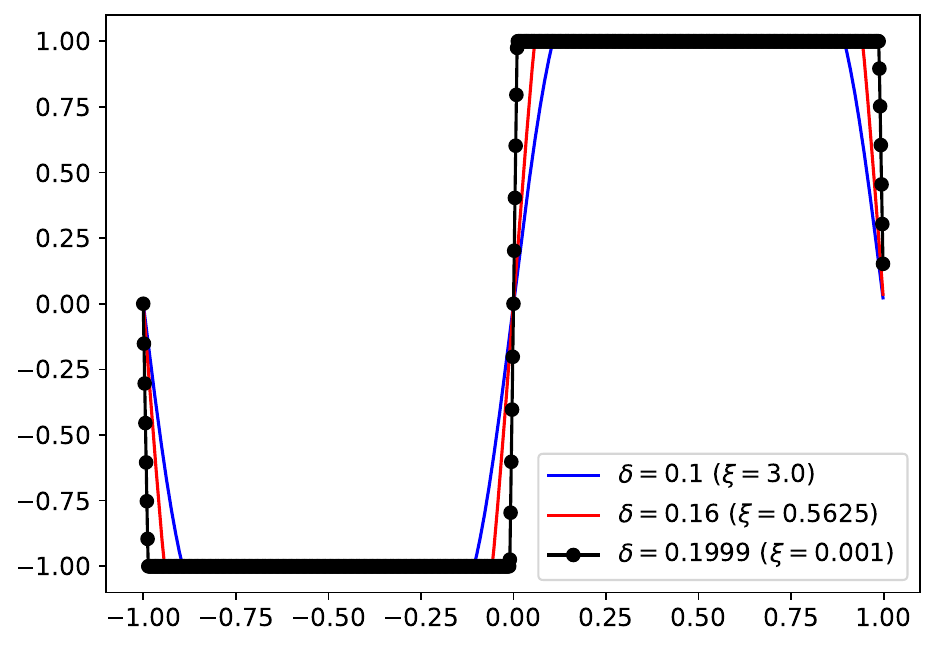}
    \end{subfigure}
    \begin{subfigure}[h]{0.24\textwidth}
        \centering        \includegraphics[width=\textwidth]{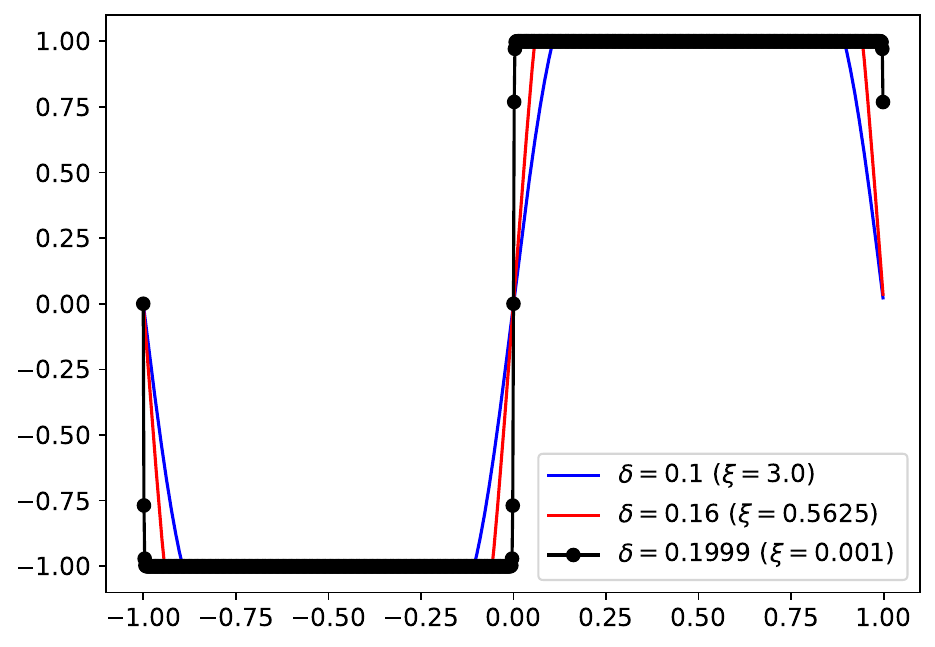}
    \end{subfigure}
    \begin{subfigure}[h]{0.24\textwidth}
        \centering        \includegraphics[width=\textwidth]{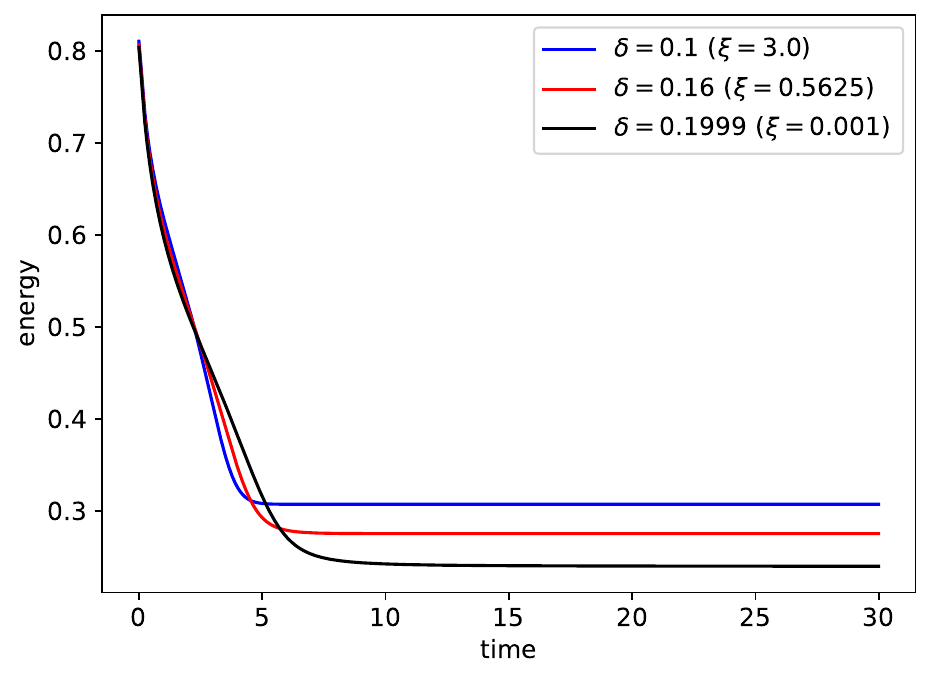}
    \end{subfigure}
    \begin{subfigure}[h]{0.24\textwidth}
        \centering        \includegraphics[width=\textwidth]{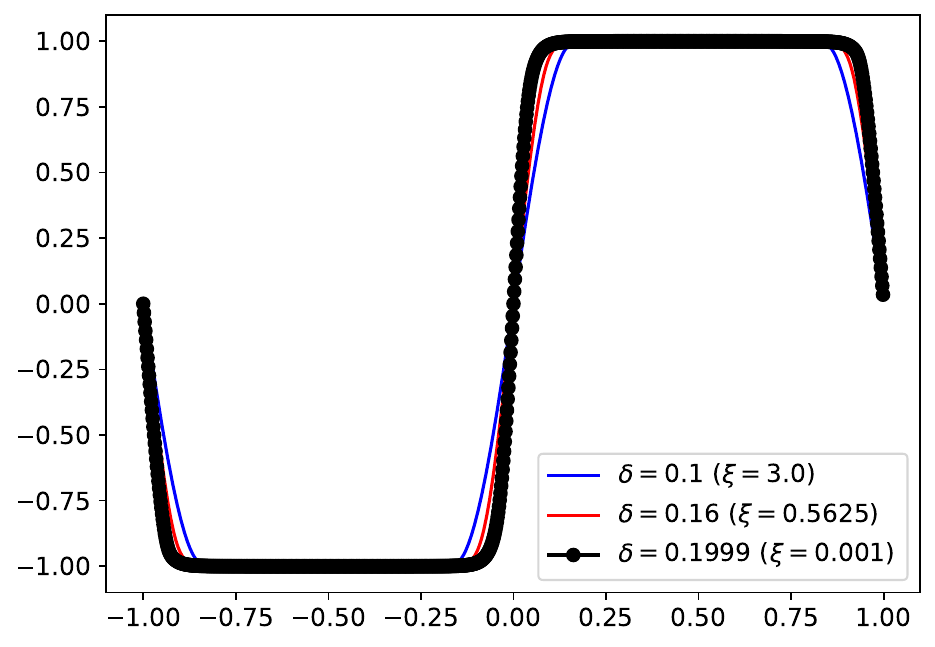}
    \end{subfigure}
    \begin{subfigure}[h]{0.24\textwidth}
        \centering        \includegraphics[width=\textwidth]{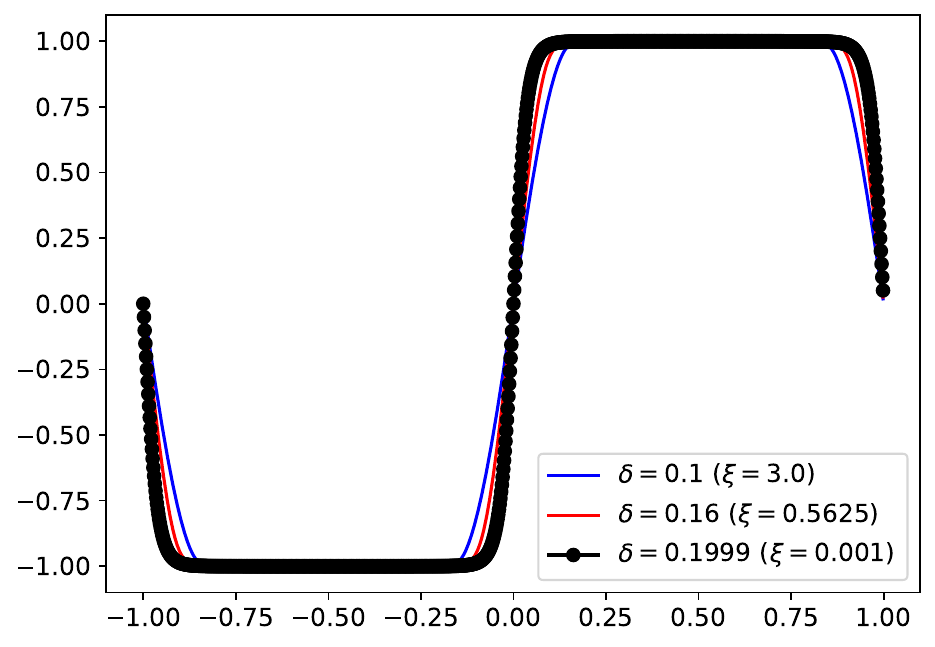}
    \end{subfigure}
    \begin{subfigure}[h]{0.24\textwidth}
        \centering        \includegraphics[width=\textwidth]{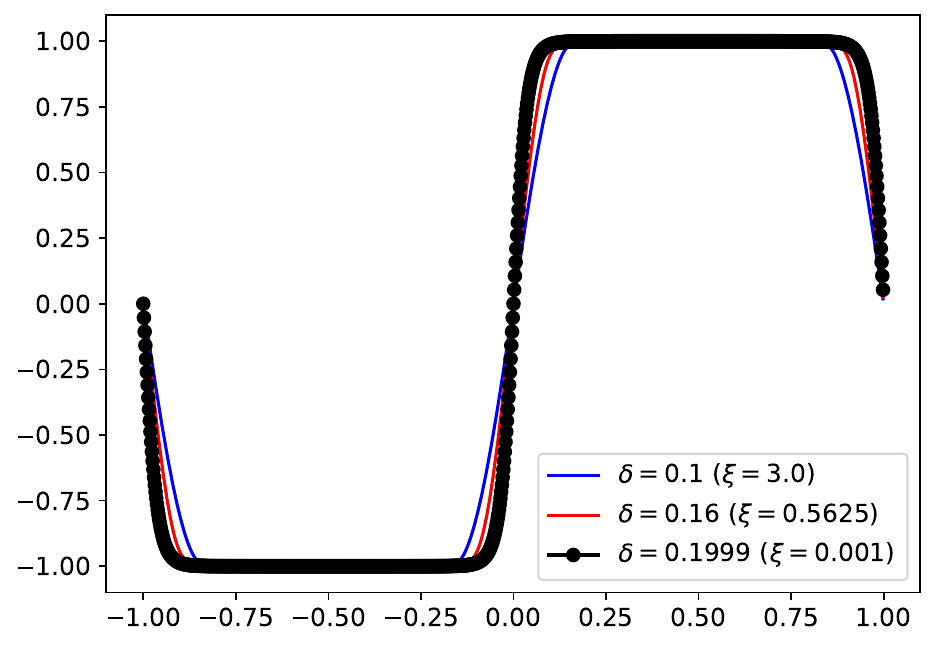}
    \end{subfigure}
    \begin{subfigure}[h]{0.24\textwidth}
        \centering        \includegraphics[width=\textwidth]{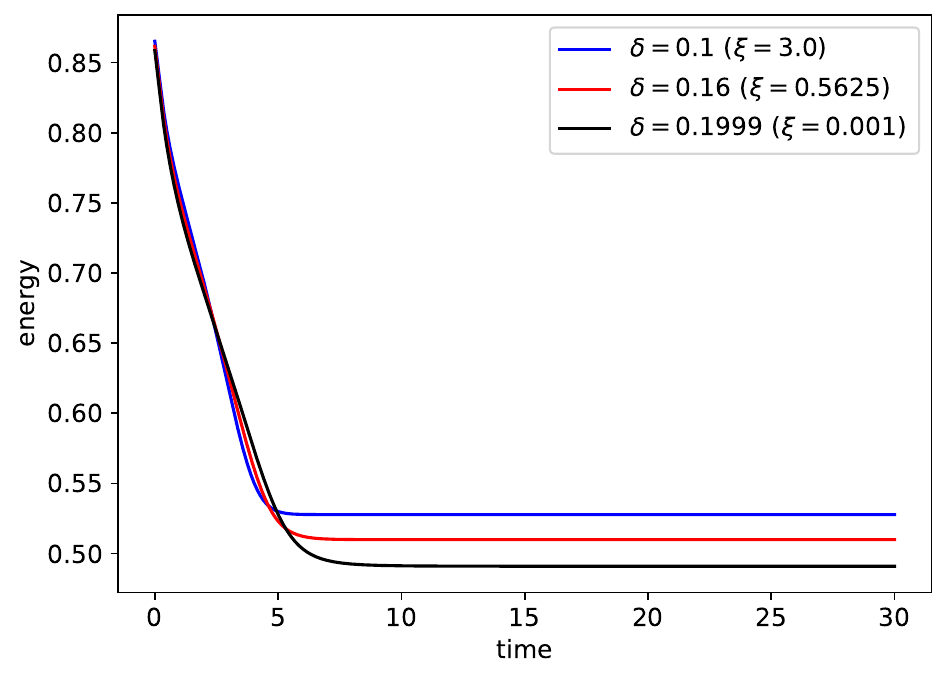}
    \end{subfigure}
    \caption{Time evolution of the solution of model~\eqref{eq:general} in the settings of Example~\ref{allen_cahn_1} with obstacle potential (first row), regular potential (second row), and logarithmic potential (with $\thetalog=0.01$ (third row) and $\thetalog=0.2$ (fourth row)); and the corresponding energies for fixed $\varepsilon=0.1$ and different values of $\delta=0.1,\ 0.16$, and $0.1999$ corresponding to $\xi=3.0,\ 0.5625$ and $\xi=0.001$, respectively. From left to right: $t=10$, $20$ and $250.$}
    \label{allen_cahn1}
\end{figure}

\begin{figure}[ht]
    \begin{subfigure}[h]{0.24\textwidth}
        \centering        
        \includegraphics[width=\textwidth]{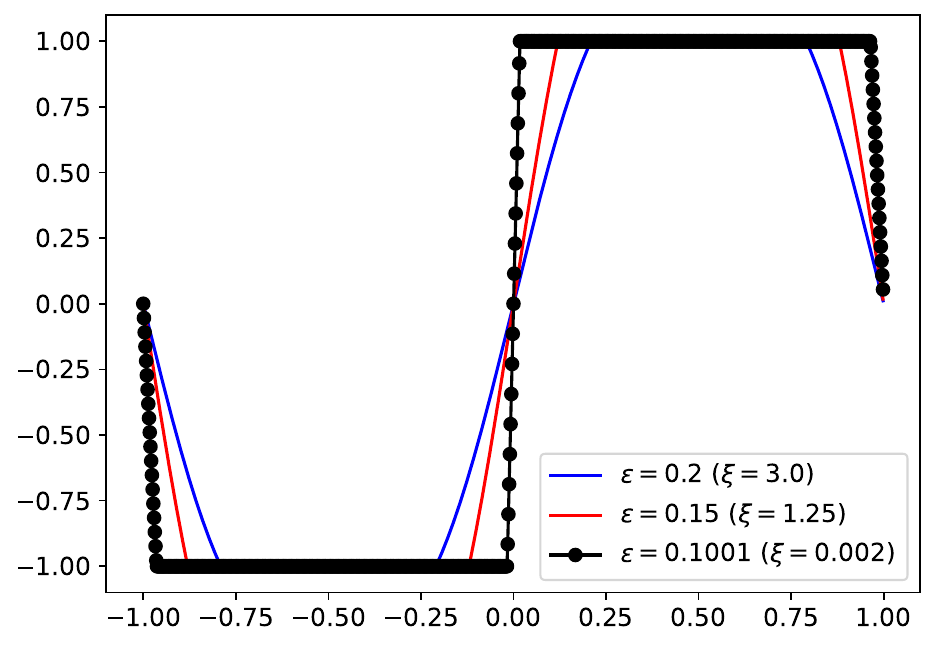}
    \end{subfigure}
    \begin{subfigure}[h]{0.24\textwidth}
        \centering        
        \includegraphics[width=\textwidth]{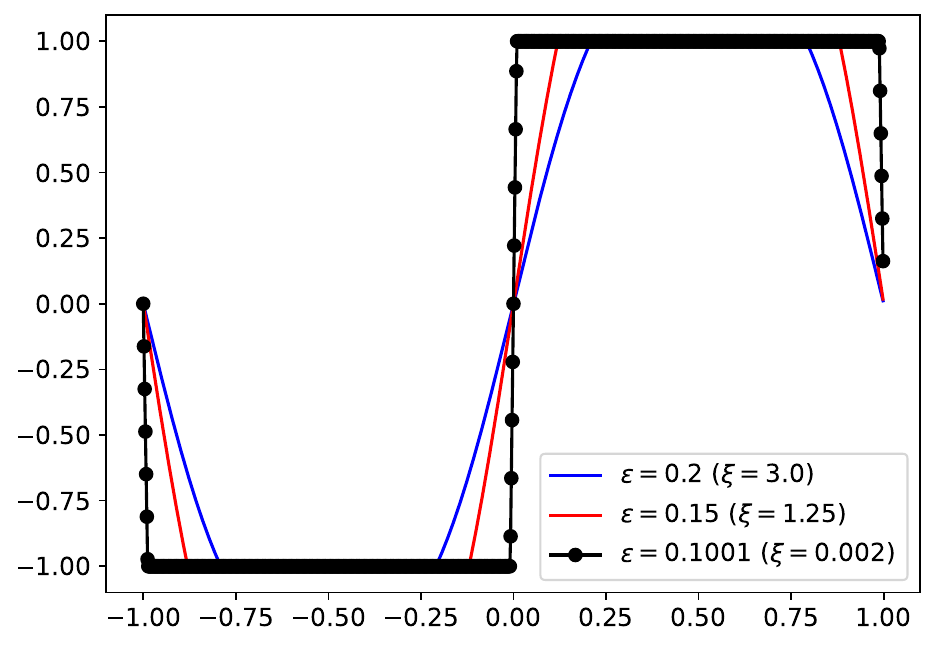}
    \end{subfigure}
    \begin{subfigure}[h]{0.24\textwidth}
        \centering        
        \includegraphics[width=\textwidth]{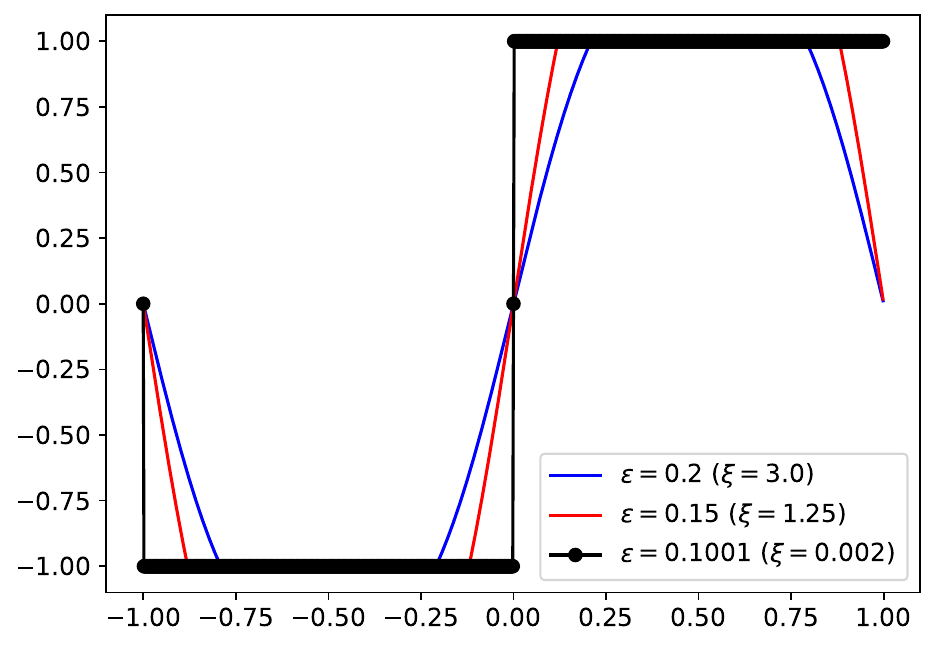}
    \end{subfigure}
    \begin{subfigure}[h]{0.24\textwidth}
        \centering        \includegraphics[width=\textwidth]{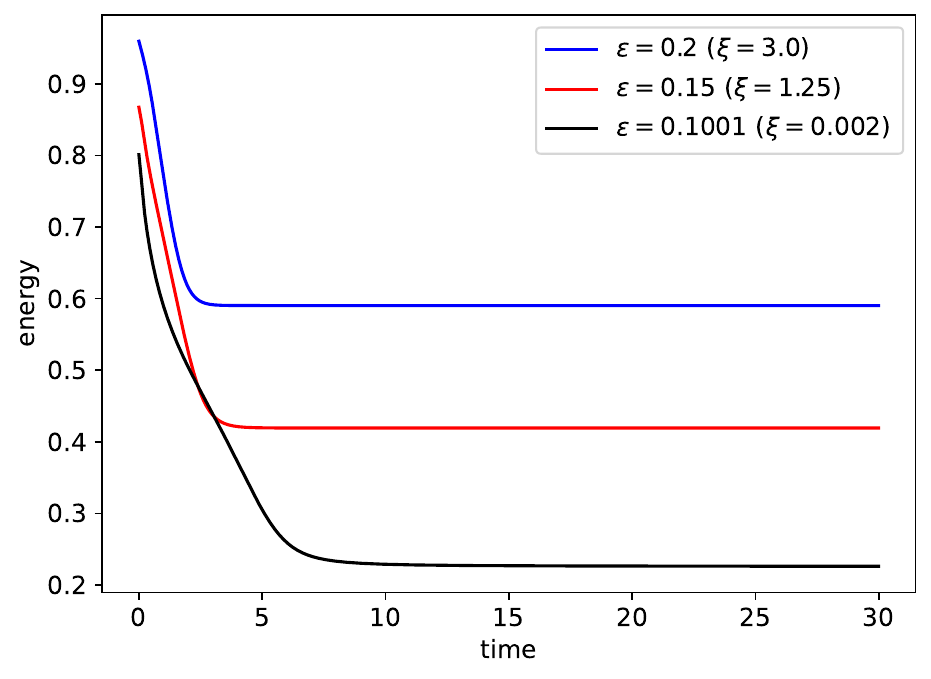}
    \end{subfigure}
    \caption{Time evolution of the solution of model~\eqref{eq:general} in the settings of Example~\ref{allen_cahn_1} with the obstacle potential; and the corresponding energy for fixed $\delta=0.2$ and different values of $\varepsilon=0.2,\ 0.15$, and $0.1001$ corresponding to $\xi=3.0,\ 1.25$ and $\xi=0.002$, respectively. From left to right: $t=10$, $20$ and $250.$}
    \label{fig:AC_var_eps}
\end{figure}

\begin{remark}
    We note that the solution at $x=1$ is excluded for all cases due to the discretization of the periodic boundary conditions.
\end{remark}

We also illustrate the sharpness of the interface in two-dimensional setting.
\begin{example}\label{sharp_2D}
    We set $\Omega = (-1,1)^2,$ $N = 512^2$, $\cpot= 1$, $\delta=0.0824$, $\varepsilon^2=0.0017$, corresponding to $\xi = 0.0015$, and $\Delta t = 0.005$. We consider an initial condition given by  $u_0(x,y)=\sin(\pi x)e^{-|y|}$ and adopt the second-order explicit scheme to simulate the phase separation until a steady-state.
\end{example}
    
Figure~\ref{sharp_interface_2D_2} shows the evolution of the phase separation at $t=1, 3$, and $50$ for obstacle potential (first row), regular potential (second row), and logarithmic potential (third row) with $\thetalog=0.1$. Additionally, we plot the interface width at $t=50$ (fourth column) for the three potentials with bounds $-1$ and $1$ for obstacle potential, $-1+10^{-6}$ and $1-10^{-6}$ for logarithmic potential, and bounds $-0.99$ and $0.99$ for a regular potential. We observe that, similarly as in 1D case, at a steady-state the sharpest interface appears in the solution with the obstacle potential, whereas the most diffuse interface happens for the case of a regular potential. 
\begin{figure}[ht!]
    \centering
    \begin{subfigure}[h]{0.24\textwidth}
        \centering       \includegraphics[width=\textwidth]{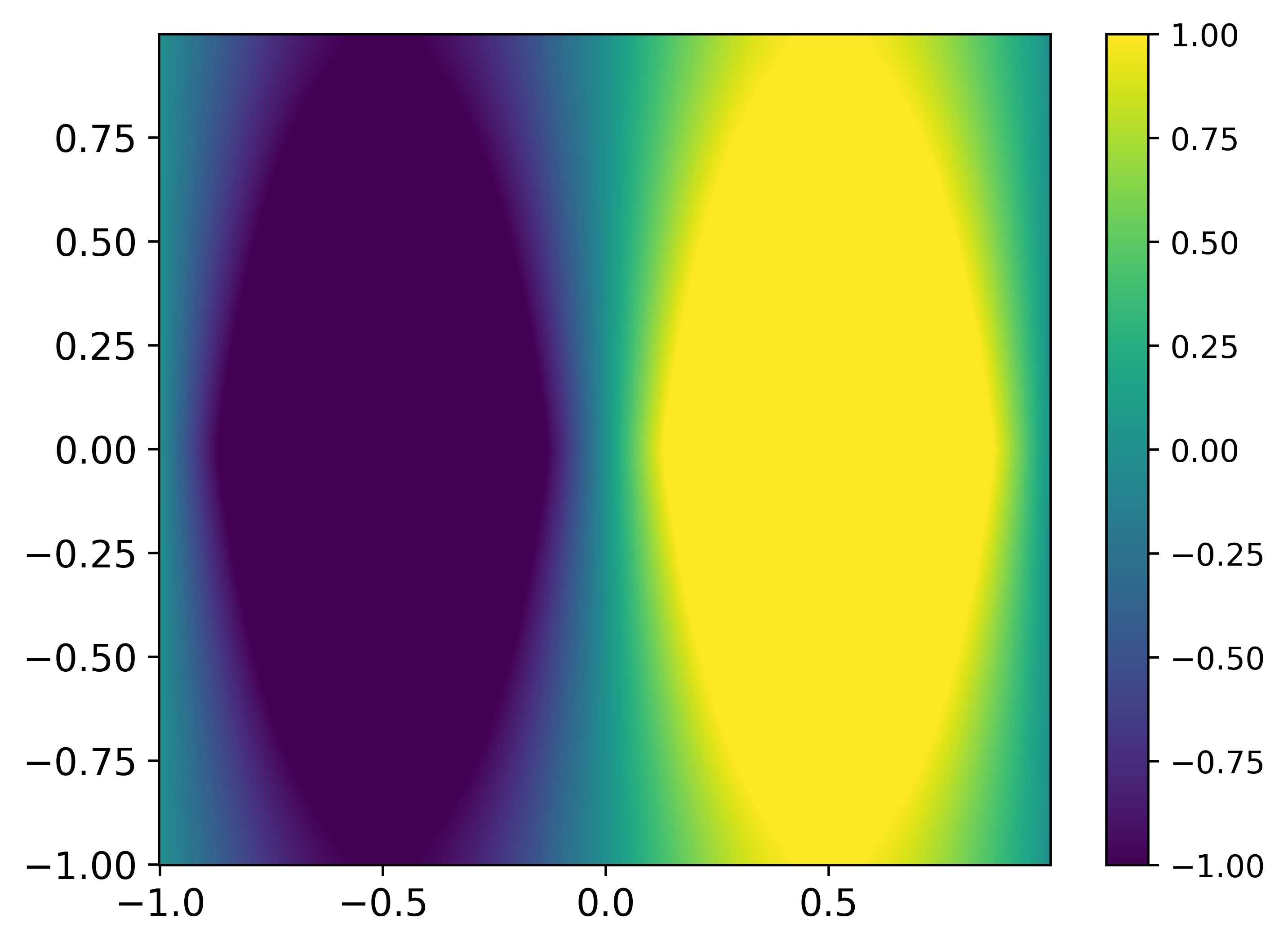}
    \end{subfigure}
    \centering
    \begin{subfigure}[h]{0.24\textwidth}
        \centering
        \includegraphics[width=\textwidth]{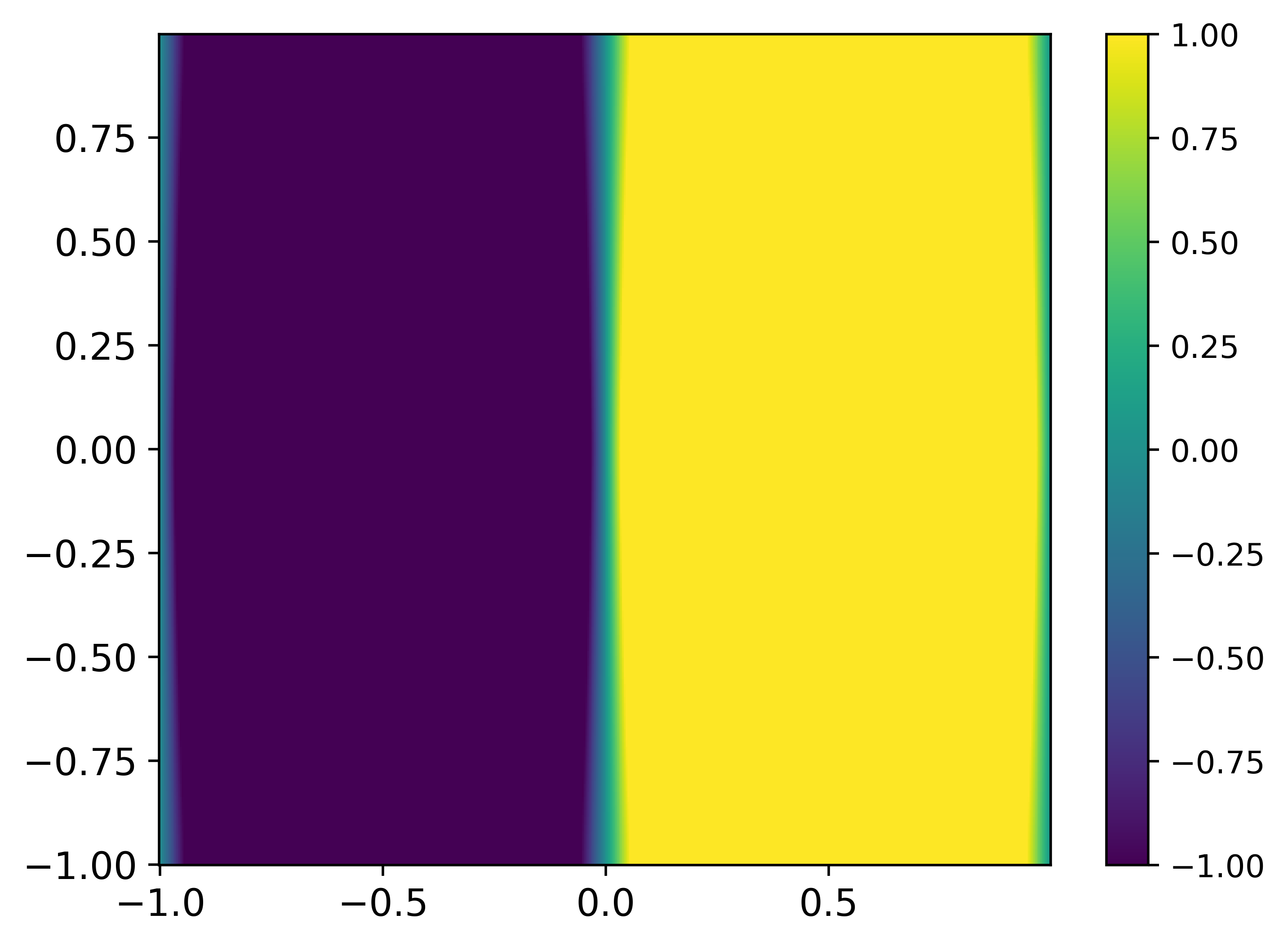}
    \end{subfigure}
    \begin{subfigure}[h]{0.24\textwidth}
        \centering
        \includegraphics[width=\textwidth]{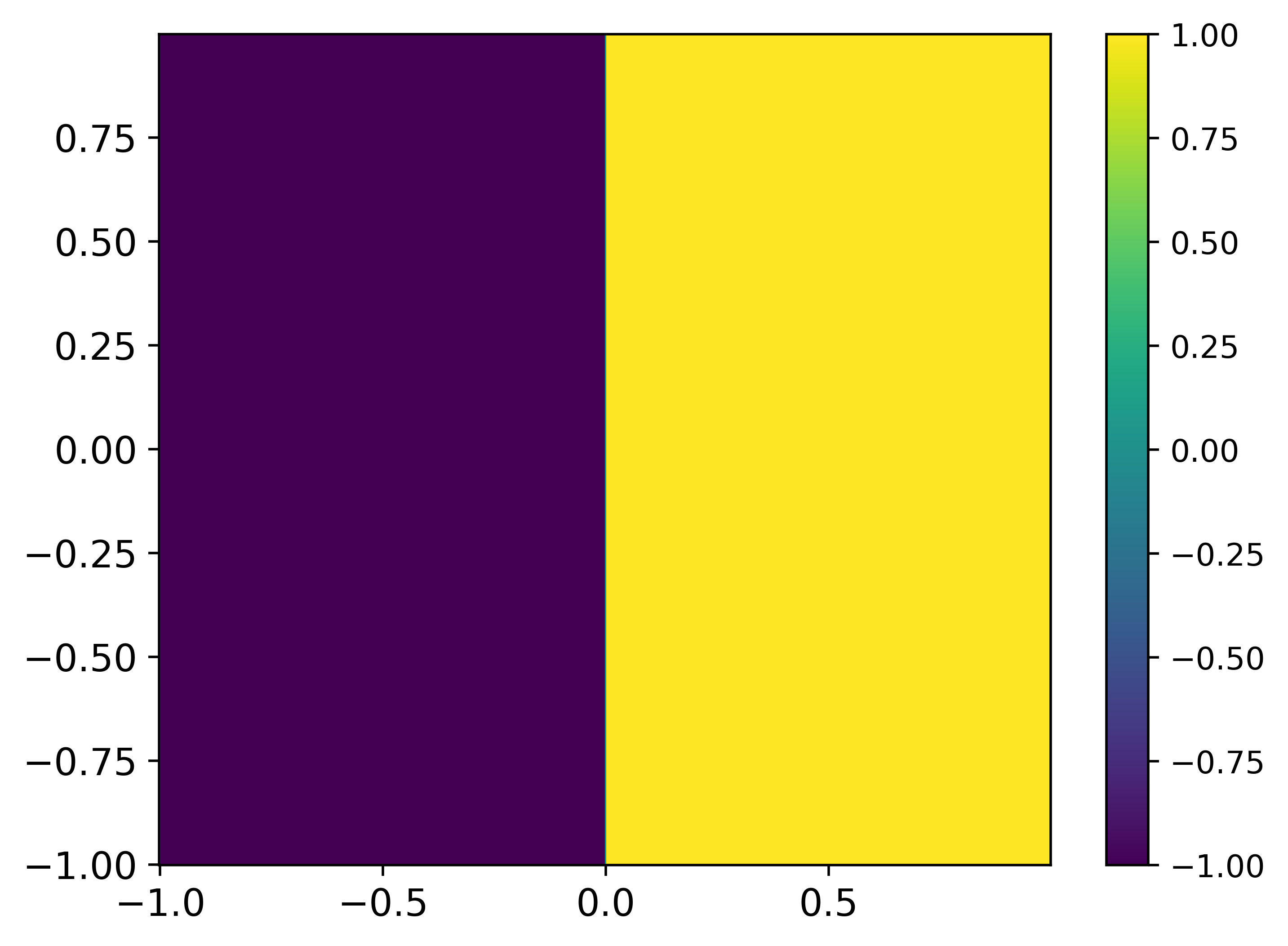}
    \end{subfigure}
    \begin{subfigure}[h]{0.24\textwidth}
        \centering
        \includegraphics[width=\textwidth]{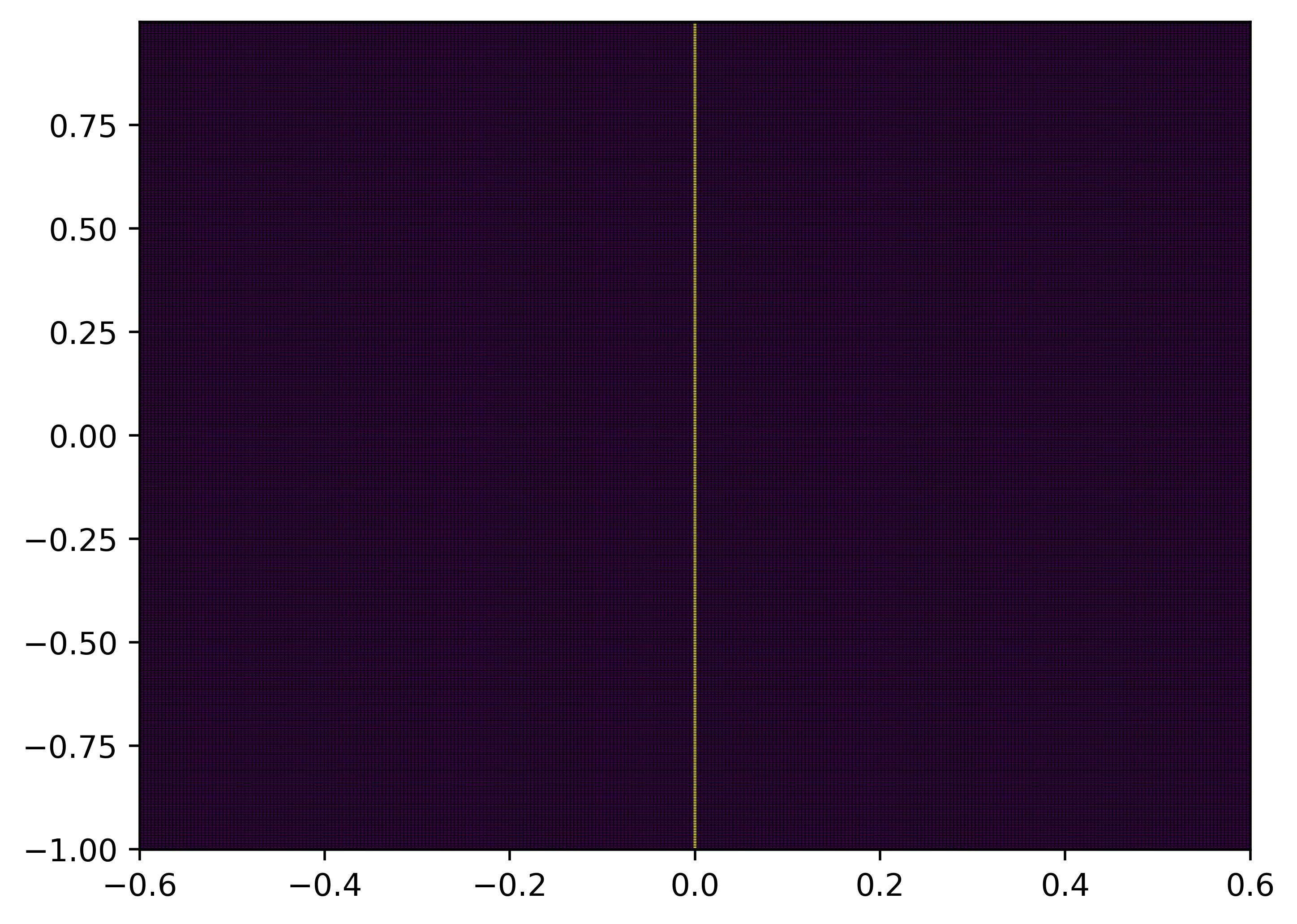}
    \end{subfigure}
    \begin{subfigure}[h]{0.24\textwidth}
        \centering
        \includegraphics[width=\textwidth]{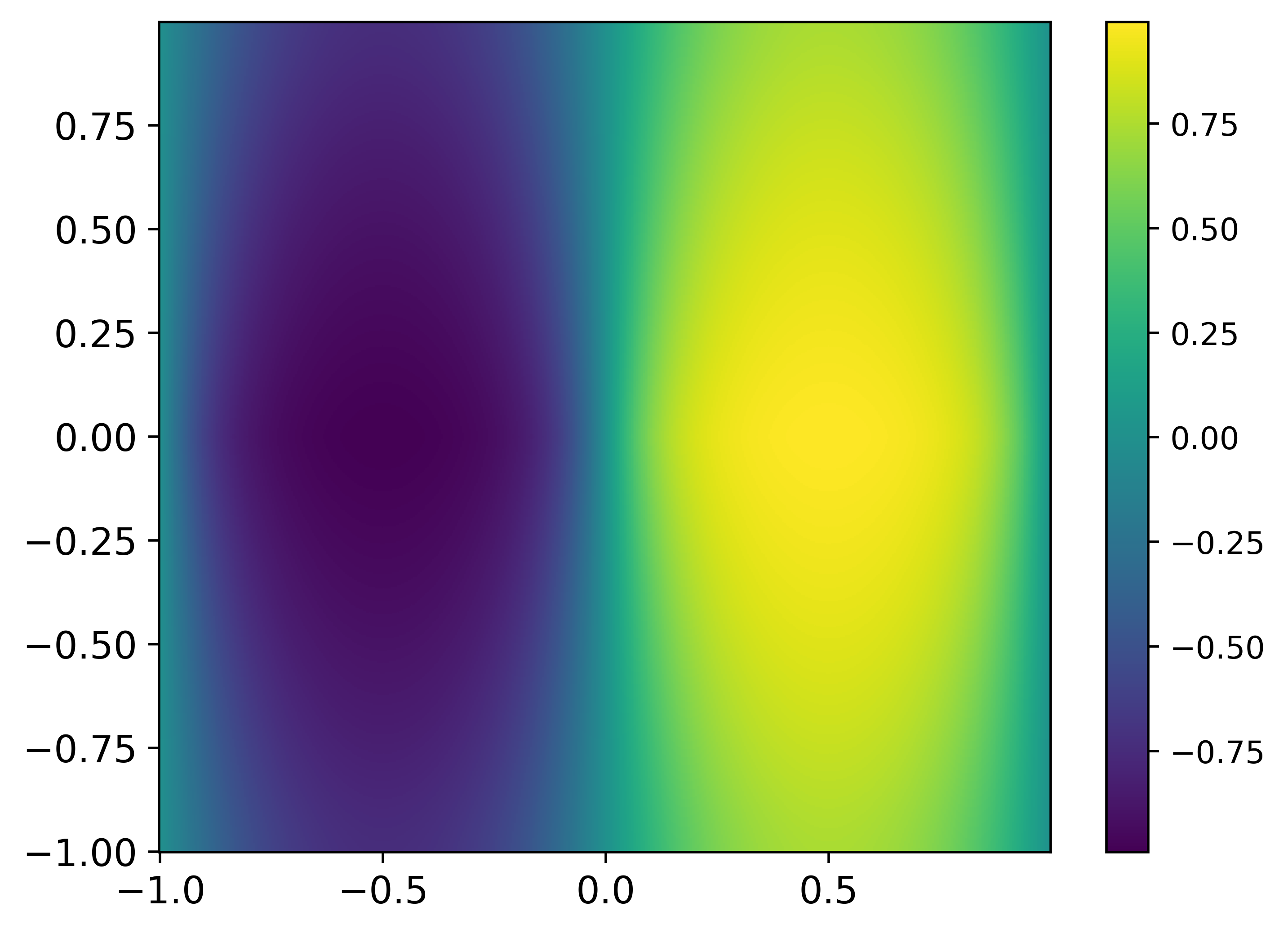}
    \end{subfigure}
    \begin{subfigure}[h]{0.24\textwidth}
        \centering
        \includegraphics[width=\textwidth]{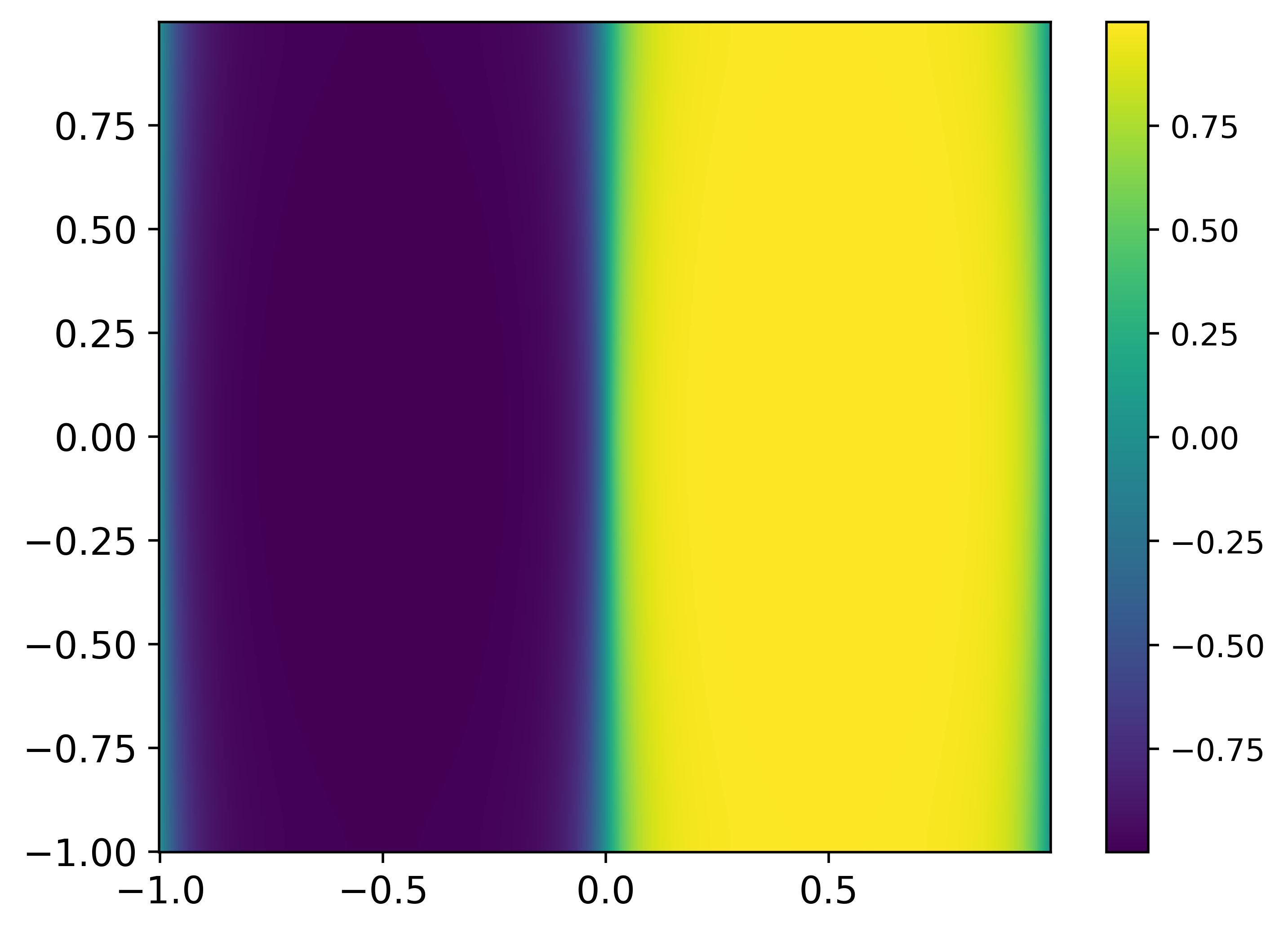}
    \end{subfigure}
    \begin{subfigure}[h]{0.24\textwidth}
        \centering
        \includegraphics[width=\textwidth]{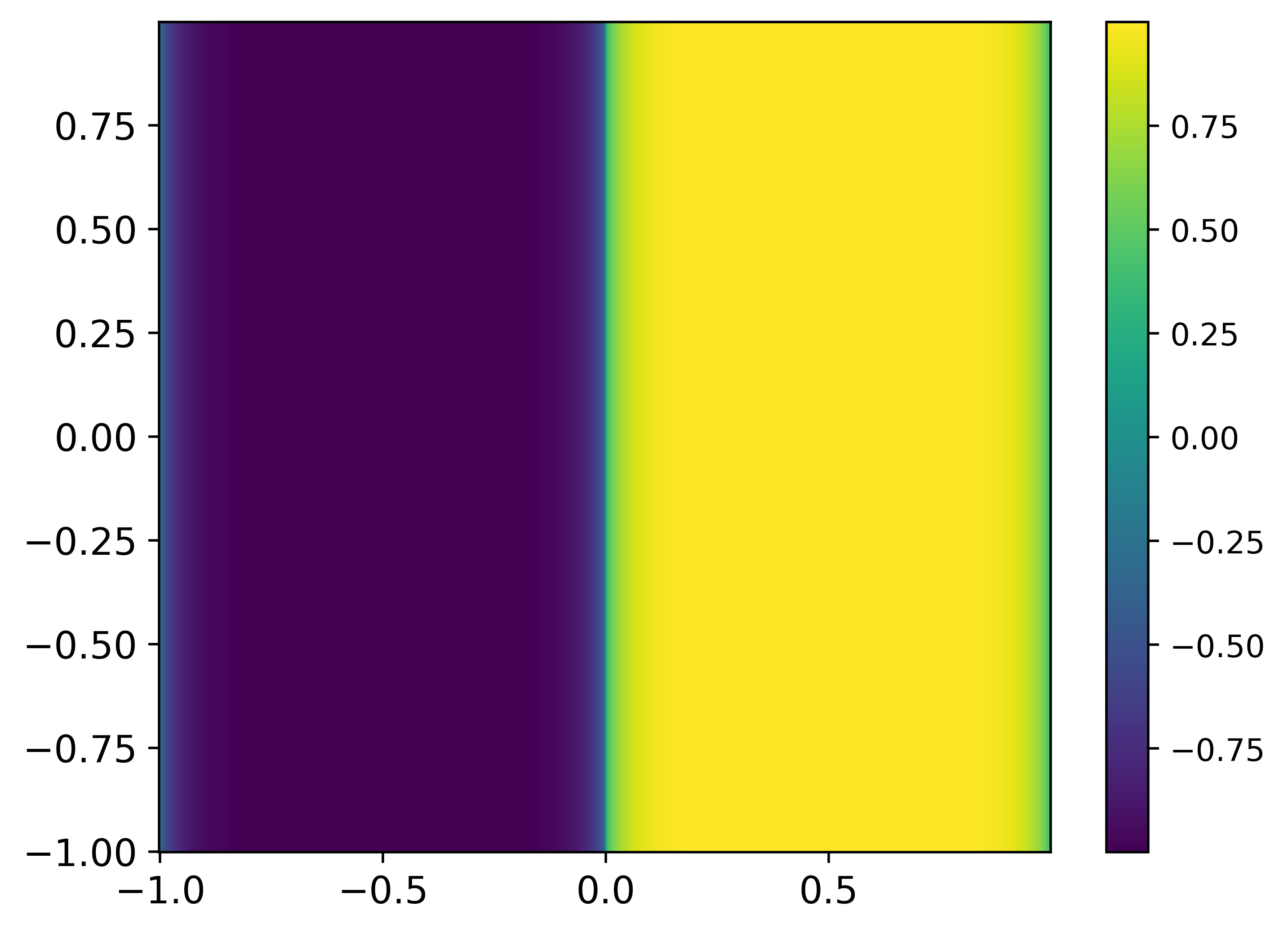}
    \end{subfigure}
    \begin{subfigure}[h]{0.24\textwidth}
        \centering
        \includegraphics[width=\textwidth]{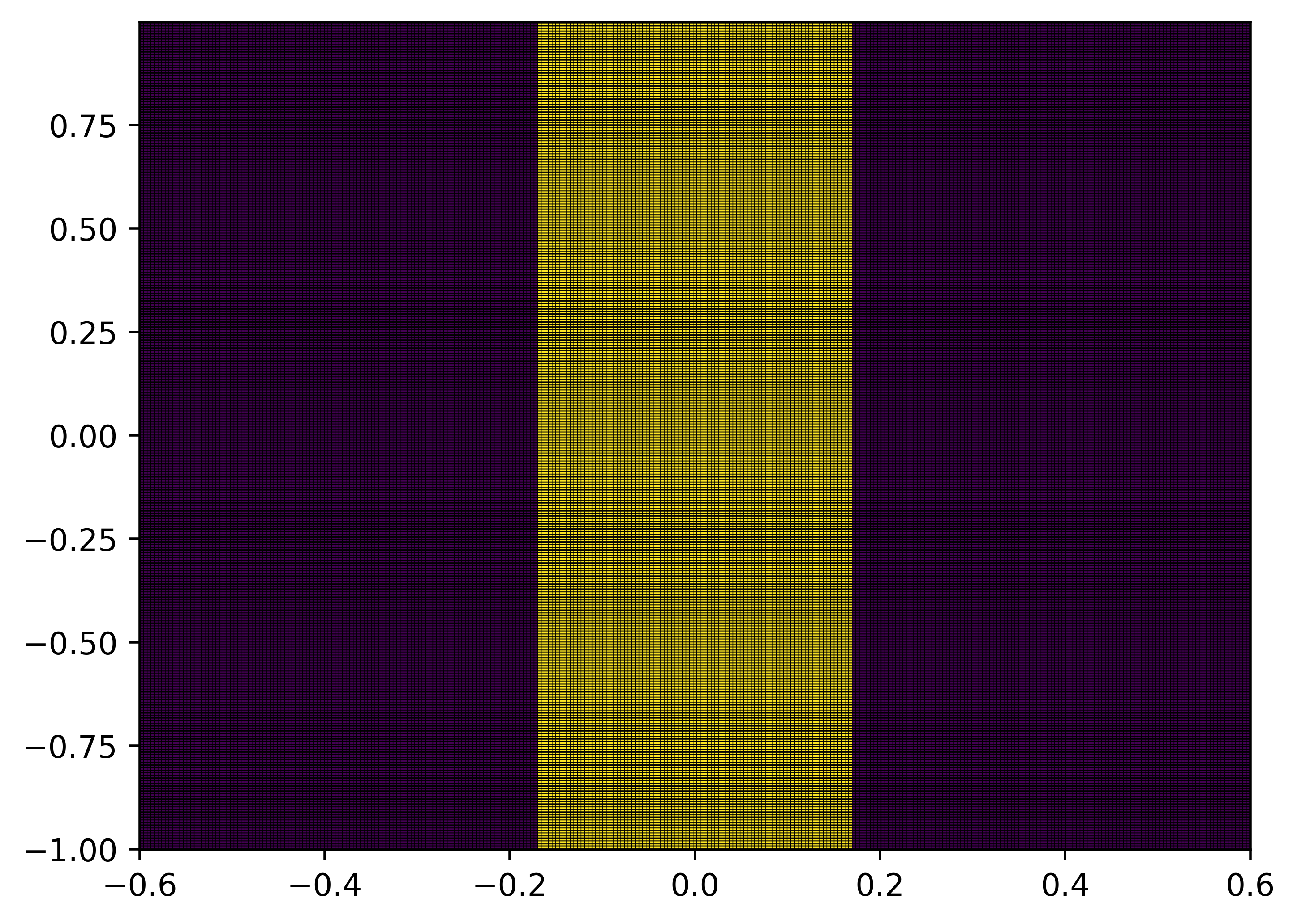}
    \end{subfigure}
    \begin{subfigure}[h]{0.24\textwidth}
        \centering
        \includegraphics[width=\textwidth]{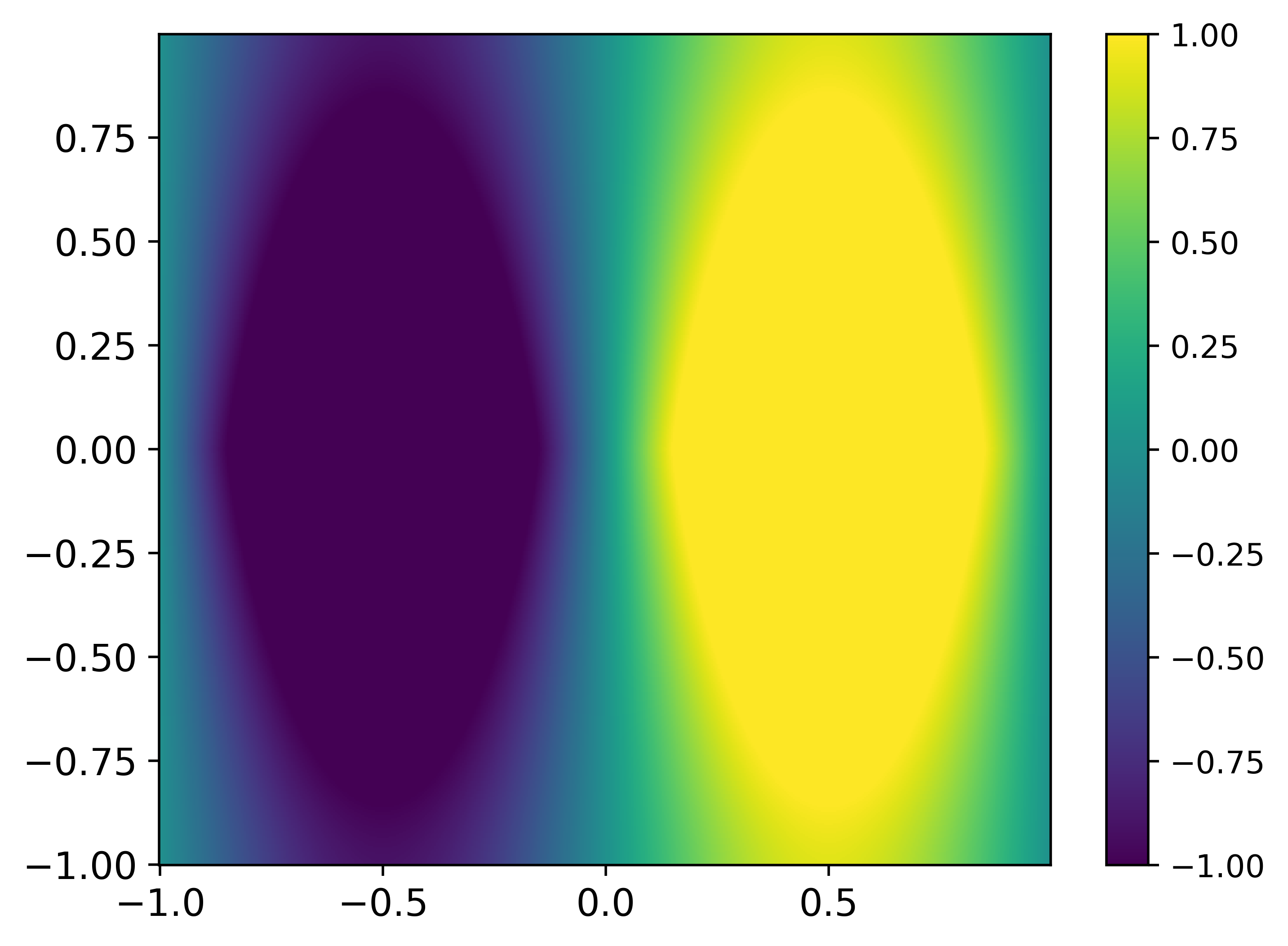}
    \end{subfigure}
    \begin{subfigure}[h]{0.24\textwidth}
        \centering
        \includegraphics[width=\textwidth]{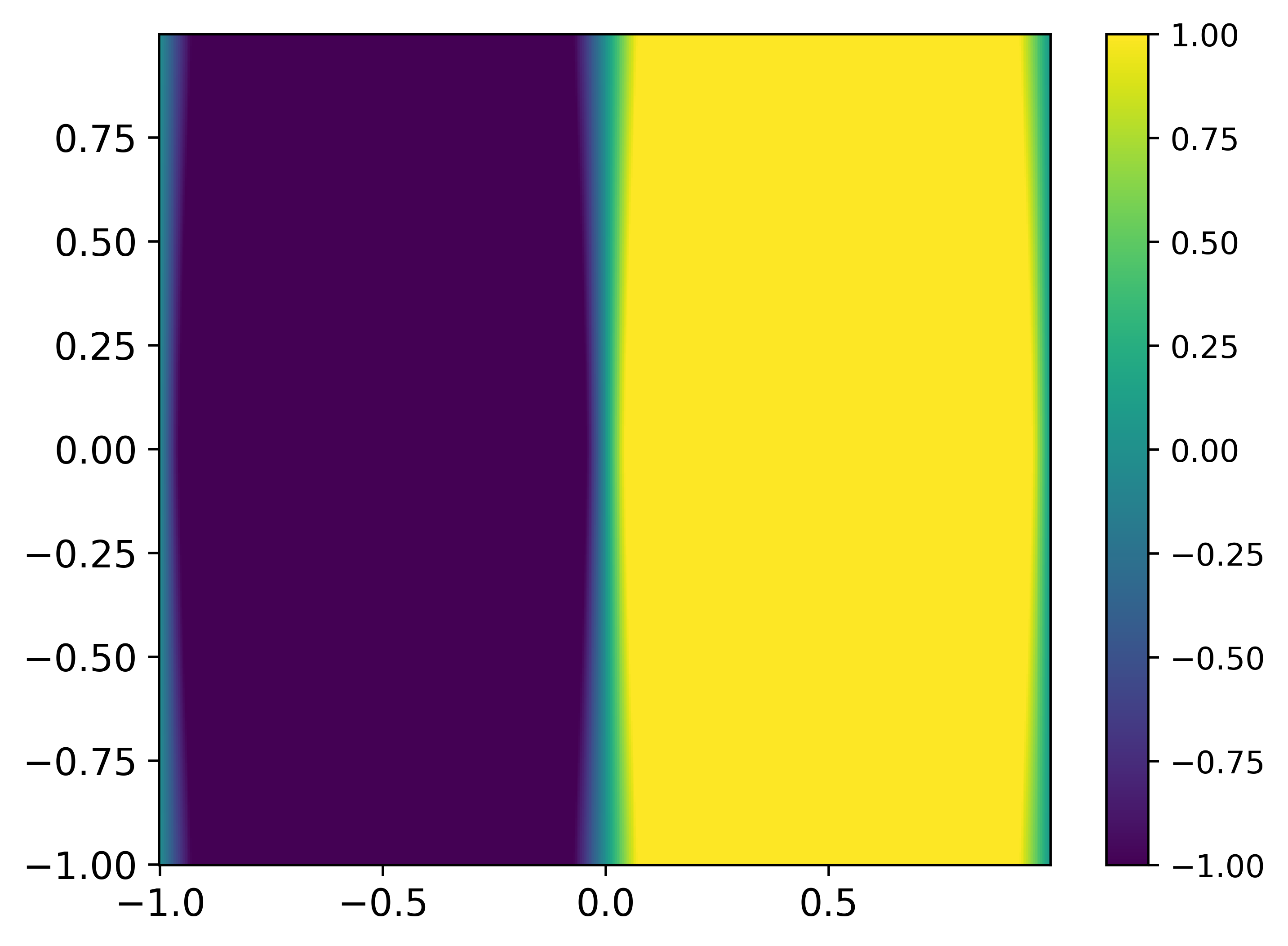}
    \end{subfigure}
    \begin{subfigure}[h]{0.24\textwidth}
        \centering
        \includegraphics[width=\textwidth]{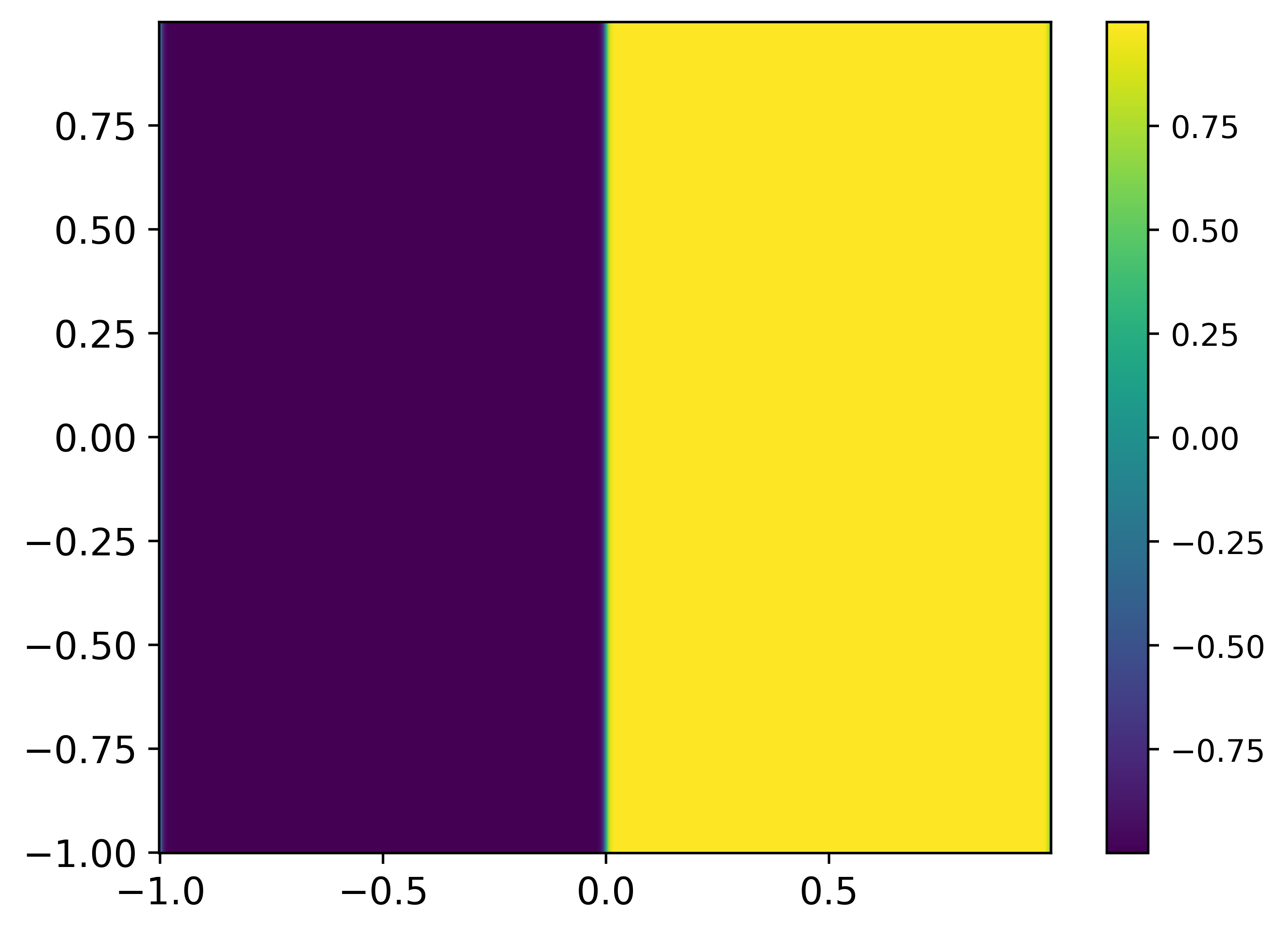}
    \end{subfigure}
    \begin{subfigure}[h]{0.24\textwidth}
        \centering
        \includegraphics[width=\textwidth]{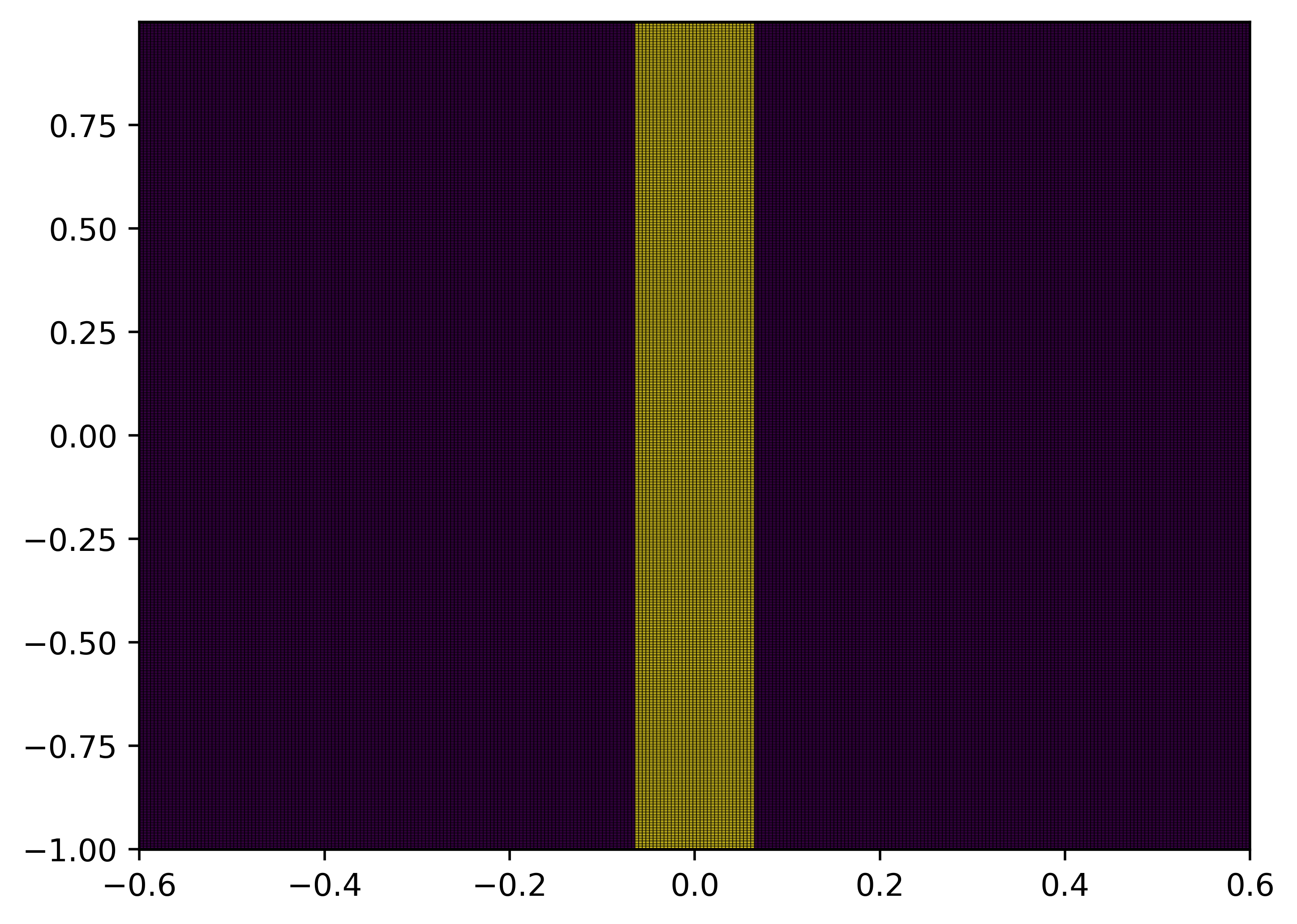}
    \end{subfigure}
    \caption{Evolution of the solution governed by the nonlocal Allen-Cahn equation~\eqref{eq:general} with an obstacle (first row), regular (second row), and Flory-Huggins (third row) potentials in the settings of Example~\ref{sharp_2D} at different time instances. From left to right: $t=1, 3, 50$. The fourth column represents the width of the interface of the corresponding solution at $t=50$ (zoomed-in).}
    \label{sharp_interface_2D_2}
\end{figure}

\subsection{Computational costs}
Next, we discuss the computational cost, which is presented in terms of the number of convolution evaluations for each time-stepping scheme. Additionally, we give the $L^2$-errors relative to each computational cost for comprehensive assessment. To compute the errors at a final time $T$, we perform the numerical simulations with $\Delta t = (T/4) \times 2^{-k},~~k = 0, 1, 2, \dots, 13$ and we set the solution obtained by the second order implicit scheme with $\Delta t = (T/4)\times 2^{-15}$ and $TOL = 10^{-15}$ in Algorithm~\ref{alg:fixed_point_iteration} as a benchmark solution.

\begin{example}\label{comp_cost}
    We consider $\Omega = (-1,1)^n, n=1,2,3~\cpot= 2$, $\delta=\varepsilon=0.15$, $\thetalog=0.5$, and $T=2$. We set $N=1024$, $256^2$, and $80^3$ for the 1D, 2D, and 3D cases, respectively. The initial conditions are
    \begin{align*}
        u_0(x) &= 0.1(\sin(\pi x) + \sin(2\pi x))&\text{(1D)},\\
        u_0(x,y) &= 0.2\sin(\pi x)\sin(\pi y)&\text{(2D)},\\
        u_0(x,y,z) &= 0.2\sin(\pi x)\sin(\pi y)\sin(\pi z)&\text{(3D)}.
    \end{align*}
\end{example}

Figures~\ref{comput_costs} and \ref{comput_costs2} display $L^2$-error at the final time with respect to different number of convolution evaluations for obstacle, regular, and logarithmic potentials. We simulate~\eqref{eq:general}  for obstacle potential with $n=1,2,3$, and regular and logarithmic potentials for $n=2$. We find that the second-order implicit scheme is the most accurate overall, although it requires a larger number of convolution evaluations, as anticipated. Additionally, in scenarios where high precision isn't essential and a quick approximation is desired, the second-order explicit scheme offers a suitable compromise.
\begin{figure}[ht!]
    \centering
    \begin{subfigure}[h]{0.49\textwidth}
        \centering
        \includegraphics[width=\textwidth]{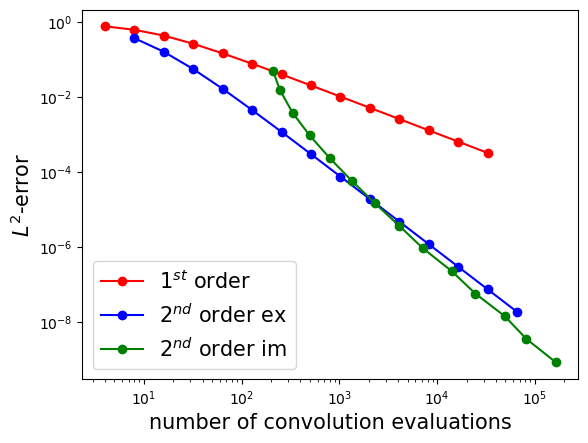}
        \caption{}
    \end{subfigure}
    \begin{subfigure}[h]{0.49\textwidth}
        \centering
        \includegraphics[width=\textwidth]{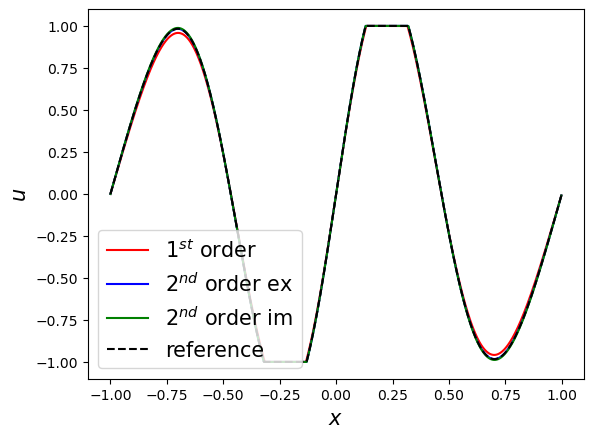}
        \caption{}
    \end{subfigure}
    \begin{subfigure}[h]{0.49\textwidth}
        \centering
        \includegraphics[width=\textwidth]{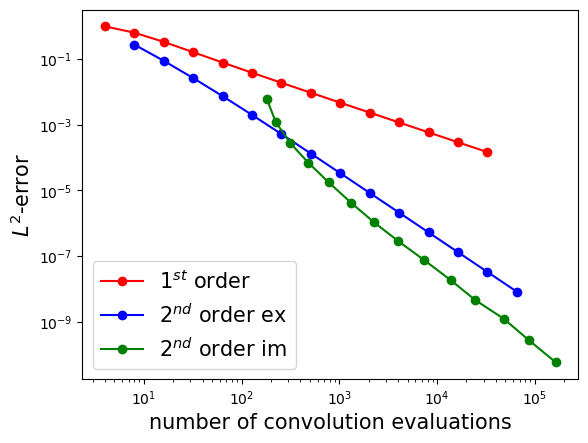}
        \caption{}
    \end{subfigure}
    \begin{subfigure}[h]{0.49\textwidth}
        \centering
        \includegraphics[width=\textwidth]{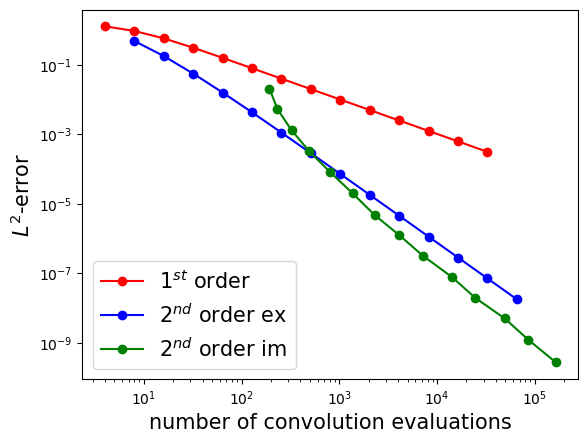}
        \caption{}
    \end{subfigure}
    \caption{Errors at the final time versus computational cost are plotted in $(a), (c),$ and $(d)$ for the 1D, 2D, and 3D cases with obstacle potential, respectively, and the solutions at the final time with relatively the same computational cost for all the three schemes are plotted in $(b)$ for the 1D case.}
    \label{comput_costs}
\end{figure}

\begin{figure}[ht!]
    \centering
    \begin{subfigure}[h]{0.49\textwidth}
         \centering
         \includegraphics[width=\textwidth]{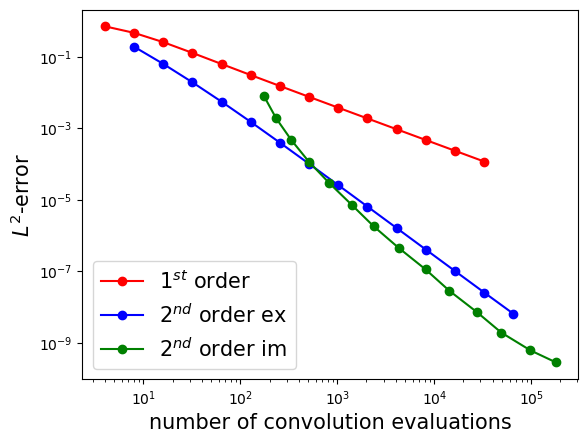}
         \caption{regular potential}
    \end{subfigure}
    \begin{subfigure}[h]{0.49\textwidth}
         \centering
         \includegraphics[width=\textwidth]{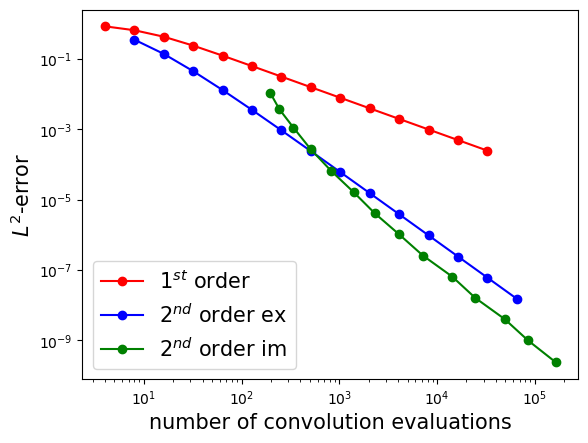}
         \caption{logarithmic potential}
    \end{subfigure}
    \caption{Errors at the final time versus computational cost for regular and logarithmic potentials with $\thetalog=0.5$.}
    \label{comput_costs2}
\end{figure}
\subsection{Dynamics prediction for 3D examples.}
We provide an illustration in a three-dimensional setting. 
Unless otherwise stated, we consider an obstacle potential case and set $\Omega=(-1,1)^3$, $\cpot=1.6, \varepsilon=\delta=0.1$, which corresponds to $\xi=2.4$, and a time step size $\Delta t = 0.005$.

\begin{example}\label{simu_3D} 
    We adopt the second-order explicit scheme to simulate the time evolution of a 3D bubble merging and a star shape up to $T=8$ with mesh size $N = 150^3.$ The initial conditions of the 3D bubble merging and the star shape are respectively given as
    \begin{align*}
        u_0(x,y,z) =&\ \frac12 \tanh\left(\frac{0.6 - \sqrt{(x-0.35)^2+y^2+z^2}}{\sqrt{2}\varepsilon}\right) 
        +\ \frac12\tanh\left(\frac{0.6 - \sqrt{(x+0.35)^2+y^2+z^2}}{\sqrt{2}\varepsilon}\right),\\
        u_0(x,y,z) =&\ \tanh\left(\frac{0.7 + 0.2\cos(6\eta) - \sqrt{x^2+y^2+z^2}}{\sqrt{2}\varepsilon}\right),
    \end{align*}
    where,
    \[
    \eta = 
    \begin{cases}
        \tan^{-1}\left(\frac{z}{x}\right), \hspace{25pt} \text{ if } x>0.5,\\
        \pi + \tan^{-1}\left(\frac{z}{x}\right), \hspace{5pt}\text{ otherwise }.
    \end{cases}
    \]
\end{example}

As observed in Figure~\ref{bubble_evolution_3D}, the two balls gradually shrink and merge into one smaller ball, which eventually disappears as expected. For the star shape, the tips of the star move inward, and the gaps between the tips move outward, forming a ball that finally disappears.
\begin{figure}[ht!]
    \centering
    \begin{subfigure}[h]{0.24\textwidth}
        \centering
        \includegraphics[width=\textwidth]{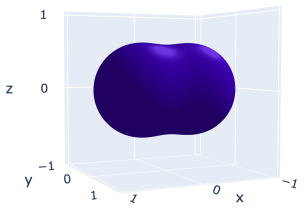}
    \end{subfigure}
    \begin{subfigure}[h]{0.24\textwidth}
        \centering
        \includegraphics[width=\textwidth]{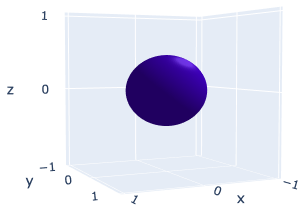}
    \end{subfigure}
    \begin{subfigure}[h]{0.24\textwidth}
        \centering
        \includegraphics[width=\textwidth]{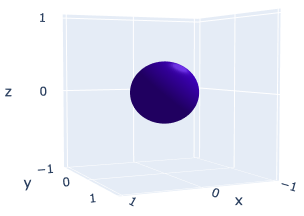}
    \end{subfigure}
    \begin{subfigure}[h]{0.24\textwidth}
        \centering
        \includegraphics[width=\textwidth]{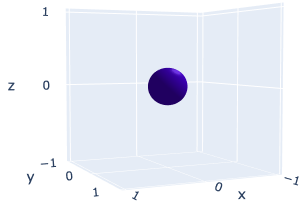}
    \end{subfigure}
    \centering
    \begin{subfigure}[h]{0.24\textwidth}
        \centering
        \includegraphics[width=\textwidth]{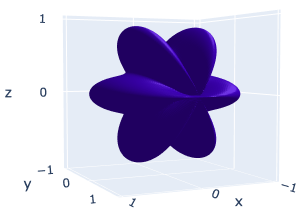}
    \end{subfigure}
    \begin{subfigure}[h]{0.24\textwidth}
        \centering
        \includegraphics[width=\textwidth]{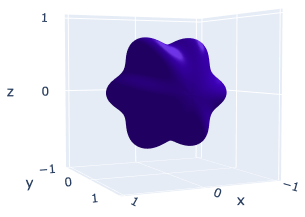}
    \end{subfigure}
    \begin{subfigure}[h]{0.24\textwidth}
        \centering
        \includegraphics[width=\textwidth]{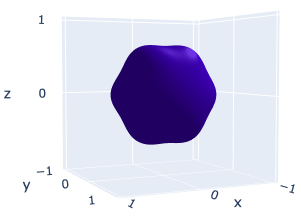}
    \end{subfigure}
    \begin{subfigure}[h]{0.24\textwidth}
        \centering
        \includegraphics[width=\textwidth]{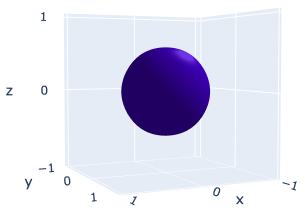}
    \end{subfigure}
    \caption{Time evolution of a 3D bubble merging (top) and a star shape (bottom) at different time instances in Example~\ref{simu_3D}. From left to right: $t=0, 2, 4$ and $8$.}
    \label{bubble_evolution_3D}
\end{figure}
Figure~\ref{random_states_2D3D} presents numerical results for the 2D and 3D simulations with a random initial condition sampled from a uniform distribution, ranging from $-0.95$ to $0.95$ (for 2D) and from $-0.5$ to $0.5$ (for 3D). For the 2D case, we choose $\Omega = (-\pi,\pi)^2$, $\cpot = 1$, $\varepsilon=0.05, \delta=0.08$, and $N = 512^2$ (corresponding to $\xi = 0.5625$). For the 3D case, we set $\Omega = (-1,1)^3$, $\cpot = 1.6$, $\varepsilon=0.1, \delta=0.1$, and $N = 60^3$ (corresponding to $\xi=2.4$).
\begin{figure}[ht!]
    \centering
    \begin{subfigure}[h]{0.24\textwidth}
        \centering        
        \includegraphics[width=\textwidth]{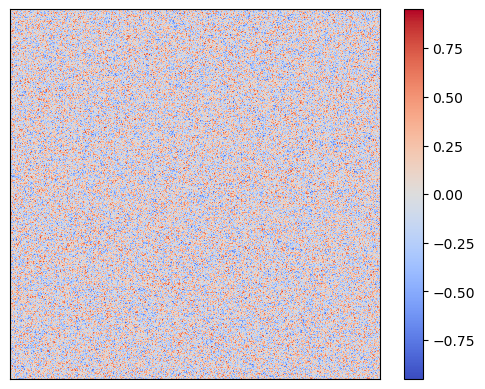}
        \caption{$t = 0$}
    \end{subfigure}
    \begin{subfigure}[h]{0.24\textwidth}
        \centering
        \includegraphics[width=\textwidth]{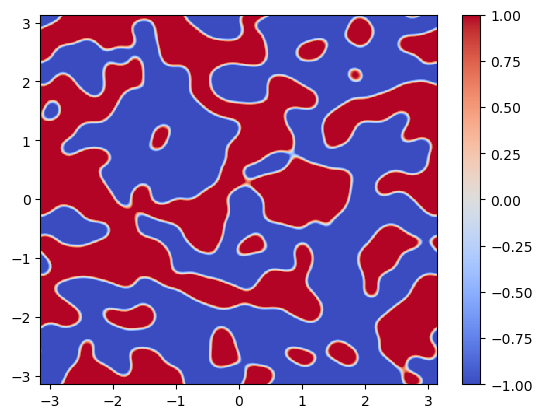}
        \caption{$t = 10$}
    \end{subfigure}
    \begin{subfigure}[h]{0.24\textwidth}
        \centering
        \includegraphics[width=\textwidth]{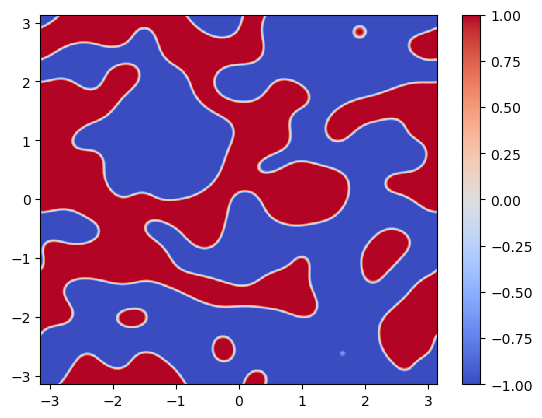}
        \caption{$t = 30$}
    \end{subfigure}
    \begin{subfigure}[h]{0.24\textwidth}
        \centering
        \includegraphics[width=\textwidth]{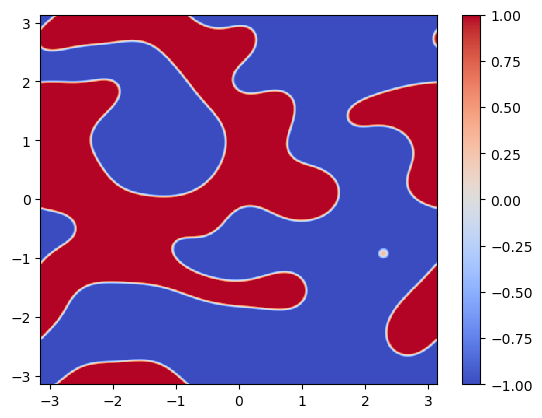}
        \caption{$t = 100$}
    \end{subfigure}
    \begin{subfigure}[h]{0.24\textwidth}
        \centering
        \includegraphics[width=\textwidth]{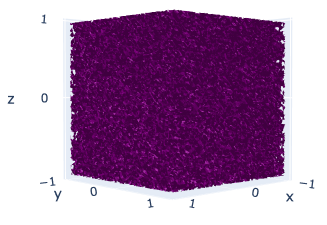}
        \caption{$t = 0$}
    \end{subfigure}
    \begin{subfigure}[h]{0.24\textwidth}
        \centering
        \includegraphics[width=\textwidth]{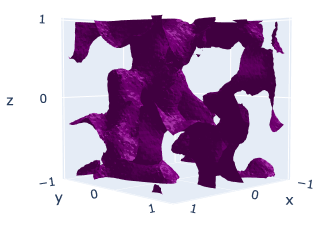}
        \caption{$t = 2$}
    \end{subfigure}
    \begin{subfigure}[h]{0.24\textwidth}
        \centering
        \includegraphics[width=\textwidth]{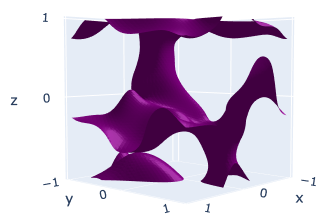}
        \caption{$t = 4$}
    \end{subfigure}
    \begin{subfigure}[h]{0.24\textwidth}
        \centering
        \includegraphics[width=\textwidth]{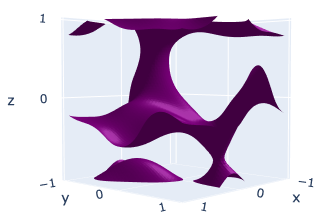}
        \caption{$t = 6$}
    \end{subfigure}
    \caption{Time evolution of the phase separation with a random uniform initial state in a two-dimensional (top) and three-dimensional (bottom) cases at different time instances.}
    \label{random_states_2D3D}
\end{figure}

\subsection{Non-isothermal Allen-Cahn model}\label{sec:NAC}
We demonstrate the applicability of our method to non-isothermal systems, commonly encountered in the solidification of pure materials. Anisotropy is typically inherent in such models and is often represented by imposing nonlinearities on the gradient terms. However, in the nonlocal framework, it can be incorporated through a suitable anisotropic kernel. While we focus on the isotropic model here, our approach can be adapted for nonlocal anisotropic cases, which will be the subject of future work.

More specifically, we consider the following coupled system of equations for the temperature $\t$ and phase-field variable $\p$:
\begin{eqnarray}
\begin{dcases}
\frac{\partial\t}{\partial t} = D \Delta\t+L\frac{\partial u}{\partial t},\\[1.5ex]
\mu\frac{\partial u}{\partial t}+\mathcal{L}u+\partial_u \Phi(u,\t)\ni 0,
\end{dcases}\label{eq:AC_nonisothermal}
\end{eqnarray}
where $D$, $\mu$ and $L$ are the diffusivity, relaxation time and latent heat coefficients, respectively. Here, $\partial_u\Phi(u,\t)$ is a generalized subdifferential of the double-well potential 
\begin{equation*}
    \Phi(u,\t) = F(u)+g(u,\t),
\end{equation*}
where $F$ is defined as in~\eqref{eq:potential_general} and $g(u,\t)$ is an appropriately chosen coupling term. Here, we demonstrate the approach for the obstacle potential, for which $F$ is defined in~\eqref{eq:potential_f0} and~\eqref{eq:potential_psi_obstacle} and the coupling term is defined as in~\cite{kobayashi1993}, see also~\cite{burkovska2023}:
\begin{equation*}
g(u,\t) = -\cpot m(\t) u,\quad\text{with}\quad m(\t)=\left(\frac{\alpha}{\pi}\right)\tan^{-1}\left(\rho(\t_e-\t)\right),\quad0<\alpha<1,\label{eq:coupling_term_m}
\end{equation*}
for which it holds that $|m(\t)|<1/2$. 

We adopt the first-order discretization as outlined in Section~\ref{subsec:first_order_disc} for the phase-field equation and an implicit Euler time-stepping scheme for the temperature equation. We also use an explicit discretization of the coupling term $m(\t)$ that allows to fully decouple the problem. Then, the semi-discrete system reads: Find $(\t^k,u^k)$, $k=1,\dots,K$ such that
\begin{equation*}
    \begin{cases}
     \left(I-\tt D\upDelta\right)\t^k=\t^{k-1}+L(u^k-u^{k-1}),\\
    u^k = {P_{[-1,1]}\left(\frac{1}{\mu/\tt + \xi}\left(\frac{\mu}{\tt}u^{k-1} + \gamma\ast u^{k-1}+c_Fm(\t^{k-1})\right)\right).}
    \end{cases}
\end{equation*}
First, having $\t^{k-1}$ using the projection formula, we compute $u^{k}$. Then, the temperature $\t^k$ is computed from the first equation above. Using the Fourier collocation approach outlined in Section~\ref{sec:spatial_disc} we solve for the temperature in the Fourier space:
\begin{equation}
    \hat{\t}^k_{lm} = {\frac{\hat{\t}^{k-1}_{lm}+L(\hat{u}^k_{lm}-\hat{u}^{k-1}_{lm})}{1-D\tt \lambda_{mul}}},
\end{equation}
where $\lambda_{mul}$ is the multiplier of the Laplacian operator. By applying the inverse Fourier transform, we obtain
{$$\t^k_{ij}=(\mathcal{F}^{-1}_N\hat{\t}^k)_{ij},$$
where $(\mathcal{F}^{-1}_N\hat{\t}^k)_{ij}$ is as defined in \eqref{f_inverse}.}
This is a simple and straightforward approach that allows for an efficient computation of the solution of a coupled non-smooth problem. In the following we provide examples for two- and three-dimensional test cases. 

\begin{example}\label{nnac_ex1}
    We set $\Omega = (-1,1)^2$, $N=512^2$, $\cpot = 1/4$, $\alpha=0.9$, $\mu = 0.0003$, $\theta_e=1$, $\rho=10$, $L=0.5$, $D=1$, and adopt the first-order semi-implicit scheme to simulate the phase separation up to $T=0.2$. We choose $\tt=0.0001$, $\delta = 0.1,$ and $\varepsilon = 0.0251$, which corresponds to $\xi=0.002$. The initial conditions are 
    \begin{eqnarray*}
        u_0(x,y)=\begin{cases}
            -1, & -0.9\le x \le 0.9, -0.9\le y \le 0.9\\
            1,& \text{ otherwise}
        \end{cases} \text{ and }
        \t_0 \equiv 0.
    \end{eqnarray*}
\end{example}

This example, inspired from an example in~\cite{kobayashi1993}, see also~\cite{burkovska2023}, is an illustration of the solidification of pure materials using nonlocal isothropic model. Here, we have a solid region around the boundaries of the domain and a pool of liquid in the interior. As the time evolves, the solidification that occurs from the walls propagates inward until all liquid solidifies. As illustrated in Figure~\ref{nnac1}, this corresponds to the liquid region (marked in blue) in the snapshot of the phase-field variable to gradually shrink over time and eventually disappears.
\begin{figure}[ht]
    \centering
    \begin{subfigure}[h]{0.24\textwidth}
        \centering        
        \includegraphics[width=\textwidth, height=2.4cm]{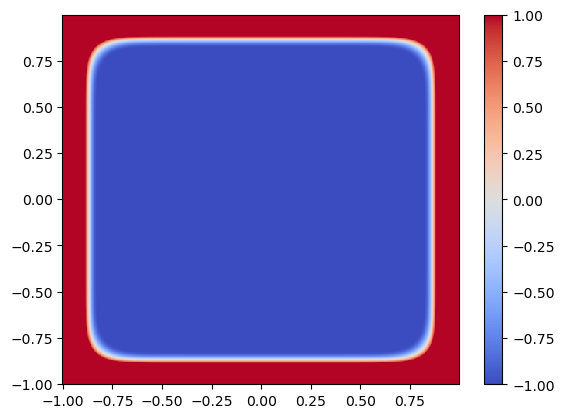}
    \end{subfigure}
    \centering
    \begin{subfigure}[h]{0.24\textwidth}
        \centering        
        \includegraphics[width=\textwidth, height=2.4cm]{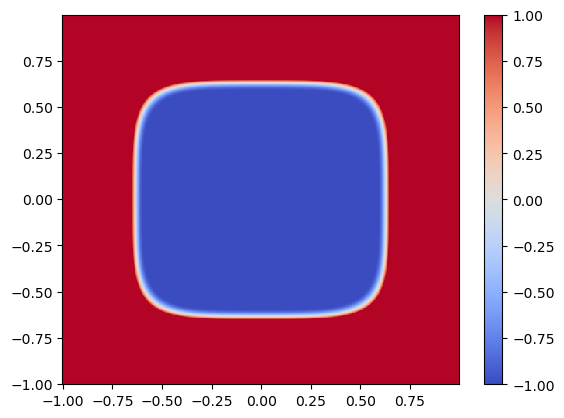}
    \end{subfigure}
    \centering
    \begin{subfigure}[h]{0.24\textwidth}
        \centering        
        \includegraphics[width=\textwidth, height=2.4cm]{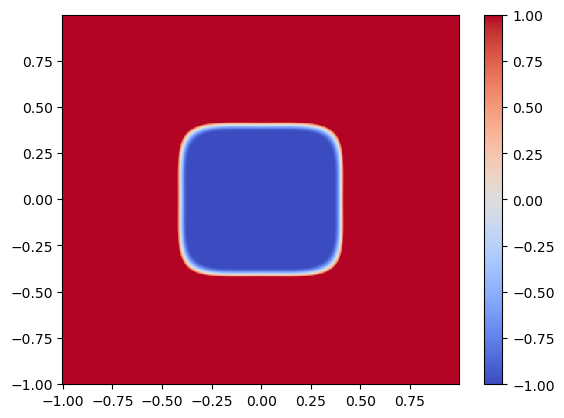}
    \end{subfigure}
    \centering
    \begin{subfigure}[h]{0.24\textwidth}
        \centering        
        \includegraphics[width=\textwidth, height=2.4cm]{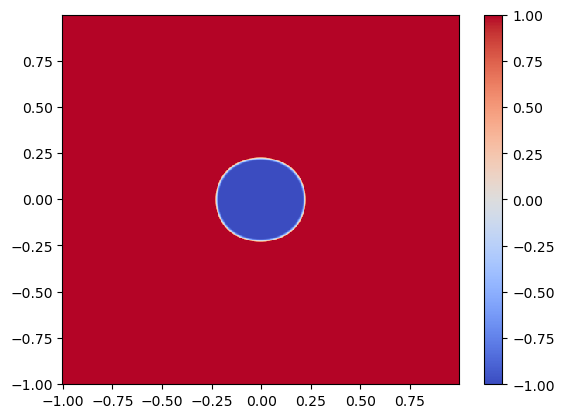}
    \end{subfigure}
    \centering
    \begin{subfigure}[h]{0.24\textwidth}
        \centering        
        \includegraphics[width=\textwidth, height=2.4cm]{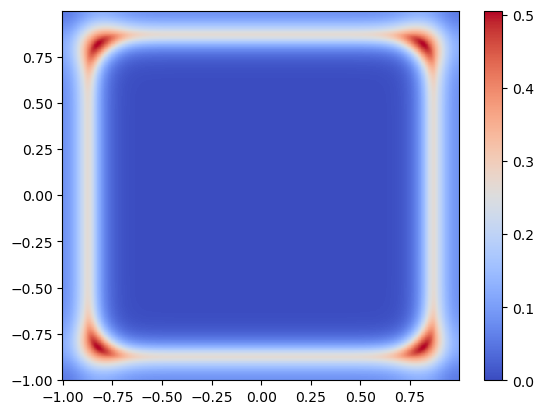}
    \end{subfigure}
    \begin{subfigure}[h]{0.24\textwidth}
        \centering
        \includegraphics[width=\textwidth, height=2.4cm]{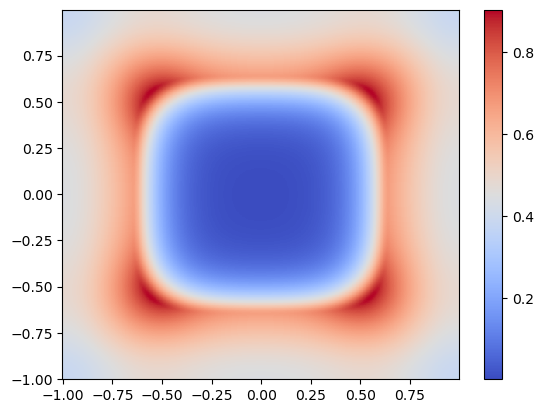}
    \end{subfigure}
    \centering
    \begin{subfigure}[h]{0.24\textwidth}
        \centering        
        \includegraphics[width=\textwidth, height=2.4cm]{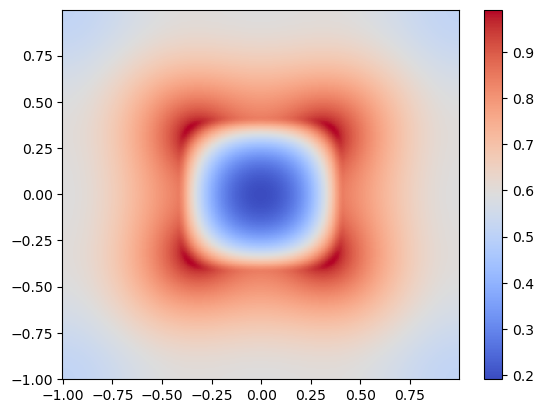}
    \end{subfigure}
    \centering
    \begin{subfigure}[h]{0.24\textwidth}
        \centering        
        \includegraphics[width=\textwidth, height=2.4cm]{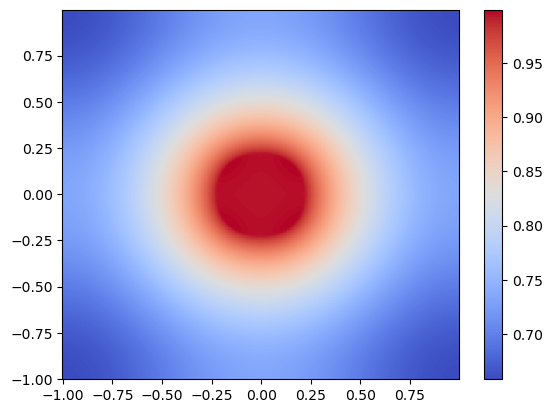}
    \end{subfigure}
   \caption{Phase field (top) and temperature (bottom) solutions at different time instances in the settings of Example~\ref{nnac_ex1}. From left to right: $t=0.006, 0.04, 0.08$, and $0.2$.}
    \label{nnac1}
\end{figure}

Lastly, we demonstrate the performance of the model in three-dimensional case using deterministic and random initial conditions for the phase-field variable and constant temperature.

\begin{example}\label{nnac3D}
    We set $\Omega = (-1,1)^3$, $N=100^3$, $\tt=0.0001$ and adopt the same parameter settings as in the previous example, along with similar initial conditions:
    \begin{eqnarray*}
        u_0({\bf{x}})=\begin{cases}
            -1, & -0.9\le {\bf{x}} \le 0.9,\ {\bf{x}}\in\bbR^3,\\
            1,& \text{ otherwise}
        \end{cases} \quad\text{ and }
        \t_0 \equiv 0.
    \end{eqnarray*}
\end{example}

The corresponding snapshots of the phase-field variable are presented on Figure~\ref{nnac_3D} (top row). We can observe a similar behavior in the solution as seen in the two-dimensional case. 
Next, we consider a random initial condition for the phase-field variable $u$, taken to be a Gaussian random field with mean zero and Gaussian covariance kernel. We consider the same settings as before, apart from  $\mu = 0.0005$, $T=0.003$ and chose $\delta = 0.1,$ and $\varepsilon = 0.08$, which corresponds to $\xi=2.31$. 
On Figure~\ref{nnac_3D} we present the snapshots of the time evolution for the phase-field variable.
\begin{figure}[ht]
    \centering
    \begin{subfigure}[h]{0.24\textwidth}
        \centering        
        \includegraphics[width=1\textwidth]{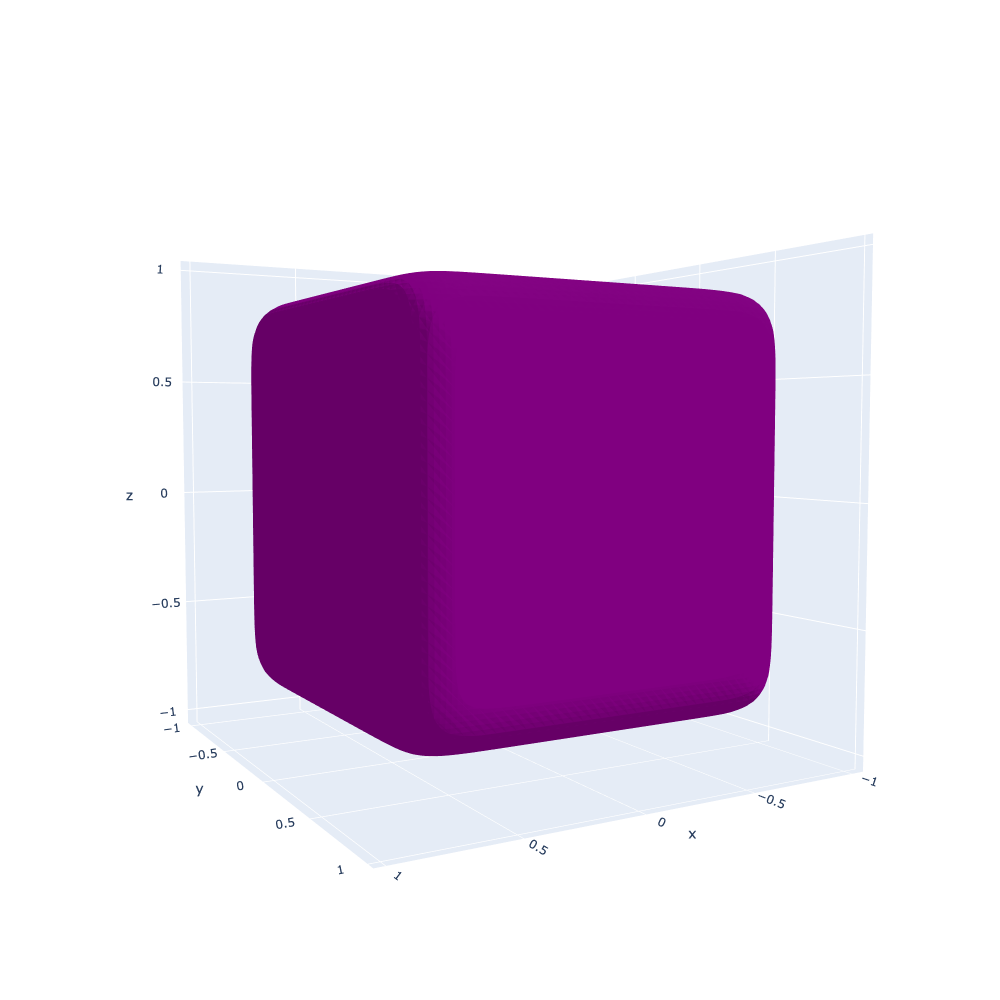}
    \end{subfigure}
    \centering
    \begin{subfigure}[h]{0.24\textwidth}
        \centering        
        \includegraphics[width=1\textwidth]{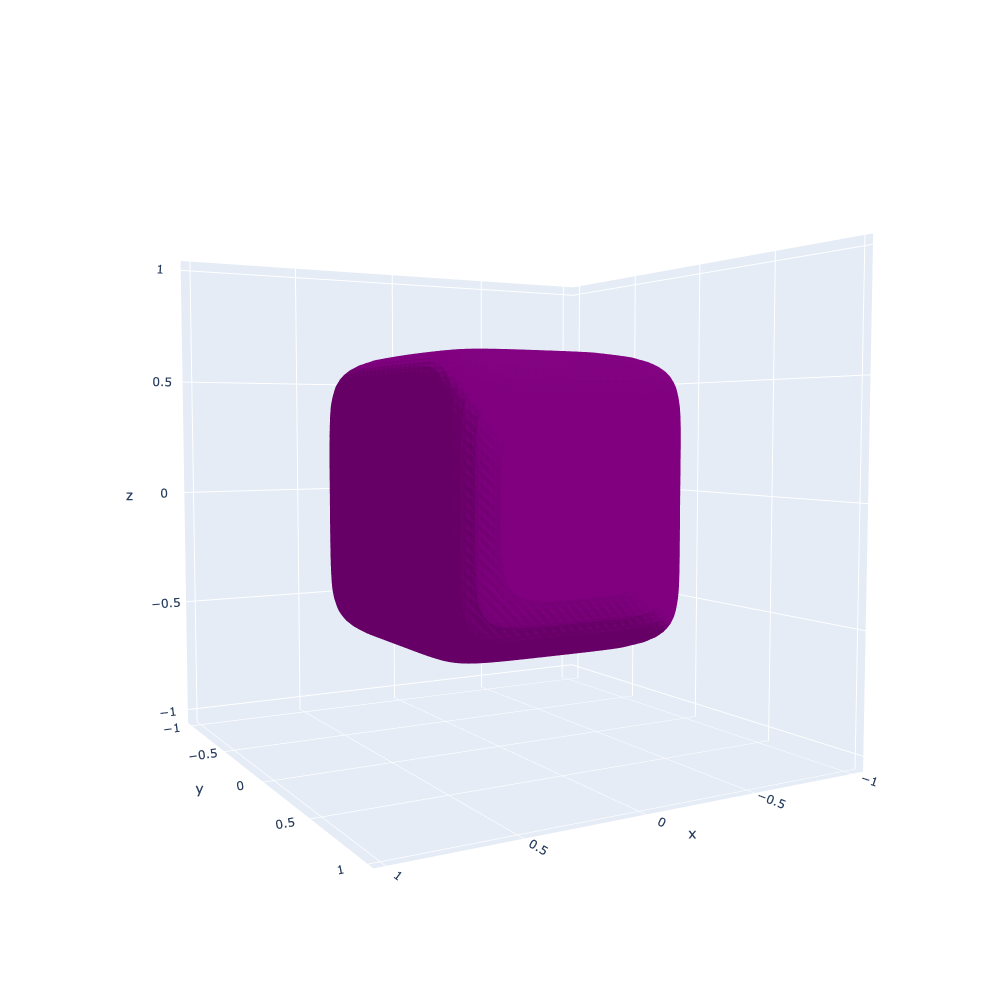}
    \end{subfigure}
    \centering
    \begin{subfigure}[h]{0.24\textwidth}
        \centering        
        \includegraphics[width=1\textwidth]{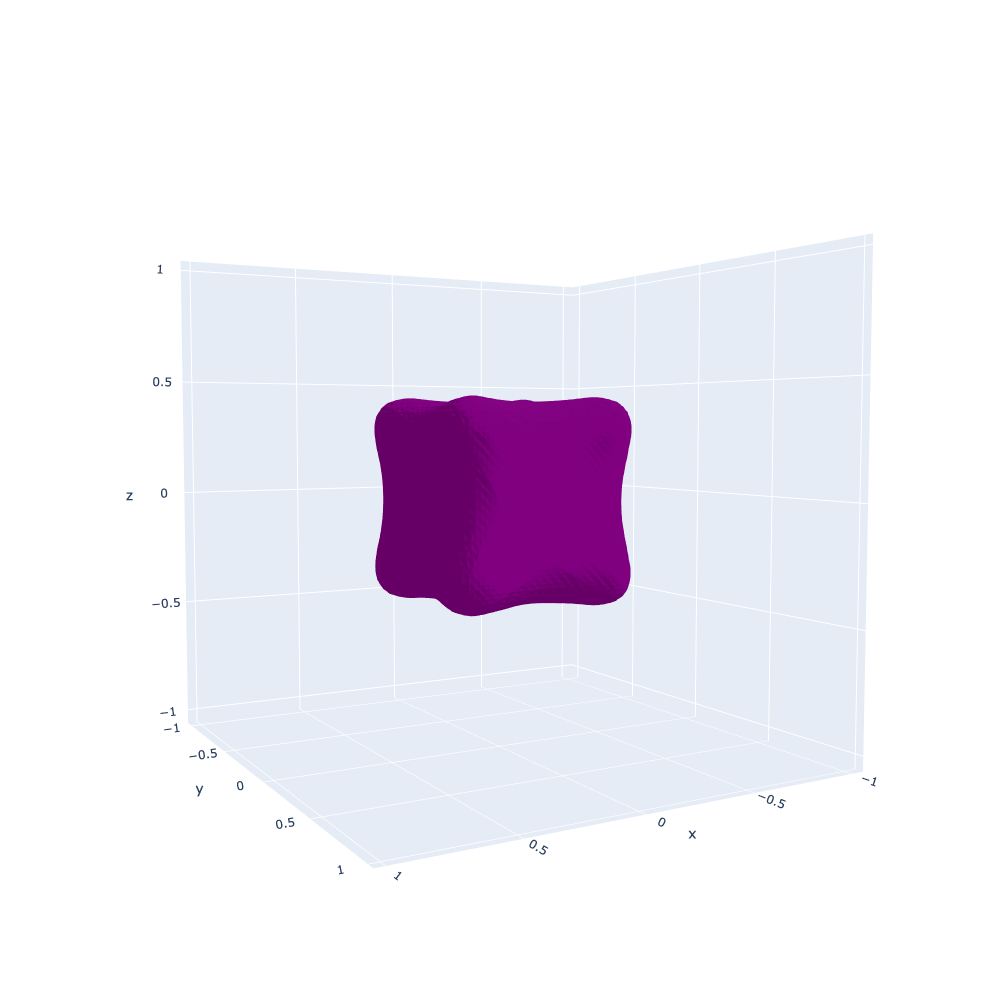}
    \end{subfigure}
    \centering
    \begin{subfigure}[h]{0.24\textwidth}
        \centering        
       \includegraphics[width=1\textwidth]{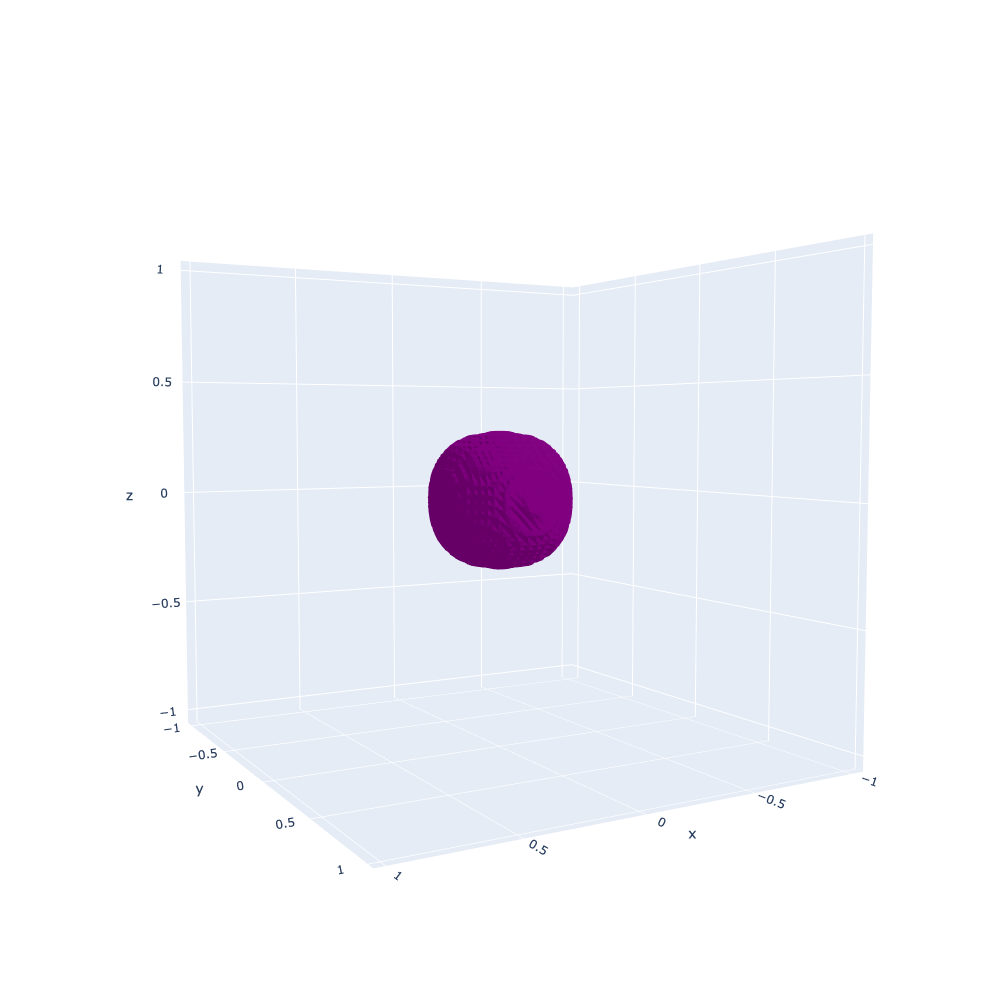}
    \end{subfigure}\\
    \centering
    \begin{subfigure}[h]{0.24\textwidth}
        \centering        
        \includegraphics[width=1\textwidth]{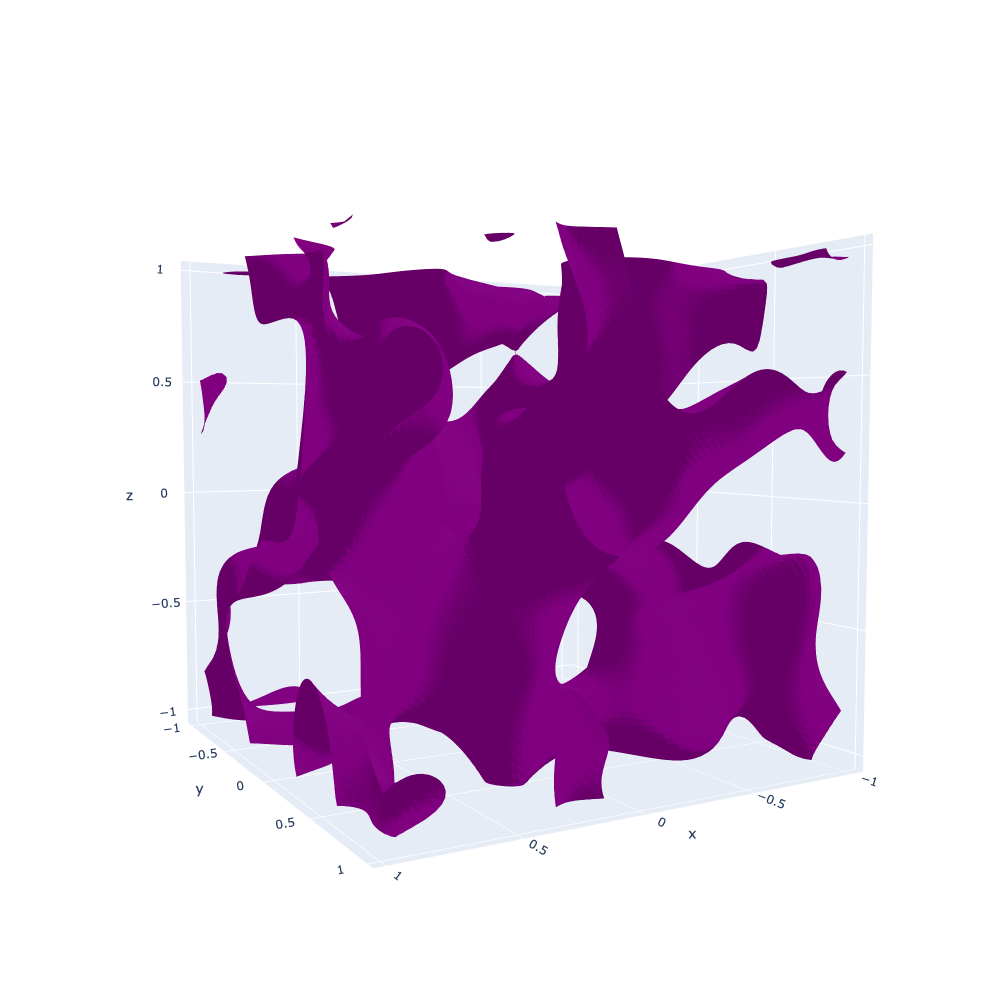}
    \end{subfigure}
    \centering
    \begin{subfigure}[h]{0.24\textwidth}
        \centering        
        \includegraphics[width=1\textwidth]{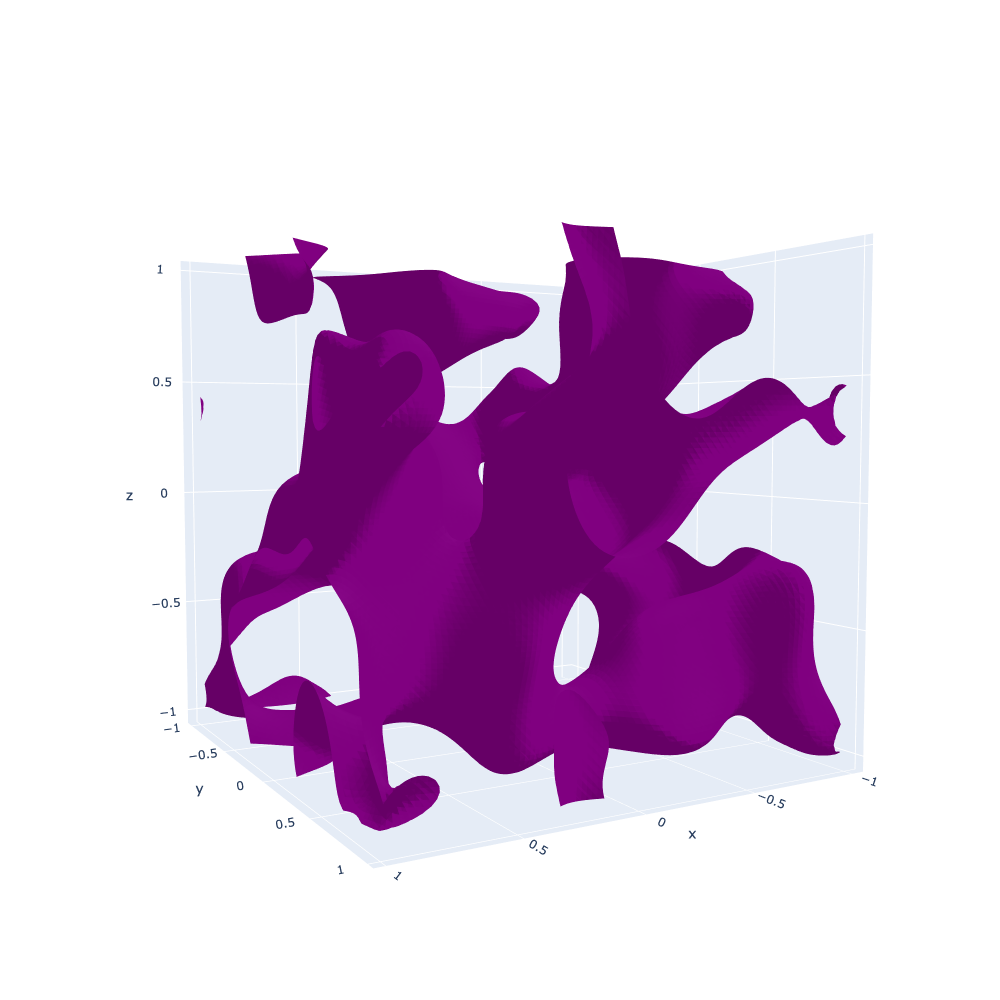}
    \end{subfigure}
    \centering
    \begin{subfigure}[h]{0.24\textwidth}
        \centering        
        \includegraphics[width=1\textwidth]{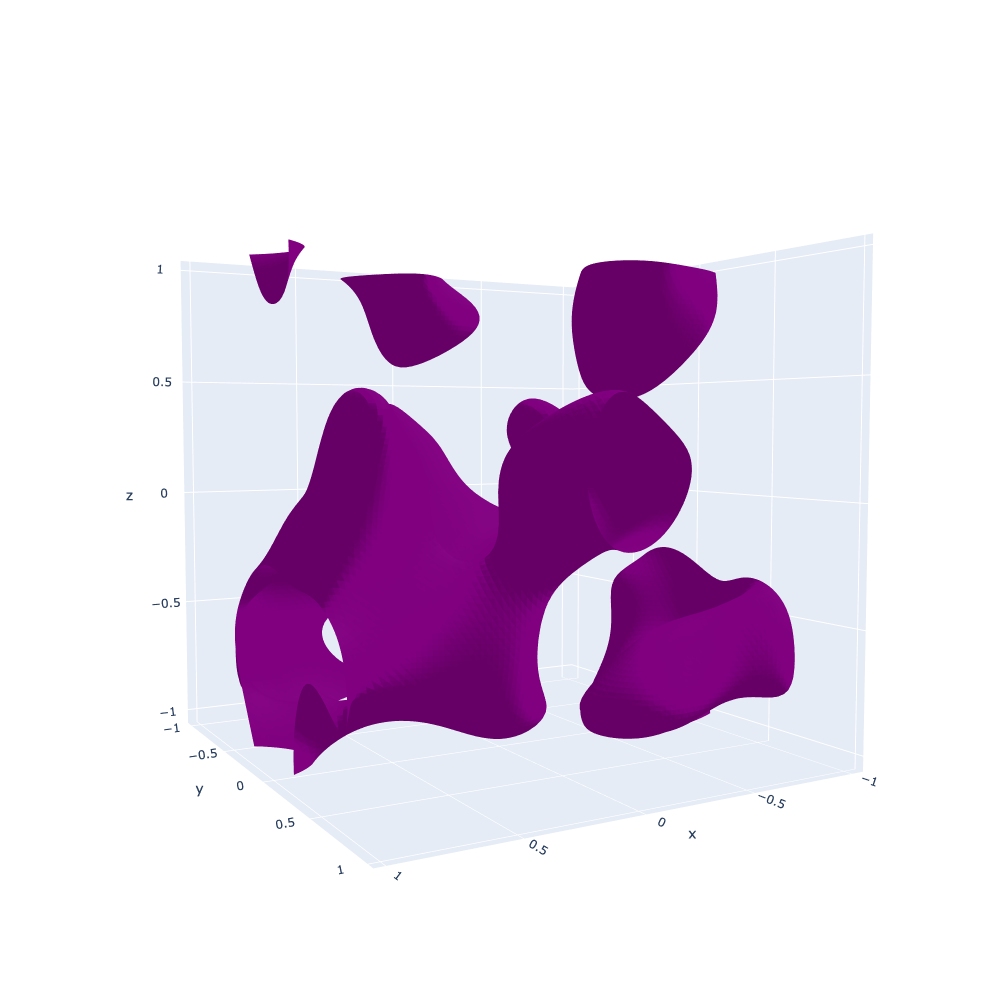}
    \end{subfigure}
    \centering
    \begin{subfigure}[h]{0.24\textwidth}
        \centering        
       \includegraphics[width=1\textwidth]{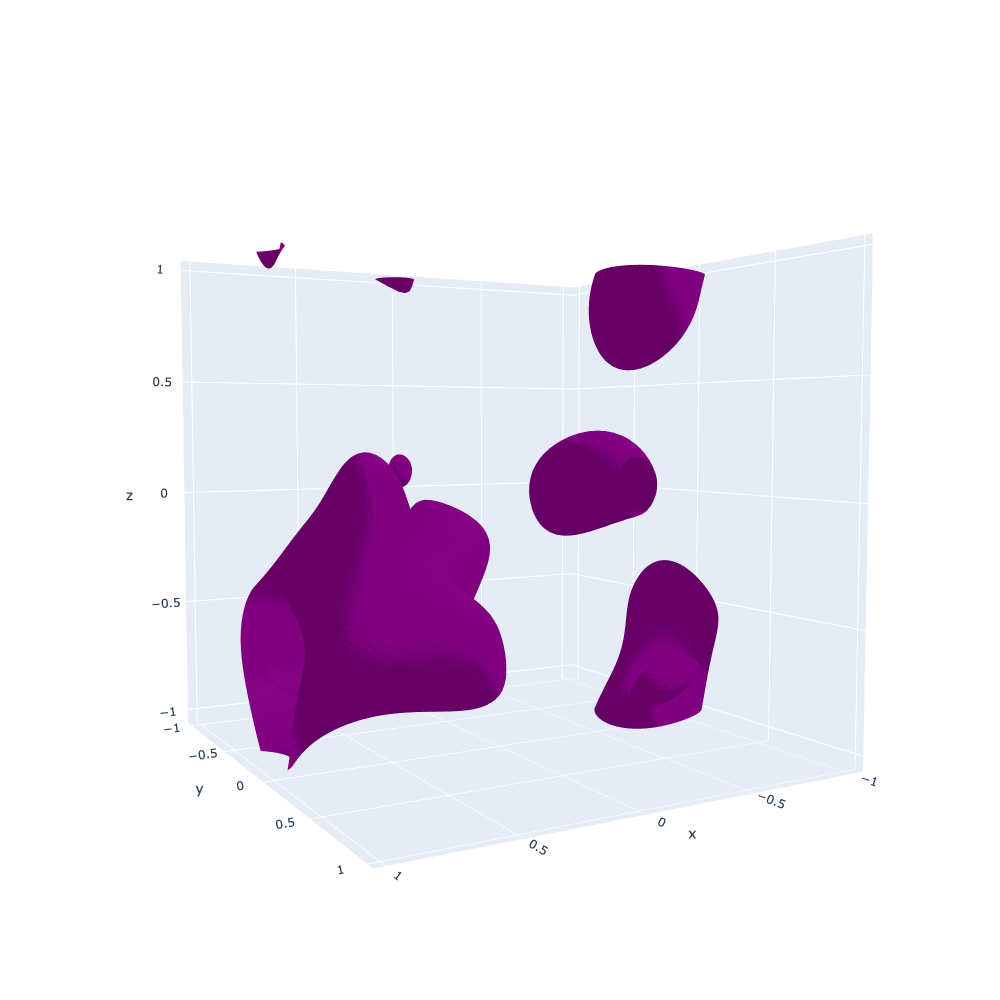}
    \end{subfigure}
   \caption{Evolution of the phase field solutions with the initial condition as in Example~\ref{nnac_ex1} (top) and random initial condition as in Example~\ref{nnac3D} (bottom). From left to right: $t=0.006, 0.04, 0.08$, and $0.2$ (top);  $t=0.0008, 0.001, 0.002$, and $0.003$ (bottom).}
    \label{nnac_3D}
\end{figure}

\section{Conclusion}\label{sec:conclusion}
In this work, we present a comprehensive approach to discretizing the nonlocal Allen-Cahn equation including regular, singular, and non-smooth potentials. We design first- and second-order energy-stable temporal discretizations, complemented by efficient spatial discretization using the Fourier collocation method based on FFT, resulting in a fully scalable and efficient scheme. The solution algorithm is efficiently realized by the proximal representation formula, which enables direct evaluation of the solution.
Future research could extend this approach by exploring its application to more complex systems, such as Cahn-Hilliard or non-isothermal models, as well as to more complex boundary conditions or geometries. Additionally, investigating parallelization techniques for scalability in high-performance computing environments would be valuable.

\section{Acknowledgment}\label{sec:acknowledgment}
This work is supported in part by the U.S. Department of Energy, Office of Science, Office of Advanced Scientific Computing Research, Applied Mathematics program and was performed at the Oak Ridge National Laboratory, which is
managed by UT-Battelle, LLC under Contract No. De-AC05-00OR22725. 
Ilyas Mustapha acknowledge the support from
{the National Science Foundation MSGI program} and Olena Burkovska reports the financial support from the U.S. Department of Energy, Office of Advanced Scientific Computing Research, Applied Mathematics Program under the award number ERKJ345.

\end{document}